\documentclass{amsart}
\usepackage{amssymb, enumerate, mathrsfs}

\usepackage{bbm}
\usepackage{verbatim}
\usepackage{tikz}
\usepackage{amsmath}
\usepackage{xcolor}
\usepackage{dsfont}
\usepackage{tikz}
\usepackage[colorlinks=true, pdfstartview=FitV, linkcolor=blue, 
            citecolor=red, urlcolor=red]{hyperref}
            
\usepackage{cleveref}
\newtheorem{thm}{Theorem}[section]
\newtheorem{cor}[thm]{Corollary}
\newtheorem*{corollary*}{Corollary}
\newtheorem*{theorem*}{Theorem}
\newtheorem{lem}[thm]{Lemma}
\newtheorem{prop}[thm]{Proposition}

\theoremstyle{definition}

\theoremstyle{definition}
\newtheorem{defi}[thm]{Definition}
\theoremstyle{remark}
\newtheorem{rem}[thm]{Remark}

\numberwithin{equation}{section}

\newcommand\be{\begin{equation}}
\newcommand\ee{\end{equation}}
\newcommand\ba{\begin{align}}
\newcommand\ea{\end{align}}
\newcommand\ben{\begin{enumerate}}
\newcommand\een{\end{enumerate}}

\def\Q{\ensuremath {{\mathbb{Q}}}}
\def\C{\ensuremath {{\mathbb{ C}}}}
\def\Z{\ensuremath {{\mathbb{Z}}}}

\def\f{\ensuremath {{\mathfrak f}}}
\def\p{\ensuremath {{\mathfrak p}}}
\def\l{\ensuremath {{\mathfrak l}}}
\def\c{\ensuremath {{\mathfrak c}}}
\def\aa{\ensuremath {{\mathfrak a}}}
\def\a{\ensuremath {{\mathfrak A}}}

\def\L{\ensuremath {{\mathcal L}}}
\def\O{\ensuremath {{\mathcal O}}}
\def\b{\ensuremath {{\mathfrak b}}}

\title
[$p$-adic properties of Eisenstein-Kronecker cocycles ]
{$p$-adic properties of Eisenstein-Kronecker cocycles over imaginary quadratic fields and $p$-adic interpolation }

\author{Jorge Fl\'orez}
\email{jflorez@bmcc.cuny.edu}
\address{Department of Mathematics, Borough of Manhattan Community College, City University of New York, NY 11216, USA}



\date{\today}

\begin{document}

\begin{abstract}
We establish integrality and congruence properties for the Eisenstein-Kronecker cocycle  in \cite{BCG2}. As a consequence, we recover the integrality of the critical values of 
Hecke $L$-functions over imaginary quadratic fields in the split case. Additionally, we construct a $p$-adic measure that interpolates these critical values.
\end{abstract}

\maketitle

\tableofcontents


\section{Introduction}

In recent years, there has been significant interest  in the use Eisenstein-Kronecker cohomology classes  for the study of  Hecke $L$-functions. These objects, which generalize classical notions such as the Eisenstein-Kronecker numbers,
have proven to be powerful tools for deducing  arithmetic properties of critical values of 
$L$-functions. For example, Bergeron, Charollois and García  \cite{BCG2} recover the algebraicity of these critical values for $L$-functions associated to arbitrary  extensions of an imaginary quadratic field $K$. Furthermore, Kings and Sprang \cite{KS} obtain integrality results of these values for arbitrary totally complex fields and, moreover, construct a  $p$-adic measure that interpolates them, leading to a generalization of Harder's result on special values of $L$-functions and the $p$-adic $L$-function of Colmez and Schneps \cite{CS}.

These results are generalizations of a classical result due to Damerell concerning the algebraicity of the critical values of Hecke $L$-functions of $K$ by relating them to Eisenstein numbers, and of  Manin, Vi\v{s}ik \cite{VM}, and Katz  \cite{K}  \cite{katz1976p},  concerning the construction of a $p$-adic $L$-function that interpolates these critical values; later generalized to $p$-adic $L$-functions for CM fields in \cite{katz1978p}. In \cite{dS}, de Shalit describes another way of constructing these (one-variable) $p$-adic $L$-functions from a norm-coherent sequence of elliptic units which is an approach  originally due to   Coates-Wiles \cite{coates1978p} and further developed by  Yager \cite{yager1982two}\cite{yager1984p}, Cassou-Nogues \cite{CN}, Gillard \cite{gillard1980unites} and Tilouine \cite{tilouine1986fonctions}.
One of the ingredients used by \cite{dS} is to show certain integrality  and congruence properties for the Eisenstein numbers.

The Eisenstein-Kronecker cocycle was originally inspired by the construction of Sczech who used this cohomological interpretation to provide another proof of the Siegel-Klingen rationality theorem on the algebraicity of special values of zeta functions associated to totally real fields of any degree.  Charollois and Dasgupta  \cite{CD} defined a smoothed version of Sczech's cocycle  and  obtained from it a new proof of the integrality of these special values and  construct a p-adic measure interpolating them.   They further applied this result to prove a conjecture of Gross on the order of vanishing of the zeta function at $s=0$.

The Eisenstein-Kronecker cocycle was initially inspired by Sczech's construction, who used its cohomological interpretation to provide another proof of the Siegel-Klingen rationality theorem on the algebraicity of special values of zeta functions associated to totally real fields of any degree. Charollois and Dasgupta \cite{CD} defined a smoothed version of Sczech's cocycle, using it to obtain a new proof of the integrality of these special values and construct a p-adic measure interpolating them. They further applied this result to prove a conjecture of Gross on the order of vanishing of the zeta function at s=0.

In this article, we extend the integrality and congruence results for Eisenstein numbers from \cite[II §3]{dS} to Eisenstein-Kronecker numbers. We then use these results to deduce corresponding properties for the Eisenstein-Kronecker cocycle in \cite{BCG2}. As a byproduct, we recover integrality results for critical values of Hecke $L$-functions and construct a $p$-adic measure interpolating these values, in the case where $p$ splits in $K$. While our results on Eisenstein-Kronecker numbers are somehow implicitly contained in the work of, for example, Bannai and Kobayashi \cite{BK}, their formulation is not directly suitable for our purposes. Their approach involves showing that a certain Theta function associated to the Poincaré bundle of an elliptic curve is the generating function for Eisenstein-Kronecker numbers and has a $p$-adic expansion with integral coefficients. In contrast, our approach utilizes the known integrality of the $p$-adic expansion of Eisenstein series from \cite{CS} and extends certain identities between Eisenstein numbers (cf. \cite[Prop. 9]{CS} and \cite[II Lemma 3.1]{dS}) to obtain explicitly  the corresponding $p$-adic properties for Eisesntein-Kronecker numbers.

 Prior to the work of Bergeron-Charollois-García  and Kings-Sprang, cohomology classes associated  to  Eisesntein series were also studied in the context of 
  of Sharifi-Venkatesh \cite{SV} from the point of view of motivic cohomology,
 of Graf \cite{graf2016polylogarithms},  
 Beilinson-Kings-Levin \cite{beilinson2018topological},  and 
 Bannai-Hagihara-Yamada-Yamamoto \cite{bannai2023canonical} for totally real fields, and of
 Obaisi \cite{Oba} and \cite{FKW} for extensions of imaginary quadratic fields. The latter two directly build upon the ideas of \cite{Scz} and \cite{Col}, which also inspire our work. Furthermore,  
 we closely follow  the construction of the smoothed cocycle of Charollois and Dasgupta in \cite{CD},
leading to slightly different notation for the Eisenstein-Kronecker cocycle compared to \cite{BCG2}. Our approach can be viewed as a cohomological interpretation of the work of  Colmez-Schneps \cite{CS} which  allows us to establish a connection with the results in  \cite{BCG2}.

We conclude this introduction by stating our main results.



\subsection{Main results}
\label{Main results}

We fix once  and for all an imaginary quadratic field $K$ and denote by $\mathcal{O}=\mathcal{O}_K$ its ring of integers. For a prime ideal $\p$ of $K$ we denote by $\mathcal{O}_{(\p)}$ the localization at $\p$. Let $S$ denote a field extension of $K$, usually $\C$,  an algebraic closure $\overline{\Q}$ of $\Q$, or the maximal abelian extension $K^{\mathrm{ab}}$ of $K$.

Let $\mathbf{G}_n:=H^2  \times K^{2n}\times \mathcal{W}_n$, where $H$ is the $K$-vector space 
of homogenous polynomials in $n$-variables with coefficients in $K$
and $\mathcal{W}_n$ is certain space of $\mathcal{O}$-lattices of $K^{2n}$ (cf.  \Cref{Notation}). Let  $\mathcal{D}_S(\mathbf{G}_n)$ denote    the $K$-vector space of $S$-valued distributions on $\mathbf{G}_n$
equipped with an action of $\mathrm{GL}_n(K)$ (cf. \Cref{The space of distributions}).

The Eisenstein-Kronecker numbers $E_{l}^k(z,w,L)$, $k\geq 0$, $l>0$, 
 give rise to a distribution  $E^n\in \mathcal{D}_\C(\mathbf{G}_n)$  which is, essentially, a $K$-linear combination of products of  Eisenstein-Kronecker numbers (cf. \Cref{Eisenstein-Kronecker Distributions}).
The Dedekind distribution is then the map $D:M_n(K)\to \mathcal{D}_\C(\mathbf{G}_n)$
given by $D(\sigma)=\sigma E^n$ if $\det (\sigma)\neq 0$ and $D(\sigma)=0$ otherwise.

In \Cref{$p$-adic properties of $E_{k,l}$}, we prove some integrality and congruence properties of $E_{l}^k(z,w,L)$
that complement some already know results for the Eisesntein numbers $E_{l}^k(z,0,L)$ (cf. e.g. \cite{dS}) and deduce from them the corresponding properties for $E^n$ (cf. \Cref{p-adic-En}) and $D$ (cf. \Cref{arithmetic-cong-of-Dk}). These are the analogous results for the generalized Bernoulli distributions stated  in propositions 2.14 and 2.15 of \cite{CD}.

For  $\phi\in \mathcal{D}_S(\mathbf{G}_n)$, we introduce a smoothing operation  with respect
to a prime $\l$ of $K$  that yields a new distribution $\phi_\l$ 
(cf. \Cref{smoothing-op}). This smoothed distribution belongs to the space of distributions $\mathcal{D}_S^\l=\mathcal{D}_S(\mathbf{G}_n^\l)$ that accounts for the $\l$-smoothing  (\Cref{smoothed-distrib}). We let $D_{\l}:M_n(\mathcal{O}_{(\l)})\to \mathcal{D}_S^\l$ denote the smoothed Dedekind distribution, i.e., $D_{\l}(\sigma)=D(\sigma)_\l$. 

Consider the congruence subgroup $\Gamma_\l\subset \mathrm{GL}_n(\mathcal{O}_{(\l)})$
consisting of matrices whose first column has elements but the first divisible by $\l$. The space $\mathcal{D}_S^\l$ comes equipped with an action of $\Gamma_\l$ (cf. \Cref{action-gamma-l}).
The Eisenstein-Kronecker cocycle (cf. \Cref{The smoothed Eisenstein-Kronecker cocycle}) is the map
$\Phi_\l: \Z[\Gamma_\l^n]\to \mathcal{D}_\C^\l$ defined by 
\[
\Phi_\l(\mathfrak{A})=D_{\l}(\sigma(1))
\]
for $\mathfrak{A}=(A_1,\dots, A_n)\in \Gamma_\l^n$ 
and extended by linearity, where  $\sigma(1)$ is the matrix whose $i$th column is the first column of the matrix $A_i$.  This map is homogenous with respect to the action of $\Gamma_\l$, i.e. $\Phi_\l\in C^{n-1}(\Gamma_\l,\mathcal{D}_\C^\l)=\mathrm{Hom}_{\Gamma_\l}(\Z[\Gamma_\l^n],\mathcal{D}_{\C}^\l)$,
and satisfies the cocycle property 
\[
\sum_{i=0}^n(-1)^i\Phi_\l(A_0,\dots, \widehat{A}_i,\dots, A_n)=0.
\]

 We denote by $[\Phi_\l]$ its cohomology  class
in $H^{n-1}(\Gamma_\l,\mathcal{D}_\C^\l)$. We remark that our definition of $\Phi_\l$ differs slightly from \cite{BCG2} since our construction is inspired by that of \cite{CD}. However, the link between the two definitions can be established using the functional equation for Eisenstein-Kronecker numbers (cf. \eqref{symmetry-E-k-l}). 

In \Cref{Complex periods}, we introduce some complex periods $\Omega_{\infty}$, $\eta_{\infty}$, $\Omega_{\infty}^*$ associated to an elliptic curve $E$ defined over $K(\f_1)$, the ray class field of $K$ modulo the integral ideal $\f_1$ of $K$, with complex multiplication by $\mathcal{O}$ so that when rescaling $\Phi_\l$ by these periods yields a cocycle $ \tilde{\Phi}_{\l}$ 
 which lives now in $H^{n-1}(\Gamma_\l,\mathcal{D}_{\overline{\Q}}^\l)$ (cf. \Cref{algeb-smoothed-psi-l-1}).

We now recall the connection of $\Phi_\l$ with $L$-functions. Let $F/K$ be a finite extension, $\f$ be an integral ideal of $F$, and $I(\f)$ be the group of fractional ideals  coprime to $\f$. Let $\chi:I(\f)\to \C^\times$ be an ideal character such that
$
\chi((a))=\varphi(a)\lambda(a),
$
for any principal ideal $(a)$ in $I(\f)$, where 
 $\varphi:(\mathcal O_F/\f)^\times\to \C^\times$ is a residue class character
 and 
  $\lambda=\lambda_{F,k,l}:F^\times\to K^\times$ is a character defined by
\begin{equation*}\label{character-lambda}
\lambda(a)=\overline{N_{F/K}(a)}^kN_{F/K}(a)^{-l}.
\end{equation*}
Here $k\geq 0$  and $l>0$ are both integers 
such that  
$\lambda(\epsilon)=1 $
for all units $\epsilon$ in the group
$
U_{\f}=U_{F,\f}:=\{\epsilon\in \mathcal{O}_F^{\times}\,:\,\epsilon\equiv 1 \pmod{\f}\}.
$
We extend the Hecke character to all fractional ideals by letting $\chi(\mathfrak{b})=0$ if $\mathfrak{b}$ is not prime to $\f$.
For another nonzero integral ideal $\mathfrak{g}$ we consider the Hecke $L$-function
\[
L_{\mathfrak{g}}(\chi,s) = \sum_{\aa\in I(\mathfrak{g})} \chi(\aa) N_{F/\Q}(\aa)^{-s}
\]
and if  $\aa\in I(\f)$ the
 partial $L$-function  
\begin{align*}\label{partial L}
L_{\mathfrak{g}}(\aa,\chi,s)
:=\sum_{\b \sim_{\mathfrak{g}}\mathfrak{a}} \frac{\chi(\b)}{N_{F/\Q}(\b)^s},
\end{align*}
where the  sum is running over all integral ideals $\b$ of $F$ in the same ray class of $\mathfrak{a}$ $\mathrm{mod}\, \mathfrak{g}$.
Let $\mathfrak{c}$ be an integral ideal of $F$ 
with prime norm to $K$ being coprime to $\f$. The smoothed partial $L$ function is defined by 
$L_{\mathfrak{g},\c}(\mathfrak{a},\chi,s)
= L_{\mathfrak{g}}(\mathfrak{a}\c,\chi,s)-\chi(\c)N_{F/\Q}(\c)^{1-s} L_{\mathfrak{g}}(\mathfrak{a},\chi,s)$ and thus we have
\[
\sum_{a\in G_{\mathfrak{g}}}L_{\mathfrak{g},\c}(\aa,\chi,s)=(1-\chi(\c)N_{F/\Q}(\c)^{1-s})L_{\mathfrak{g}}(\chi,s),
\]
where $G_\mathfrak{g}= I(\mathfrak{g})/P_\mathfrak{g}$ is the ray class group of $F$ modulo $\mathfrak{g}$, and $P_\mathfrak{g}= \{(\alpha) : \alpha\in F^{\times},\ \alpha\equiv 1\text{ mod}^*\,\mathfrak{g}\}$.

Let $\l$ be the conjugate in $K$ of the prime ideal $N_{K/\Q}(\c)$ and denote by $\mathcal{D}_S^{\l \,\vee}$  the $K$-vector space
$K$-linear maps $f:\mathcal{D}_S^\l\to S$ equipped with a right action of $A\in \Gamma_\l$ given by $(Af)(\phi):=f(A^{-1}\phi)$. For an element $Y\in \mathbf{G}_n^\l$, let $f_Y$ be the element of $\mathcal{D}^{\l\,\vee}_S$ defined by evaluation at $Y$: $f_Y(\phi)=\phi(Y)$. For an $(n-1)$-chain $\mathfrak{E}\in \Z[\Gamma_\l^n]$, we denote by $[\mathfrak{E}\otimes f]$ the class of 
$\mathfrak{E}\otimes f\in C_{n-1}(\Gamma_\l,\mathcal{D}_S^{\l\,\vee})=\Z[\Gamma_\l^n]\otimes_{\Gamma_\l}\mathcal{D}^{\l\,\vee}_S$ in $H_{n-1}(\Gamma_{\l},\mathcal{D}^{\l\,\vee}_S)$.
The cap product 
\[
C^{n-1}(\Gamma_\l,\mathcal{D}_S^\l)\times C_{n-1}(\Gamma_\l,\mathcal{D}_S^{\l\,\vee})\to S
\]
given by $\langle \Phi,\mathfrak{A}\otimes f\rangle=f(\Phi(\mathfrak{A}))$
and extended by linearity,
yields a canonical pairing  
\[
\langle\cdot,\cdot\rangle:
H^{n-1}(\Gamma_{\l},\mathcal{D}_S^\l)\times H_{n-1}(\Gamma_{\l},\mathcal{D}_S^{\l\,\vee})\to S.
\]
Note that both
$H^{n-1}(\Gamma_{\l},\mathcal{D}_S)$ and $H_{n-1}(\Gamma_{\l},\mathcal{D}_S^{\l\,\vee})$ are $K$-vector spaces and
this pairing is $K$-linear.

According to \cite[Theorem 1.2]{BCG2} ( cf. \Cref{special-value--cocycle-param-EK-co}  and \Cref{integ-smoothed}), given  $\mathfrak{a}$, $\f$, $\chi$ and  $\c$, there exists  an $(n-1)$-chain
 $\mathfrak{E}^*\in \Z[\Gamma_\l^n]$ and an $f^*\in\mathcal{D}_{\overline{\Q}}^{\l\,\vee}$
such that the homology class of $\mathfrak{Z}_{\mathfrak{a},\f,\chi,\c}=\mathfrak{E}^*\otimes f^*\in \Z[\Gamma_\l^n]\otimes_{\Gamma_\l}\mathcal{D}_{\overline{\Q}}^{\l\,\vee}$,  depends only on $\mathfrak{a}$, $\f$, $\chi$ and  $\c$ and not on any other choices,
and for which we have  
\begin{equation}\label{critical-val-period}
 [U_\f:V_\f]\,
 \Gamma(l)^n\,
\frac{(2\pi i)^{nk}}
{\Omega_{\infty}^{nk}\Omega_{\infty}^{*nl}}\,
    L_{\f,\c}(\mathfrak{a},\chi,0)
  \,  =\,
 \big\langle\, \big[\tilde{\Phi}_{\l}\big],\,\big[\mathfrak{Z}_{\mathfrak{a},\f,\chi,\c}\big]\,\big\rangle,
\end{equation}
 where  $V_\f$ is the free part of the subgroup of units of $U_{\f}$ of $F$ that have norm to $K$ equal to 1. Thus,  the algebraicity of the critical values follows. 

Let $p>3$ be a rational prime that splits in $K$ as $\p\overline{\p}$
 and it is coprime to $\l$. Furthermore, suppose $\p$ is prime to $\f_1$.
Let $C_{\p}^{n-1}(\Gamma_\l,\mathcal{D}_{\overline{\Q}}^\l)$
be the $\mathcal{O}_{(\p)}$-submodule of 
$C^{n-1}(\Gamma_\l, \mathcal{D}_{ \overline{\Q} }^\l)$
determined by those $\Phi$ such that $\Phi(\mathfrak{A}\otimes f_Y)$ is $\p$-integral
for every $\mathfrak{A}\otimes f_Y$ satisfying \eqref{p-adic-condition-tuple}. 
Thus, the $\p$-integrality of $\tilde{\Phi}_\l$ can de expressed as follows.

 \begin{theorem*}[\Cref{congruence-psi}]
 $\tilde{\Phi}_{\l}\in C_{\p}^{n-1}(\Gamma_\l,\mathcal{D}_{\overline{\Q}}^\l)$.
 \end{theorem*}

Let $H_{\p}$ be the $\mathcal{O}_{(\p)}$-submodule of 
$H^{n-1}(\Gamma_\l,\mathcal{D}_{   \overline{\Q} }^\l)$ generated by the classes of elements in 
$C_{\p}^{n-1}(\Gamma_\l,\mathcal{D}_{ \overline{\Q}   }^\l)$. Let $H_{\p}^{\vee}$
be the $\mathcal{O}_{(\p)}$-submodule of 
$H_{n-1}( \Gamma_{\l}  , \mathcal{D}_{ \overline{\Q} } ^ {\l\,\vee}   )$
of all homology classes that have a  $\p$-integral value when paired with every cohomology class of  $H_{\p}$. Suppose $\f_1$ is coprime to $p$, divides $\f$,
and satisfies $n!<N_{K/\Q}(\f_1)$.  If  $\p$ is coprime to
 $\f\mathfrak{a}\mathfrak{d}_F$
  then
 $[\mathfrak{Z}_{\mathfrak{a},\f,\chi,\c}]\in H_{\p}^{\vee}$ (cf. \Cref{Integrality of the special values}) and thus

\begin{corollary*}[\Cref{integ-smoothed}]
The critical value
 \[
[U_\f:V_\f]
\cdot 
 \Gamma(l)^n
\cdot
\frac{(2\pi i)^{nk}}
{\Omega_{\infty}^{nk}\Omega_{\infty}^{*nl}}
\cdot 
    L_{\f,\c}(\mathfrak{a},\chi,0)
\]
is $\p$-integral.
\end{corollary*}

Finally, we discuss the interpolation of critical values via the construction of a $p$-adic measure. Write $\f=\f_0\f_\p\f_{\overline{\p}}$, where $f_0$ is the prime to $p$-part of $\f$ and $\f_\p$, respectively $\f_{\overline{\p}}$, is divisible only by primes above $\p$, resp. $\overline{\p}$.  Let $\mathfrak{e}_\p$, respectively  $\mathfrak{e}_{\overline{\p}}$, denote the product of the primes of $F$ dividing $\p$, resp. $\overline{\p},$ that do not divide $\f$.
For an ideal $\mathfrak{a}\in I(\f_0)$ we define
\begin{equation*}
L^*_{\f_0}(\aa,\chi,s):=
\sum_{\mathfrak{g}_\p|\mathfrak{e}_\p,\mathfrak{g}_{\overline{\p}}|\mathfrak{e}_{\overline{\p}} }
\mu(\mathfrak{g}_\p\mathfrak{g}_{\overline{\p}})
\frac{\chi(\mathfrak{g}_\p^{-1}\mathfrak{g}_{\overline{\p}})}
{   N(\mathfrak{g}_\p)^{1-s}   N(\mathfrak{g}_{\overline{\p}})^s    }
L_{\f_0}(\aa \mathfrak{g}_\p\mathfrak{g}_{\overline{\p}}^{-1},\chi,s),
\end{equation*}
where the sum is running over nonzero integral ideals $\mathfrak{g}_{\p}$ and $\mathfrak{g}_{\overline{\p}}$ dividing $\mathfrak{e}_{\p} $
and $\mathfrak{e}_{\overline{\p}}$, respectively,
and
 $\mu(\mathfrak{b})=\pm 1$ is determined  by the parity of the number of prime factors of $\mathfrak{b}$. 
Note that
\begin{align*}
\sum_{a\in G_{\f_0}}L_{\f_0}^*(\aa,\chi,s)=
\prod_{\mathfrak{P}|\p}\left(1-\frac{\chi(\mathfrak{P}^{-1})}{N(\mathfrak{P})^{1-s}}\right)
\prod_{\mathfrak{P}|\overline{\p}}\left(1-\frac{\chi(\mathfrak{P})}{N(\mathfrak{P})^s}\right)
L_{\f_0}(\chi,s),
\end{align*}
where $\mathfrak{P}$ denotes a prime ideal of $F$.  Moreover, for $\c$ as above, we define
\begin{equation*}\label{smoothed-p-partial-L-functions}
L^*_{\f_0,\c}(\aa,\chi,s):=L^*_{\f_0}(\aa \c,\chi,s)-\chi(\c)N_{F/\Q}(\c)^{1-s}L^*_{\f_0}(\aa,\chi,s).
\end{equation*}

Interpreting $\tilde{\Phi}_\l$ in terms of $p$-adic measures,  following the construction in \cite{CD}, we define a measured-valued cocycle $\mu_\l$ whose cohomology class $[\mu_\l]$ can be used, in conjunction with \eqref{critical-val-period}, to obtain the following result on the interpolation of the critical values $L^*_{\f_0,\c}(\aa,\chi,0)$.

\begin{theorem*}[\Cref{interp-cirt-val-padic-meas}]
  There exists a unique $\p$-integral valued measure
 $\mu_{\mathfrak{a},\f_0,\c,\p}$ on $\mathcal{O}_{F}\otimes \Z_p$  
 and a character $\tilde{\chi}:(\mathcal{O}_{F} \otimes \Z_p)^{\times} \to\overline{\Q}^{\times}$  associated to $\chi$  such that
 \begin{align*}\label{Main-ident-p-adic-meas-L-special}
 \frac{  1  }
   {   \Omega_{\p}^{nk} \Omega_{\p}^{*n(l-1)}    }
\int_{  (\mathcal{O}_{F} \otimes \Z_p)^{\times} } 
&\tilde{\chi}
\, 
d\mu_{\mathfrak{a},\f_0,\c,\p}=\\
&
[U_{\f_0}:V_\f]
\cdot 
\tau(\chi_{\f_\p}^{-1})
\cdot 
 \Gamma(l)^n
\cdot
\frac{(2\pi i)^{nk}}
{\Omega_{\infty}^{nk}\Omega_{\infty}^{*nl}}
L_{\f_0,\c}^*(\aa,\chi,0).\notag
  \end{align*}
  \end{theorem*}
  
  This result strengthens the work of Colmez and Schneps in \cite{CS}, who proved interpolation for specific cases: $n=1$, $2$ with  $k\geq 0$ and $l>0$, and   $n\geq 3$ with $k=0$ and $l=1$. Our approach, which can be viewed as a cohomological interpretation of the results of Colmez-Schneps, allows us to obtain a more general result. The key to this improvement lies in the parametrization of critical values via the Eisenstein-Kronecker cocycle in \cite{BCG2}.

It is worth noting that Kings and Sprang \cite{KS} have obtained remarkably general interpolation results using different methods.

\subsection{Outline of  the paper}

This paper is organized as follows.
 In 
 Section \ref{sec:eisco}, we precisely define the space of distributions and introduce the Eisenstein-Kronecker and Dedekind distributions. We also recall the construction of the (nonsmoothed) Eisenstein-Kronecker cocycle $\Psi$
  from \cite{Oba} and \cite{FKW}, which is inspired by the work in \cite{Scz}, \cite{Col} and \cite{CS}. Additionally, we describe its connection to the partial $L$-function.
   In 
  Section \ref{Eisenstein-Kronecker numbers and properties}, we prove integrality and congruence properties for Eisenstein-Kronecker numbers and deduce corresponding properties for Dedekind distributions.
   In
\Cref{The smoothed cocycle and its integrality properties},
 we introduce the $\l$-smoothing operation for distributions and recall the definition of the ($\l$-smoothed) Eisenstein-Kronecker cocycle. We prove its integrality properties and describe the connection between $\Psi_\l$ and $\Phi_\l$, from  which we deduce the integrality of the critical values.
Finally, in \Cref{p-adic measures and p-adic zeta functions}, we
construct the $p$-adic measure  $\mu_{\mathfrak{a},\f_0,\c,\p}$
from $\tilde{\Phi}_\l$ and prove the interpolation of  critical values.


\section{The Eisenstein-Kronecker cocycle}
\label{sec:eisco}

We begin this section by fixing some notation and formally defining the space of distributions. Next, we introduce the various distributions that will be used in this article, namely, the Eisenstein-Kronecker and Dedekind distributions (cf. \Cref{Eisenstein-Kronecker Distributions}), and the Colmez distribution (cf.  \Cref{The Eisenstein cocycle}). In  \Cref{The Eisenstein cocycle}, we recall the cocycle $\Psi(s)$ constructed in \cite{Oba} and \cite{FKW},which serves as a foundation for applying the methods of Charollois-Dasgupta \cite{CD}. Finally,  in \Cref{sec:heckeL}, we describe the parametrization of the partial $L$-function in terms of the cocycle $\Psi(s)$.

\subsection{Notation}\label{Notation}

 Let $\alpha \mapsto \overline{\alpha}$
denote the non-trivial automorphism of $K$.
Let $x=(x_1,\dots, x_n)$ denote a row vector in $K^n$ and $y=(y_1,\dots, y_n)^t$ a column vector in $K^n$. We 
let $\langle x,y\rangle$ denote the usual dot product between $x$ and $y$, i.e., 
\be\label{dot-prod}
\langle x,y\rangle= \sum x_iy_i.
\ee
We define the pairing
$e(\cdot|\cdot) \,:\, K^n\times K^n \to \C^{\times}$
given 
by 
\begin{equation*}\label{exp-pairing}
e(x|x^*)=e^{2\pi i  \mathrm{Tr}_{K/\Q}\left(\sum_{i=1}^nx_i\overline{x_i^*}\right) }
\end{equation*}
for vectors $x=(x_1,\dots, x_n)$ and $x^*=(x^*_1,\dots, x^*_n)$  in $K^n$.

Let $\Xi\subset K^n$ be a finitely generated $\mathcal{O}$-module   of rank $n$, i.e. $\dim_K \Xi \otimes_{\mathcal{O}}K=n $. Let $W_n$ denote the group of all such $\mathcal{O}$-modules.
We denote
by $\Xi^*\in W_n$ the dual of $\Xi$ with respect to the above pairing, i.e.,
$
\Xi^*=\{v\in K^n\,:\, e(x|v)=1, \ \forall x\in \Xi \}.
$
We will reserve the letter $\Lambda$ exclusively for 
$\mathcal{O}$-modules that are of the form 
\begin{equation*}\label{defin-lambda}
\Lambda:=\Lambda_1\times\dots\times \Lambda_n
\end{equation*}
 for some fractional ideals $\Lambda_i$ of $K$. Then
$
\Lambda^*=\Lambda_1^*\times \cdots \times \Lambda_n^*$,
where $\Lambda_i^*$ is the fractional ideal $K$ such that $\mathrm{Tr}_{K/\Q}(\Lambda_i \overline{\Lambda_i^*})\subset \mathcal{O}$, i.e.,
\[
\Lambda^*_i=(\overline{\Lambda_i\mathfrak{D}_{K}})^{-1}=\frac{1}{2i\mathrm{Vol(\Lambda_i)}}\Lambda_i,
\]
where $\mathfrak{D}_{K}$ is the absolute different ideal
of $K$ and $\mathrm{Vol(\Lambda_i)}$
is the volume of a fundamental parallelogram in $\C$ determined by $\Lambda_i$. This volume is defined in the following way: if $w_1$ and $w_2$ form a basis for $\Lambda_i$,
with $\Im(w_2/w_1)>0$, then $\mathrm{Vol}(\Lambda_i)=\Im(\overline{w_1}w_2)$. Moreover, 
$
\mathrm{Vol}(\Lambda_i)=2^{-1}N_{K/\Q}(\Lambda_1)\sqrt{|\mathfrak{d}_K|},
$
 where $\mathfrak{d}_K$ denotes the absolute discriminant of $K$.

If $\Xi$ and $\Upsilon$ are finitely generated $\mathcal{O}$-modules of rank $n$ contained in $K^n$, then we let
$
[\Xi:\Upsilon]_{\mathcal{O}}
$
be the module index of $\Xi$ and $\Upsilon$ (cf. e.g. \cite{ballew1970module} and \cite[Chap. I]{cassels1967algebraic}). This satisfies
$
[\Upsilon^*:\Xi^*]_{\mathcal{O}}=\overline{[\Xi:\Upsilon]}_{\mathcal{O}}.$
Suppose furthermore that $\Upsilon\subset \Xi$, then let $[\Xi:\Upsilon]$
denote the cardinality of  the finite  $\mathcal{O}$-module $\Xi/\Upsilon$. 
We have the character relations for every $u\in \Xi$
\begin{equation}\label{character-identity-2}
\frac{ 1 }{ [\Xi:\Upsilon] }
\sum_{y\in\Upsilon^*/\Xi^*} e(u| y)
=
\begin{cases}
1, &  \text{if $u\in \Upsilon$,}\\
0, & \text{if $u  \notin \Upsilon$.}
\end{cases}
\end{equation}
For arbitrary $\Xi$ and $\Upsilon$, we define $[\Xi:\Upsilon]$
as the quotient of $[\Xi:\Xi\cap \Upsilon]$ by $[\Upsilon:\Xi\cap \Upsilon]$.
Then $
[\Xi:\Upsilon]
=N_{K/\Q} \,[\Xi:\Upsilon]_{\mathcal{O}}
$
and
$
[\Xi:\Upsilon]=[\Upsilon^*:\Xi^*].
$

For $A\in \mathrm{GL}_n(K)$, let $A^h$ denote the Hermitian matrix of $A$
, i.e., $A^h=\overline{A}^t$ and by $A^*$  the inverse matrix $A^{-h}$, so that
$
e(xA|x^*)=e( x|x^*A^h )$
and
$
e(xA\,|\,x^*A^*)=e(x|x^*)
$
 for all $x,\,x^*\in K^n$.
Moreover, $(\Xi A)^*=\Xi ^* A^*$, as well as $[\Xi:\Xi A]_{\mathcal{O}}=\det (A) \mathcal{O}$
and $[\Xi:\Xi A]=\det (A)\cdot \overline{\det(A)}$.

Let $R_n$ the set of matrices $M$ in $\mathrm{GL}_n(\C)$ 
such that there exists a degree $n$ extension $F/K$ such that its columns $M_1,\cdots, M_n$,
 are conjugate over $K$, that is,
\begin{equation}\label{defin-matrix-M}
M=\begin{pmatrix}
\rho_{1}(m_1)&\cdots &\rho_{n}(m_1)\\
\vdots & \ddots &\vdots\\
\rho_{1}(m_n)&\cdots &\rho_{n}(m_n)\\
\end{pmatrix}
\end{equation}
where $m_1,\dots, m_n\in F$ and the $\rho_i$ are  distinct (up to conjugation) embeddings of $F$ into $\C$ fixing $K$. Furthermore,  if $A\in \mathrm{GL}_n(K)$, then $M^{-t}$, $\overline{M}$,  $M^*$ and $AM$ are in $R_n$ as well.

Let  $H$ the space of homogeneous polynomials $q(X)\in K[X_1,\dots, X_n]$ with the right action of $A\in \mathrm{GL}_n(K)$ given by $qA(X)=q(X\overline{A}^{-1})$.
We define  the polynomials  $N(X)=N(X_1,\dots, X_n)=X_1\cdots X_n$,  $N_M(X):=N(XM)$ and $X^{\mathbf{r}}:=X_1^{r_1}\cdots X_n^{r_n}$,
where  $M\in R_n$ and $\mathbf{r}=(r_1,\dots, r_n)\in \Z_{n\geq 0}^{n}$, which are all in $H$. For a fix $k\geq 0$, let $H_k$ be the $K$-subspace of $H$ generated by all the $N_M(X)^k$ for all $M\in R_n$.

Let $\mathcal{W}_n:=\{\varXi=\Xi\times\Xi^*\in W_n\times W_n\,:\,\Xi\in W_n\}$.
We can define a right action of $A\in \mathrm{GL}_n(K)$ on $Q=(q^*,q)\in H^2$, $U=(u,u^*)\in K^{2n}$ and $\varXi\in \mathcal{W}_n$, respectively, by
\[
QA=(qA,q^*A^*),\quad
UA=(uA,u^*A^*),\quad
\varXi A=\Xi A\times \Xi^*A^*.
\]
This action defines scalar multiplication by $t\in K^{\times}$ as
$Qt=(\overline{t}^{-\deg q}q, t^{\deg q^*}q^*)$,
$Ut=(u t,u^*\overline{t}^{-1})$, and
$\varXi t=\Xi t\times \Xi^*\overline{t}^{-1}$. Also, note that $Q\in H^2$ defines the map $K^{2n}\to K\,: X=(x,x^*)\mapsto Q(X)=q(x)\cdot q^*(x^*)$.

\subsection{The space of distributions}
\label{The space of distributions}
Consider the product
$
\mathbf{G}_n:=H^2\times K^{2n}\times \mathcal{W}_n$
 with the right action of
$A\in \mathrm{GL}_n(K)$ 
given by
\begin{equation}\label{Action-Gln-tuple}
(Q,U,\varXi)\cdot A=(QA,UA,\varXi A).
\end{equation}
and scalar multiplication by $t\in K^{\times}$ by
\begin{equation}\label{scalar-mult-tuple}
(Q,U,\varXi)\cdot t= (Qt,Ut,\varXi t).
\end{equation}

For $(Q,U,\varXi)\in \mathbf{G}_n$ we define the dual element 
$(Q,U,\varXi)^*:=(Q^*,U^*,\varXi^*)\in  \mathbf{G}_n$,
where $Q^*=(q^*,q)$, $U^*=( -u^*, u )$ and $\varXi^*= \Xi^* \times \Xi$.

For $k\geq 0$, 
we let
$\mathbf{G}_{n,k}=(H_k\times H)\times K^{2n}\times \mathcal{W}_{n}$.
Let $\mathbf{H}_n=J\times K^{2n}\times \mathcal{W}_{n}$, where $J$ is any $K$-subspace of $H^2$.
 Recall that $S$ denotes an extension field of $K$.

\begin{defi}\label{defin-space-dist}
We define $K$-vector  space of  distributions 
 $\mathcal{D}_S(\mathbf{H}_n)$ to be the space of all maps
$\phi :\mathbf{H}_{n} \to S$ that are $K$-linear in the $Q$-component and 
such that for every $(Q,U,\varXi)\in \mathbf{H}_n$ the following holds:
\begin{enumerate}[(i)]
\item $\phi(Q,U+V,\varXi)=e(u|v^*)\phi(Q,U,\varXi)$ for  all $V=(v,v^*)\in \varXi$ ,
\item\label{homon-t} $\phi(Q,U,\varXi)=t^n\phi\big(\,(Q,U,\varXi)\cdot t\,\big)$  for all $t\in K^{\times}$, and

\item\label{prop-dist-main-1} 
for any $\varUpsilon=\Upsilon\times \Upsilon^*\in \mathcal{W}_n$
\begin{equation}\label{prop-dist-main-3}
\phi(Q,U,\varXi)=
\frac{1}{[\Upsilon:\Xi\cap \Upsilon]}
\sum_{V\in \varXi/ \varXi\cap \varUpsilon}
e(-u|v^*)
\phi(Q,U+V,\varUpsilon).
\end{equation}

\end{enumerate} 

If $A\in \mathrm{GL}_n(K)$, we can define a distribution 
$A\phi\in \mathcal{D}_S(\mathbf{H}_n A)$ by
\begin{equation}\label{action-on-SE}
(A\cdot
\phi)(Q,U,\varXi)=
\det(A)\,\phi\big( \,(Q,U,\varXi)\cdot A \,\big).
\end{equation}
 If  $\mathbf{H}_n$ is invariant under $\mathrm{GL}_n(K)$,
then \eqref{action-on-SE} defines a left action of $\mathrm{GL}_n(K)$ on 
$\mathcal{D}_S(\mathbf{H}_n)$.
According to \Cref{defin-space-dist} \eqref{homon-t}, this action is well-defined on 
$\mathrm{PGL}_n(K)$.

\end{defi}

\begin{rem}\label{notation-dis-Gn}
Let $X=U+\varXi$ and $\phi\in \mathcal{D}_S(\mathbf{G}_n)$. We
 define the map $\phi_X:\mathbf{G}_n\to S$ 
\begin{equation*}\label{phi-x-nota}
\phi_X(Q,V,\varUpsilon)=
   \frac{e(-u|v^*-u^*)}
         {[\Upsilon:\Xi\cap \Upsilon]}
         \phi(Q,V,\varUpsilon).
\end{equation*}
Note that if $V\in X$, the value $\phi_X(Q,V,\varUpsilon)$ is independent of the class
of $V$ in $K^{2n}/\varUpsilon$. To further simplify the notation, when $Q=(1,1)$ we remove the letter $Q$ and simply write  $\phi(U,\varXi)$, $\phi_X(V,\varUpsilon)$,  etc.

\end{rem}

With this notation, we can write \eqref{prop-dist-main-3} as 
\begin{equation}\label{dist-gen-xi-upsilon}
\phi(Q,U,\varXi)
=\sum_{
\substack{ V\in K^{2n}/ \varUpsilon\\
 V\equiv U \,\mathrm{mod}\,\varXi}
}
\phi_{X}(Q,V,\Upsilon).
\end{equation}
Furthermore, we have the following result.

\begin{lem}\label{extra-prop-dist-13}
Let $\varTheta=\Theta\times \Theta^*\in \mathcal{W}_{n}$ such that
$\Xi\cap \varTheta\subset \varXi\cap \varUpsilon$. Suppose that 
$\varXi\cap \varUpsilon\subset \varUpsilon$ induces a bijective map
\begin{equation}\label{bij-map-dist-123-23}
\varXi\cap \varUpsilon\,\big/\,\varXi\cap \varTheta
\ \xrightarrow{\sim}\
\varUpsilon \,\big/\,\varUpsilon\cap \varTheta.
\end{equation}
Then, for any $V\in X$,
\begin{equation}\label{phi-Xi-upsioon-12}
\phi_X(Q,V,\varUpsilon)
=\sum_{
\substack{
W\in K^{2n}/\varTheta\\
W\equiv V\,\mathrm{mod}\, \varXi\cap \varUpsilon
}
}
\phi_X(Q,W,\varTheta).
\end{equation}
The bijection \eqref{bij-map-dist-123-23} will be satisfied, for example,
 if 
 \begin{equation}\label{coprime-condo}
 \text{
  $[\Xi:\Theta\cap \Xi]_{\mathcal{O}}$ is coprime to
 $[\Theta:\Theta\cap \Xi]_{\mathcal{O}}$.
 }
 \end{equation}
\end{lem}

\begin{proof}
By  \Cref{defin-space-dist} \eqref{prop-dist-main-1}, the right-hand side of \eqref{phi-Xi-upsioon-12} is equal to
\[
\frac{1}{[\Upsilon:\Xi\cap \Upsilon][\Theta:\Theta\cap \Upsilon]}
\sum_{Z=(z,z^*) \in \varUpsilon/\varUpsilon\cap \varTheta}
e(-u|v^*+z^*-u^*)e(u-v|z^*)
\phi(Q,V+Z,\varTheta).
\]
On the other hand, we have the obvious identity
\[
[\Theta:\Xi\cap \Theta]
=
[\Theta:\Theta\cap \Upsilon]
[\Theta\cap \Upsilon:\Upsilon\cap \Xi].
\]
Thus, \eqref{phi-Xi-upsioon-12} follows immediately 
from \eqref{bij-map-dist-123-23} after showing that
$[ \Upsilon: \Upsilon\cap \Xi ]=[ \Upsilon\cap \Theta : \Xi \cap \Theta   ]$.
This, in turn, follows from the indentity
\begin{equation}\label{iden-mod-inex-upsi-Xi-theta}
[ \Upsilon: \Upsilon\cap \Xi ]_\O
[ \Upsilon\cap \Xi : \Xi \cap \Theta   ]_\O
=
[ \Upsilon: \Upsilon\cap \Theta ]_\O
[ \Upsilon\cap \Theta : \Xi \cap \Theta   ]_\O
\end{equation}
and the bijection \eqref{bij-map-dist-123-23}.

As for the second part of the lemma, observe that the map \eqref{bij-map-dist-123-23} is clearly injective. Thus, it will be bijective if and only if
\begin{equation}\label{ident-Xi-Xi-Upsilon-Theta-kok}
[\Upsilon\,:\,\Upsilon\cap \Theta]_{\mathcal{O}}
=[ \Xi\cap \Upsilon\,:\,\Xi \cap \Theta]_{\mathcal{O}}
\end{equation} 
and 
$[\Upsilon^*\,:\,\Upsilon^*\cap \Theta^*]_{\mathcal{O}}
=[ \Xi^*\cap \Upsilon^*\,:\,\Xi^* \cap \Theta^*]_{\mathcal{O}}$. However,
\eqref{ident-Xi-Xi-Upsilon-Theta-kok} follows directly from \eqref{iden-mod-inex-upsi-Xi-theta} after observing that
$[ \Upsilon\cap \Xi : \Xi \cap \Theta   ]_\O$ divides the left-hand side of \eqref{ident-Xi-Xi-Upsilon-Theta-kok} and is coprime to
$[ \Upsilon\cap \Theta : \Xi \cap \Theta   ]_\O$ by \eqref{coprime-condo}.
\end{proof}

We finish this section by formally describing how to obtain a distribution in $\mathcal{D}_S(\mathbf{G}_n)$ from distributions in $\mathcal{D}_S(\mathbf{G}_1)$.

\begin{prop}\label{product of distributions}
Given the  distributions $\phi_i\in \mathcal{D}_S(\mathbf{G}_1)$, $i=1,\dots,n$,
there exists a distribution $\phi\in \mathcal{D}_S(\mathbf{G}_n)$ such that
\begin{equation}\label{defin-phi-general}
\phi(X^{\mathbf{t}},X^\mathbf{r},U,\varLambda)=
\prod_{i=1}^n \phi_i(X^{t_i},X^{r_i},u_i,u_i^*,\varLambda_i)
\end{equation}
for every  $\mathbf{t},\mathbf{r}\in \Z_{\geq 0}^n$, $U=(u_1,\dots,u_n, u_1^*,\dots, u_n^*)\in K^{2n}$, and $\varLambda=\prod_{i=1}^n \varLambda_i$, where
$
\Lambda_i=\Lambda_i\times \Lambda_i^*.$
We denote this distribution by $\prod _{i=1}^n \phi_i$.
\end{prop}

\begin{proof}
Starting with \eqref{defin-phi-general}, we can extend this definition to all
of $Q=(q(X),q^*(X))\in H^2$, with $q(X)=\sum_{\mathbf{t}}q_{\mathbf{t}}X^{\mathbf{t}}$
and $q^*(X)=\sum_{\mathbf{r}}q_{\mathbf{r}}X^{\mathbf{t}}$, by
\[
\phi(Q,U,\varLambda)
=\sum_{\mathbf{t},\mathbf{r}}
\,
q_{\mathbf{t}}
\cdot
q^*_{\mathbf{r}}
\cdot 
\phi(X^{\mathbf{t}},X^{\mathbf{r}},U,\varLambda).
\]

Now, for a general $\varXi\in \mathcal{W}_n$ we define $\phi(Q,U,\varXi)$
using \eqref{prop-dist-main-3} with $\varUpsilon=\varLambda$.
The fact that $\phi$ satisfies the properties of a distribution follows from the fact that
the 
$\phi_i$ are  one-dimensional distributions.
\end{proof}

\subsection{Eisenstein-Kronecker Distributions}
\label{Eisenstein-Kronecker Distributions}

Let $L$ be a lattice in $\C$.
For integers $k\geq 0$ and $l>0$, and  $z,w\in \C$, we define the Eisenstein-Kronecker series
\[
E_{l}^k(z,w,L,s):= \sideset{}{'}\sum_{x\in L+z}e( x|w)\frac{1}{x^l}\frac{\overline{x}^k}{|x|^{2s}}.
\]
This series is absolutely convergent for $\mathrm{Re}(s)>1+\frac{k-l}{2}$ and has analytic continuation to $\C$.
In order to state the functional equation for Eisenstein-Kronecker series, it is convenient to introduce 
\[
E_{k,l}(z,w,L,s):=\Gamma(l+s)A(L)^{s-k}E_l^k(z,w,L,s),
\]
where $A(L)=\mathrm{Vol}(L)/\pi$. 
This series is related to the series  $H_k(s,z,u,L)$ in \cite{CS} via
\be\label{H-Kronecker}
e(-z|w)E_{k,l}(z,w,L,s)=
H_{k+l}(s+l,z,-2i\mathrm{Vol}(L)w,L).
\ee

The Eisenstein-Kronecker series satisfy the functional equation
(cf.  \cite[Equation (20),\S 4]{CS}) 
\be\label{symmetry-E-k-l}
E_{k,l}(z,w,L,s)=e(z|w)\cdot E_{l-1,k+1}\left(-2i\mathrm{Vol}(L)w,\frac{z}{-2i\mathrm{Vol}(L)},L,-s\right)
\ee
or, equivalently, 
\[
E_{k,l}(z,w,L,s)
=\frac{e(z|w)}{\big(\,2i\mathrm{Vol}(L)\,\big)^{l+k}}
\cdot 
E_{l-1,k+1}\left(-w,z,L^*,-s\right).
\]
When we evaluate at $s=0$, we shall simply write $E_l^k(z,w,L) := E_{l}^k(z,w,L,0)$ and   $E_{k,l}(z,w,L) := E_{k,l}(z,w,L,0)$ and refer to it as an Eisenstein-Kronecker number. 

We also have the following identities: Suppose $L'$ is a sublattice of $ L$
and  $c\in \C ^{\times}$ then
\[
E_l^k(z,w,L',s)=\frac{1}{[L:L']}\sum_{\tau \in (L')^*/L^*}e(-z|\tau)
E_l^k(z,w+\tau,L,s),
\]
\[
E_l^k(z,w,L,s)=
\sum_{\tau \in L/L'}
E_l^k(z+\tau,w,L',s), \ \text{and}
\]
\begin{equation}\label{homog-Eisenstein-k-l}
E_{k,l}(cz,\overline{c}^{-1}w,cL)=\frac{1}{c^{l+k}}E_{k,l}(z,w,L).
\end{equation}

The Eisentein-Kronecker sums define a distribution $E(\cdot,s)$ on $\mathbf{G}_1$ as follows
\[
E(Q,U,\varXi,s):=E^k_{l}(u,u^*,\mathfrak{a},s)
\]
 for $Q=(X^k,X^{l-1})$, $U=(u,u^*)\in K^2$ and $\varXi=\mathfrak{a}\times \mathfrak{a}^*$,
 where $\aa$ nonzero fractional ideal of $K$. Let $E^n(\cdot,s)$ be the distribution in $\mathcal{D}_{\C}(\mathbf{G}_{n})$
 defined as the product of the  distribution $E(\cdot,s)$ with itself $n$-times according to
\Cref{product of distributions}, i.e.,
\begin{equation}\label{defin-F-general}
E^n(X^{\mathbf{t}},X^{\mathbf{r}},U,\varLambda,s):=
\prod_{i=1}^n\Gamma(r_i+1) E_{r_i+1}^{t_i}(u_i,u_i^*,\Lambda_i,s),
\end{equation}
for  all $\mathbf{r}=(r_1,\dots, r_n)$ and $\mathbf{t}=(t_1,\dots, t_n)$ in $\Z^n_{\geq 0}$. 

Same as above,   when we evaluate at $s=0$ we will simply write $E^n(Q,U,\varXi)$.

\subsection{Colmez' distribution}

 In this section we  recall a distribution $F(s)$ introduced by Colmez in \cite{Col}, which is the base of the work of \cite{Oba} and \cite{FKW}, and then establish its connection with $E^n$.

For $Q=(N_{\overline{M}}(X)^k,q^*(X))\in H_k\times H$ and
 $(Q,U,\varXi)\in \mathbf{G}_{n,k}$, $k\geq 0$, we define
\[
F(Q,U,\varXi,s)
=
\sum_{\mathbf{r}}
q_{\mathbf{r}}^*
\sideset{}{'}\sum_{x\in \Xi+u}
e(x|u^*)\cdot 
\prod_{j=1}^n
\frac{r_j!}{x_j^{1+r_j}}
\cdot 
\frac{\overline{N_M(x)}^k}{|N_M(x)|^{2s}},
\]
where the sum is taken over all $x$ such that the denominator does not vanish.
The case $\varXi=\Lambda\times \Lambda^*$ and $q^*(X)=X^{\mathbf{r}}$ was first considered in
 \cite[Chapitre II. \S 0]{Col} where it is shown that the sum converges absolutely for $\text{Re}(s)>1+\frac{k}{2}$.  Since $F(s)=F(\cdot,s)$ satisfies the properties of   \Cref{defin-space-dist}, we can use   \eqref{prop-dist-main-3} to deduce the absolute convergence on all of $(Q,U,\varXi)\in \mathbf{G}_{n,k}$; in particular 
 $F(s)\in \mathcal{D}_\C(\mathbf{G}_{n,k})$ for $\text{Re}(s)>1+\frac{k}{2}$.
 For example, to see that $F(s)$ satisfies   \Cref{defin-space-dist} \eqref{prop-dist-main-1},
 we simply introduce the character identity (cf. \eqref{character-identity-2}), for $x\in \Xi+u$ and $v\in \Xi$,  
\begin{equation*}\label{char-iden-sigma}
\frac{ 1 }{ [\Upsilon:\Xi\cap \Upsilon] }
\sum_{v^*\in \Xi^*/\Xi^*\cap \Upsilon^*} e(x-u-v| v^*)
=
\begin{cases}
1, &  \text{if $x-u-v\in \Xi\cap\Upsilon$,}\\
0, & \text{if $x-u-v\in  \Upsilon\setminus (\Xi \cap \Upsilon)$,}
\end{cases}
\end{equation*}
 into the expression  
 \[
 \sum_{\mathbf{r}}
   q_{\mathbf{r}}^*
\sum_{v\in \Xi/\Xi\cap \Upsilon}\,
 \sideset{}{'}\sum_{x\in\, \Xi\cap \Upsilon+v+u} 
e(x|u^*)\cdot 
\prod_{j=1}^n
\frac{r_j!}{x_j^{1+r_j}}
\cdot 
\frac{\overline{N_M(x)}^k}{|N_M(x)|^{2s}}.
\]

 According to \cite[Th\'eor\'eme 2]{Col}, 
$F(s)$ has an analytic continuation
to $X(n,k)$, where
\begin{equation}\label{conditions-for-analyt.-cont}
X(n,k):=
\begin{cases}
\{s\in \C: \mathrm{Re}(s)>\frac{k}{2}-\frac{1}{2(n-2)}\},\ \text{if $n\geq 3$},\\
\C,\ \text{if $n=1$ or 2.} 
\end{cases}
\end{equation}
Originally, this result is proven for $\varXi$ of the form $\varLambda=\Lambda\times\Lambda^*$, however by \Cref{defin-space-dist} \eqref{prop-dist-main-1}  this extends to any $\varXi\in \mathcal{W}_n$. Whenever there is analytic continuation at $s=0$, we will write  $F$ instead of  $F(0)$. Therefore,
$F\in \mathcal{D}_\C(\mathbf{G}_{n,k})$ for $n=1,2$ with $k\geq 0$, and for $n\geq 3$ with $k=0$. 
Note that for $n=1$ the above definition coincides 
with the  classical Eisenstein-Kronecker series  defined in  
\Cref{Eisenstein-Kronecker Distributions}. 

The following result from Colmez  \cite[Th\'eor\'eme 4]{Col} relates  the distributions $F$ and $E^n$.

\begin{prop}\label{identity-colmez}
The distributions $F$ and $E^n$ coincide on $\mathbf{G}_{n,k}$
for $n=1,2$ with $k\geq 0$ and for $n\geq 3$ with $k=0$.
\end{prop}

\begin{proof}
By the definition of $F$, 
it will be enough to prove the result for
$F(Q,U,\varLambda)$ with $q(X)=N_{\overline{M}}(X)^k$ and $q^*(X)=X^{\mathbf{r}}$, $\mathbf{r}\in \Z_{\geq 0}^n$.
Suppose
$
q(X)=\sum_{\mathbf{t}}q_{\mathbf{t}}X^{\mathbf{t}}$. 
Then we will prove that
\begin{equation*}\label{defin-Fk}
F(Q,U,\varLambda)
=
\sum_{\mathbf{r},\mathbf{t}}
\,
q_{\mathbf{t}}
\cdot
\prod_{i=1}^n \Gamma(e_i)E_{e_i}^{t_i}(u_i,u_i^*,\Lambda_i),
\end{equation*}
where $e_i=r_i+1$ for $i=1,\dots, n$.

If we apply the result of \cite[Th\'eor\'eme 4]{Col} to 
 $\mathcal{L}:=M^{-1}$, $\mathcal{M}:=\overline{M}^t=M^h$ and $u^*=0$, we obtain
 \[
F(Q,u,0,\varLambda)=\Box_k\left( \prod_{i=1}^n \pi^{e_i}H_{e_i}(u_i, -zM^h_i,\Lambda_i, e_i))\right),
\]
where  $-zM_i^h$ denotes the $i$th entry of the vector $-zM^h$, $H_k$ is the series defined in  \eqref{H-Kronecker} and
\[
\Box_k(\varphi)=\left(\frac{i}{\pi}\right)^{nk}\prod_{i=1}^n\left( \frac{\partial}{\partial z_i}\right)^k\varphi(z)|_{z=0}.
\]
 However, to obtain a similar result for the sum $F(Q,U,\varLambda)$, for 
 $U=(u,u^*)\in K^{2n}$,
we need to simply  modify the 
 distribution $T_{u,\underline{t},\Lambda}(z)$ defined in \cite[page 184]{Col} for
 \[
T_{u,u^*,\mathbf{t},\Lambda}=\sum_{\omega\in \Lambda}\prod_{i=1}^n\frac{\Gamma(t_i)}{L_i(z)^{t_i}}e^{2\pi i \langle L_i(z)|u^* \rangle }\delta_{w+u},
 \]
thus obtaining
\[
F(Q,U,\varLambda)=\Box_k\left( \prod_{i=1}^n \pi^{e_i}
     H_{e_i}(u_i, -zM_i^h+u_i^*,\Lambda_i, e_i)\right),
\]
Noticing from \cite[ Chapitre II, \S 4]{Col} that
\[
\frac{\partial}{\partial y}H_k(x,y,\Lambda_i,s)=i\pi H_{k+1}(x,y,\Lambda_i,s)
\]
we deduce that
\[
\prod_{i=1}^n\left( \frac{\partial}{\partial z_i}\right)^k
\left(\prod_{i=1}^n \pi^{e_i} H_{e_i}(u_i,-zM_i^h+u_i^*,\Lambda_i, e_i)\right )
\]
equals
\begin{align*}
\sum_{\underline{t}=nk}
q_{\mathbf{t}}
\prod_{i=1}^n(-i\pi)^{t_i}\pi^{e_i}H_{e_i+t_i}(u_i,-zM_i^h+u_i^*,\Lambda_i, e_i)\\
=
(-\pi i)^{nk}\sum_{\underline{t}=nk}
q_{\mathbf{t}}
\prod_{i=1}^n\pi^{e_i}H_{e_i}(u_i,-zM_i^h+u_i^*,\Lambda_i, e_i).
\end{align*}
The result is now a consequence of the identity $\pi^{e_i}H_{e_i+t_i}(u_i,u_i^*,\Lambda_i, e_i)=\Gamma(e_i)E_{e_i}^{t_i}(u_i,u_i^*,\Lambda_i)$.
\end{proof}

For $\sigma\in M_{n}(K)$
we define 
the  Dedekind sum $D(s,\sigma)$ by
$
\sigma F(s)$ if $\det(\sigma)\neq 0$ and by zero if $\det(\sigma)=0$.
Observe, from the very definition, that for any $t\in K^{\times}$
\begin{equation}\label{hom-sigma-D}
D(s,t\sigma)
=
D(s,\sigma).
\end{equation}
By the \Cref{identity-colmez}, when $s=0$, we can define $D(\sigma)$ to be $\sigma E$ so that it is now defined on all of $\mathbf{G}_n$.


\subsection{The Sczech-Eisenstein-Kronecker cocycle}
\label{The Eisenstein cocycle}

We now recall the cocycle constructed in \cite{Oba} and \cite{FKW} up to a minor modification. This cocycle, built in the style of Sczech \cite{Scz}, will serve as a foundation for applying the methods of Charollois-Dasgupta \cite{CD} to our setting.

 Let  $\sigma=(\sigma_1,\dots,\sigma_n)$ be a matrix of  $\mathrm{GL}_n(K)$. 
Define, for $x\in K^n$,
\[
f(\sigma)(x)=\frac{\det(\sigma)}{\langle x,\sigma_1\rangle\dots\langle x,\sigma_n\rangle},
\]
and for any homogeneous polynomial $q^*\in H$
we let
\begin{align*}
f(\sigma)(q^*,x):=q^*(-\partial_{x_1},\dots,-\partial_{x_n})f(\sigma)(x)=\det(\sigma)\sum_{\mathbf{r}}q^*_{\mathbf{r}}(\sigma)
\prod_{j=1}^n\frac{r_j!}{\langle x,\sigma_j\rangle ^{1+r_j}},
\end{align*}
 where $ q^*_{\mathbf{r}}(\sigma)$ are the coefficients in the expansion of the  polynomial $q^*\sigma^*$, i.e.,
\be\label{coeff-poly-P}
q^*\sigma^*(X)=q^*(X\sigma^t)
=
\sum_{\mathbf{r}} q^*_{\mathbf{r}}(\sigma)X^{\mathbf{r}}.
\ee

Let $\mathfrak{A}=(A_1,\dots,A_n)$ be an $n$-tuple of $\mathrm{GL}_n(K)$. 
For every $x\in K^n-\{0\}$ and $k=1,\dots,n$, there is a smallest index $j_k$ such that the dot product $\langle x, A_{kj_k}\rangle$ is nonzero, where $A_{kj_k}$ denotes the $j_k$-th column of $A_k$; recall that $\langle \cdot,\cdot \rangle$ is the usual dot product defined in \eqref{dot-prod}.
For ease of notation, let $\sigma_k$ denote $A_{kj_k}$ for $1\leq k\leq n$, and by $\sigma=(\sigma_{ij})$ the square matrix whose $k$th column in given by $\sigma_k$. For this $\sigma$, associated to the given $x$, we define
\[
\psi(\a)(q^*,x)=f(\sigma)(q^*,x).
\]
The map $\psi$ satisfies the identity
\begin{equation*}\label{homogen-psi}
\psi(A\mathfrak{A})(q^*,x)=\det(A)\psi(\mathfrak{A})(q^*A^*,xA)
\end{equation*}
for all $A\in \mathrm{GL}_n(K)$ and the cocycle identity
\begin{equation*}\label{cocycle-property}
\sum_{i=0}^n(-1)^i\psi(A_0,\dots, \widehat{A}_i,\dots, A_n)=0
\end{equation*}
 as can be verified from  \cite[Section 2.2]{FKW}. We extend $\psi$ to all of
 $\Z[\mathrm{GL}_n(K)^n]$ by linearity and define for $\mathfrak{A}\in\Z[\mathrm{GL}_n(K)^n]$  
\begin{equation*}\label{eisco}
\Psi(s,\mathfrak{A},Q,U,\varXi)
=\sum_{x\in \Xi+u}e( x|u^*)\cdot
  \psi(\mathfrak{A})(q^*,x)
  \cdot 
    \frac{\overline{N_M(x)}^k}{|N_M(x)|^{2s}},
\end{equation*}
for $Q=(N_{\overline{M}}(X)^k),q^*(X))$.
Just as for $F(s)$ (cf. \cite[Theorem 3.13]{FKW}), we also have absolute convergence of $\Psi(s,\mathfrak{A})$ on $\mathbf{G}_{n,k}$ for $\text{Re}(s)>1+\frac{k}{2}$. Observing that  \eqref{homogen-psi} and  \eqref{action-on-SE} imply $
\Psi(s,A\mathfrak{A})=A\Psi(s,\mathfrak{A})$ we have  that $\Psi(s):\Z[\mathrm{GL}_n(K)^n]\to \mathcal{D}_\C(\mathbf{G}_{n,k})$ is a homogenous cocycle that represents a cohomology class in 
$H^{n-1}\big(\mathrm{GL}_n(K),\, \mathcal{D}_\C(\mathbf{G}_{n,k})\,\big)$. 
Whenever there is analytic continuation at $s=0$ we will omit the $s$ from the notation.

Since $f(\sigma)$ does not change if $\sigma_j$ is replaced by $t\sigma_j$, $t \in \C^\times$, it follows that
\begin{equation*}\label{homogen-psi}
\Psi(s,t\,\mathfrak{A})=\Psi(s,\mathfrak{A}).
\end{equation*}

For $\mathfrak{A}=(A_1,\dots, A_n)\in \mathrm{GL}_n(K)^n$, let $\sigma(1)$ be the matrix whose $i$th column  is  the first column of $A_i$. Then, similarly as above, we see that
the sum
\[
\Phi(s,\a):=D(s,\sigma(1))
\]
also satisfies the cocycle property and thus $\Phi(s):\Z[\mathrm{GL}_n(K)^n]\to \mathcal{D}_\C(\mathbf{G}_{n,k})$ also represents a cohomology class in 
$H^{n-1}\big(\mathrm{GL}_n(K),\, \mathcal{D}_\C(\mathbf{G}_{n,k})\,\big)$. Furthermore, according to the analytic continuation properties of $F$ at $s=0$, we have that 
 $\Phi=\Phi(0):\Z[\mathrm{GL}_n(K)^n]\to \mathcal{D}_\C(\mathbf{H}_{n})$  represents a cohomology class in 
$H^{n-1}\big(\mathrm{GL}_n(K),\, \mathcal{D}_\C(\mathbf{H}_{n})\,\big)$, where $\mathbf{H}_{n}=\mathbf{G}_{n,k}$ for $n=1,2$ with $k\geq 0$, and $\mathbf{H}_{n}=\mathbf{G}_{n,0}$ for $n\geq 3$.

In \Cref{The smoothed Eisenstein-Kronecker cocycle} we will show that  $\Psi(s)$ and $\Phi(s)$ coincide after smoothing by a prime $\l$ of $K$ and, moreover, they coincide with 
the Eisenstein-Kronecker cocycle $\Phi_\l$.

\subsection{The partial Hecke $L$-function and its functional equation}
\label{sec:heckeL}
In order to describe  the parametrization of the partial $L$-function via the cocycle $\Psi(s)$,
we  define a slightly more general partial Hecke $L$-function in the following way.
Let $\alpha\mapsto \overline{\alpha}$ denote  the non-trivial automorphism of $K$ which 
we consider as extended to all of $\overline{K}$--- an algebraic closure of $K$. Let $F^*$ denote the image of 
$F$ under this automorphism. 
Consider the  pairing $e(\cdot,|\cdot)\,:\,F\times F^*\to \C^{\times}$ given by
\[
e(x|y)=e^{2\pi i \mathrm{Tr}_{F/\Q}( x\overline{y}) }.
\]
For a fractional ideal $\mathfrak{b}$ of $F$ let $\mathfrak{b}^*$
denote the fractional ideal of $F^*$ dual to $\b$
 with respect to this pairing, i.e., $\mathfrak{b}^*=\{y\in F^*: e(\mathfrak{b}|y)=1\}$.
 Note that 
$
\mathfrak{b}^*=(\overline{\mathfrak{b}\mathfrak{D}_F})^{-1},
$
where $\mathfrak{D}_F$ is the absolute different of $F$.

For an integral ideal $\mathfrak{g}$ of $F$, let 
 $r\in F$ and $r^*\in F^*$ such that 
\begin{equation}\label{condition-r-gamma}
\text{$\mathfrak{g}\, r \in \mathfrak{b}$
and 
$\overline{\mathfrak{g}}\,r^*\in \mathfrak{b}^*$},
\end{equation}
and also
\be\label{condition-gamma}
\text{ $e(\,\epsilon r|r^*)$ for all $\epsilon\in U_{\mathfrak{g}}$. }
\ee
These two conditions imply
\be\label{implicat-r-gamma}
\text{
 $e(\epsilon r|\b^*)=e( r|\b^*)$
 and   
 $e(\b|\overline{\epsilon} r^*)=e(\b| r^*)$ 
 for all $\epsilon\in U_{\mathfrak{g}}$. }
\ee

Suppose $\phi:F^{\times}\to \C$ is a map such that $\phi(xU_{\mathfrak{g}})=\phi(x)$ and $\phi(x)=\varphi(x)\lambda(x)$   for all $x\in F^{\times}$, where $\lambda=\lambda_{F,k,l}$ is as in \Cref{Main results} and $\varphi:F^{\times}\to \C$ is a bounded map.  
Then    we define the function
\be\label{partial-L-f-a,r,gamma,s}
\L_{\mathfrak{g}}(\phi,\mathfrak{b},r,r^*,s)=
\sum_{\alpha \in (\mathfrak{b}+r)/U_{\mathfrak{g}}}
   e(\,\alpha| r^*\,)
    \frac{\phi(\alpha)}{|N_{F/\Q}(\alpha)|^s}.
\ee
The extended Hecke character $\chi$ from \Cref{Main results}   induces the map $\phi(x)=\chi((x))$ for $x\in F^{\times}$, which we will denote by $\chi$ as well, so that we have
\begin{equation*}\label{relation-partial-partial}
L_{\mathfrak{g}}(\mathfrak{a},\chi,s)=
\frac{\chi(\aa)}{N_{F/\Q}(\aa)^s}
\,
\L_\mathfrak{g}(\chi,\aa^{-1}\mathfrak{g},1,0,s)
.
\end{equation*}
Furthermore, for $\mathfrak{g}=\mathfrak{f}$ note that
\begin{equation*}
\L_\mathfrak{f}(\chi,\aa^{-1}\f,1,0,s)
=
\L_\f(\lambda,\aa^{-1}\f,1,0,s).
\end{equation*}
In the case where $\mathfrak{a}$ is not an integral ideal, we simply take $r\in \aa^{-1}$
such that $r\equiv 1 \text{ mod}^*\,\f$, in which case we have
\be\label{relation-partial-partial-non-integral-a}
L_{\f}(\mathfrak{a},\chi,s)=
\frac{\chi(\aa)}{N_{F/\Q}(\aa)^s}
\,
\L_\f(\lambda,\aa^{-1}\f,r,0,s).
\ee

For the Hecke character $\chi:I(\f)\to \overline{\Q}^{\times}$, we define the dual character $\chi^*:I(\overline{\f})\to \overline{\Q}^{\times}$ on ideals of $F^*$ prime to $\overline{\f}$  by 
$\chi^*(\mathfrak{a})=\chi^{-1}(\overline{\aa})N_{F^*/\Q}(\overline{\aa})^{-1}$, so that 
$\chi^*((a))=\lambda^*(a)$ for $(a)\in P_{\overline{\f}}$, where $\lambda^*=\lambda^*_{F^*,l-1,k+1}:F^{*\times}\to K^{\times}$ is given  by
\[
\lambda^*(a)=\overline{N_{F^*/K}(a)}^{l-1}N_{F^*/K}(a)^{-(k+1)}=\frac{\lambda_{F,k,l}^{-1}(\overline{a})}{N_{F/\Q}(\overline{a})}.
\]
The assumption  \eqref{character-lambda} on $k$ and $l$ imply that  $\lambda^*(\epsilon)=1 $
for all units $\epsilon$ in the group
$
U_{F^*,\overline{\f}}^*.
$
For the
$L$-function $\mathcal{L}_\f(\lambda, \mathfrak{b},r,r^*,s)$ in \eqref{partial-L-f-a,r,gamma,s}, the function $\mathcal{L}_{\overline{\f}}(\lambda^*,\mathfrak{b}^*,-r^*,r,s)$ is well defined. Furthermore, these $L$-functions  satisfy the functional equation  ( cf. \cite{CS})
\be\label{final-fun-Eq-new-L}
\mathcal{Z}(\lambda,\mathfrak{b},r,r^*,s)=(-1)^{n(l-1+s)}W
\mathcal{Z}(\lambda^*,\mathfrak{b}^*,-r^*,r,-s),
\ee
where
\[
\mathcal{Z}(\lambda,\mathfrak{b},r,r^*,s)
:= \Gamma(l+s)^n (2\pi i)^{n(k-s)}
\mathcal{L}_\f(\lambda, \mathfrak{b},r,r^*,s),
\]
and
\begin{equation}\label{functional-factor-2}
W=W(\lambda, \mathfrak{b},r,r^*)=
\frac{ e(r|r^*) }
        {   i^n N_{F/\Q}(\mathfrak{b}) \sqrt{|\mathfrak{d}_F|}  } 
      =
      \frac{  e(r|r^*)  }
        {   (2i)^n\mathrm{Vol}(\mathfrak{b}) } .
     \end{equation}
     Note that
     $
W(\lambda,\mathfrak{b},r,r^*)    W(\lambda^*,\mathfrak{b}^*,-r^*,r)=(-1)^n. 
     $

\subsection{Parametrization of the Hecke $L$-function in terms of $\Psi(s)$} 
\label{Partial $L$-function in terms of $K$}
Let  $\lambda=\lambda_{k,l}$, $\mathfrak{b}$, $r\in F$, $r^*\in F^*$ as in the previous section.
We can express the partial $L$-function $\L_\f(\lambda,\b,r,r^*,s)$ in terms of $K$ as follows.
Let $\{m_i\}$ be a basis for $F/K$ and let $m:=(m_1,\dots, m_n)^t$. 
Denote by $\{m_i^*\}$ the basis of $F^*$ dual to $\{m_i\}$ 
 with respect to the pairing $F\times F^*\to K$ given by 
 $(x,y)\mapsto \mathrm{Tr}_{F/K}(x\overline{y})$, i.e.,
 $\mathrm{Tr}_{F/K}(m_i\overline{m_j^*})=\delta_{ij}$, where $\delta_{ij}$ is the Kronecker delta. 
  For the basis $\{m_i\}$ we let $M\in R_n$ denote matrix \eqref{defin-matrix-M},
  so that $M^*$ is the corresponding matrix of $\{m_i^*\}$. 
We have the bijective map
$
(m, m^*  ):K^{2n}\to F\times F^*$
given by
$(x,x^*)\mapsto (x,x^*)(m,m^*)=(xm, x^*m^*)$.

Let 
$\mathbf{G}_{n}(\lambda,\mathfrak{b},r,r^*)$ be the set of tuples $(Q,U,\varXi)\in \mathbf{G}_n$ for which there exists an $m$ such that
\begin{equation}\label{r and gamma}
\b\times \b^*=\varXi (m,m^*), 
\
(r,r^*)=U(m,m^*),\
Q=(N_{\overline{M}}(X)^k,N_{M^{-t}}(X)^{l-1}).
\end{equation}

We define the representation $\varrho=\varrho_M:\mathcal{O}_F^{\times}\to \mathrm{GL}_n(K)$ in the following way: For a unit $\eta\in \mathcal{O}_F^{\times}$, 
 let $\varrho(\eta)$ be the matrix 
in $\mathrm{GL}_n(K)$
representing multiplication-by-$\eta$ with respect to the basis $\{m_i\}$, i.e.,
$
(m_1\eta,\dots,m_n\eta)^t
=\varrho(\eta)m.$
More precisely,
\begin{equation}\label{matrix-rep-trace}
\varrho(\eta)=
\begin{pmatrix}
\mathrm{Tr}_{F/K}(m_1\eta\, \overline{m_1^*})&\dots 
&\mathrm{Tr}_{F/K}(m_1\eta\, \overline{m_n^*})\\
\vdots &\ddots & \vdots\\
\mathrm{Tr}_{F/K}(m_n\eta\, \overline{m_1^*})&\dots 
&\mathrm{Tr}_{F/K}(m_n\eta \,\overline{m_n^*})
\end{pmatrix}
\end{equation}
Note that 
$\varrho_{M^*}(\overline{\eta})=\varrho_{M}(\eta)^h$
is the corresponding matrix for
 multiplication-by-$\overline{\eta}$ in the basis $\{m_i^*\}$ of $F^*/K$, i.e.,
$
(
m_1^* \overline{\eta},\dots,
m_n^*\overline{\eta}
)^t
=\varrho_M(\eta)^h m^*.$

The  conditions     \eqref{condition-gamma} and  \eqref{implicat-r-gamma} are equivalent to
\begin{equation*}\label{condition-v}
e( xA|u^*)=e(x|u^*)
\quad \text{and}\quad
e( uA| u^*)=e(u|u^*)
\end{equation*}
for  all $x\in \Xi$ and all $A=\varrho(\eta)$ with $\eta\in U_{\mathfrak{f}}$.

Let $N_M$, $N_{M^*}$, $N_{M^{-t}}$ and $N_{\overline{M}}$ 
be as in \Cref{Notation} .
 Note that
\[
N_{F/K}(r)=N_M(u)
\quad \text{and}\quad
N_{F^*/K}(r^*)=N_{M^*}(u^*),
\]
and also
$
e^{2\pi i\mathrm{Tr}_{F/\Q}(r\overline{r^*})}=e(u|u^*).
$
Furthermore, we also have the following relationship $(\varrho(\eta)N_M)(x)=N_{F/K}(\eta)N_M(x)$ for any $\eta\in \mathcal{O}_F^{\times}$.
Thus we have
\begin{equation*}\label{Lf=LLambda}
\L_\f(\lambda,\mathfrak{b},r,r^*,s)=\mathcal{L}_\f(Q,U,\varXi,s),
\end{equation*}
where
\[
\mathcal{L}_\f(Q,U,\varXi,s):=\sum_{x\in (\Xi+u)/\mathfrak{U}_{\mathfrak{f}}} e( x|u^*)
\frac{\overline{N_M(x)}^k}{N_M(x)^l}\frac{1}{|N_M(x)|^{2s}},
\]
and $\mathfrak{U}_\f:=\varrho(U_\f)$.

For $(Q,U,\varXi)\in \mathbf{G}_n(\lambda,\mathfrak{b},r,r^*)$ as in \eqref{r and gamma}, we see that  $(Q^*,U^*,\varXi^*)\in  \mathbf{G}_n(\lambda^*,\mathfrak{b}^*,-r^*,r)$. Observe that
\[
Q^*
=
(N_{M^{-t}}(X)^{l-1},N_{\overline{M}}(X)^k)
=
(N_{\overline{M^*}}(X)^{l-1},N_{ (M^*)^{-t} }(X)^k),
\]
then
\be\label{-dual-L-terms-K}
\mathcal{L}_{\overline{\f}}(\lambda^*,\mathfrak{b}^*,-r^*,r,s)=
\mathcal{L}_{\overline{\f}}(Q^*,U^*,\varXi^*,s).
\ee

Let $V_\f=V_{F,\f}$ be the free part of the group
$U_\f^{(1)}:=\{\epsilon\in U_\f:N_{F/K}(\epsilon)=1\}$, i.e., the group generated by units of infinite order in $U_\f^{(1)}$.
Choose generators $\epsilon_{1},\dots, \epsilon_{n-1}$ of $V_\f$ and let $\mathfrak{V}_\f=\varrho(V_\f).$  For the regulator
\[
R_{F/K}=R(\epsilon_1,\dots, \epsilon_{n-1})=\det (2\log|\rho_i(\epsilon_j)|), \qquad 1\leq i,j\leq n-1,
\]
let $
\rho=(-1)^{n-1}\mathrm{sign}(R_{F/K})
$. Define the $(n-1)$-chain in $\Z[\mathfrak{V}_\f^n]$
\be\label{cycle-defin}
\mathfrak{E}=\mathfrak{E}_{V_\f,M}=\rho\sum_{\pi}\mathfrak{A}_{M,\pi},
\ee
where  $\pi$  runs over all the permutations of the set 
$\{1,\dots, n-1\}$ and
\[
\mathfrak{A}_{M,\pi}:=[\varrho_M(\epsilon_{\pi(1)})|\cdots |\varrho_M(\epsilon_{\pi(n-1)})];
\]
here we are using the notation $[A_1|\cdots |A_{n-1} ]=(1,A_1,A_1A_2,\dots, A_1\cdots A_{n-1})$. 

Following the exact same argument as in  \cite[Theorem 3.2.5]{FKW} (the key step being  c.f. \cite[Lemma 3.2.3]{FKW}), we have the parametrization
\begin{equation}\label{theorem-paramretrization-L}
[U_{\mathfrak{f}}:V_{\mathfrak{f}}]\,\Gamma(l)^n\,
\L_\f(\lambda,\mathfrak{b},r,r^*,s)=
\det(M)^{-1}\Psi(s,\mathfrak{E},Q,U,\varXi),
\end{equation}
for all  $(Q,U,\varXi)\in \mathbf{G}_n(\lambda,\mathfrak{b},r,r^*)$.

As a consequence of \Cref{theorem-paramretrization-L} and the above functional equations
we have the following parametrization of the $L$-function by the Eisenstein-Kronecker cocycle.
\begin{cor}\label{dual-L-func-parametrization-via cocycle}

The special value
\[
[U_{\mathfrak{f}}:V_{\mathfrak{f}}]\cdot
\Gamma(l)^n
\cdot
(2\pi i)^{nk}
\cdot
\mathcal{L}_\f(\lambda,\mathfrak{b},r,r^*,0)
\]
is equal to
\[
W\cdot
   \det(\overline{M}) 
       \cdot
         (-2\pi i)^{n(l-1)}
          \cdot
              \Psi  (\mathfrak{E}^*,Q^*,U^*,\varXi^*),
\]
where $W=W(\lambda, \mathfrak{b},r,r^*)$ is as in  \eqref{functional-factor-2} and $\mathfrak{E}^*$ be as in \eqref{cycle-defin} for $V_{F^*,\overline{\f}}:=\overline{V_{F,\f}}$ and $M^*$, i.e.,
\be\label{dual-cycle}
\mathfrak{E}^*=\mathfrak{E}_{V_{\overline{\f},F^*},M^*}=\rho\sum_{\pi}\mathfrak{A}_{M^*,\pi}
\ee
where 
$
\mathfrak{A}_{M^*,\pi}=\mathrm{sign}(\pi)
 \big[\,
   \varrho^h( \epsilon_{\pi(1)} )   \,|\,   \cdots   \,|\,   \varrho^h( \epsilon_{\pi(n-1)} )
    \,\big].$

\end{cor}

\begin{proof} 
This parametrization is a consequence of    \eqref{final-fun-Eq-new-L}  and
 the identity
\[
[U_{F^*,\overline{\f}}:V_{F^*,\overline{\f}}]\,\Gamma(k+1)^n\,
\mathcal{L}_{\overline{\f}}(\lambda^*,\mathfrak{b}^*,-r^*,r,s)
=   \det(\overline{M})  \cdot \Psi(s,\mathfrak{E}^*,Q^*,U^*,\varXi^*);
\]
which follows \Cref{theorem-paramretrization-L}, bearing in mind  \eqref{-dual-L-terms-K}
and $[U_{F^*,\overline{\f}}:V_{F^*,\overline{\f}}]=[U_{F,\f}:V_{F,\f}]$, where $U_{F^*,\overline{\f}}:=\overline{U_{F,\f}}$. 

\end{proof}

\section{Arithmetic properties of  Eisenstein-Kronecker sums }
\label{Eisenstein-Kronecker numbers and properties}

In this section, we extend the integrality and congruence properties for Eisenstein numbers established in \cite[Chapter II, Proposition 3.3(iv) and §3.4, Equation (12)]{dS} to all Eisenstein-Kronecker numbers, without any restrictions on $k\geq 0$ and $l>0$. This is the content of  \Cref{$p$-adic properties of $E_{k,l}$}. 
We achieve this by deriving $p$-adic expansions of Eisenstein-Kronecker numbers
 (cf. \Cref{$p$-adic expansions of $E_{k,l}$}), which are obtained from the $p$-adic expansions of $E_1^0(u,0,L)$ given by Colmez-Schneps in \cite{CS}. Finally, we deduce corresponding properties for the distributions $E^n$
  and $D$ in Sections \ref{padic properties of En} and \ref{padic properties of D}, respectively.

\subsection{Complex periods and algebraicity of  Eisenstein-Kronecker sums}
\label{Complex periods}
Let $\mathfrak{f}_1$ be an integral ideal of $K$ and $\mathbf{a}$ a fractional ideal of $K$. We further assume that  the number $w_{\f_1}$ of roots of unity congruent to 1 mod $\f_1$ is equal to 1. 
According to 
 \cite[II, Lemma 1.4 (i) and (ii)]{dS} there exists a 
 grossencharacter $\varphi$ of $K$ of type $(1,0)$ and conductor $\f_1$ (i.e., $\varphi((a))=a$ for $a\in K^{\times}$ with $a\equiv 1 \mod{\f_1}$) and an elliptic curve $E$ over $K(\f_1)$,
 the ray class field of $K$ modulo $\f_1$, with complex multiplication by the ring of integers of $K$ such that $j(E)=j(\mathbf{a})$ and
 \[
\psi_{E/K(\f_1)}=\varphi \circ N_{K(\f_1)/K}, 
 \]
where $\psi_{E/K(\f_1)}$ is the \textit{grossencharacter of  type $A_0$}  to $E/K(\f_1)$ by the theory of Complex Multiplication. According to   \cite[Chap II \S1.4 ]{dS}, these assumptions  imply that
\[
\text{$K(\f_1)(E_{\mathrm{tors}})
$ is an abelian extension of $K$,}
\]
where $E_{\mathrm{tors}}$ is the torsion subgroup of $E(\C)$ and $K(\f_1)(E_{\mathrm{tors}})
$ is the field extension of  $K(\f_1)$ generated by the coordinates of the points in  $E_{\mathrm{tors}}$.

Let $L$ be a lattice of $\C$ and the Weierstrass isomorphism  $\xi: \C/L\xrightarrow{\sim} E(\C)$  is given by
$\xi(z,L)=(\wp(z,L),\wp'(z,L))$ and thus
\begin{equation}\label{Weierstrass model-E}
E:\ y^2=4x^3-g_2(L)x-g_3(L),\quad g_2(L), g_3(L)\in K(\f_1).
\end{equation}
Since $j(E)=j(\mathbf{a})$, then $L=\Omega_{\infty} \mathbf{a}$ for some $\Omega_{\infty}\in \C^{\times}.$ We also introduce the dual period 
$
\Omega_{\infty}^*:=-2i\mathrm{Vol}(\mathbf{a})\Omega_{\infty},
$
so that $L=\Omega_{\infty}^*\mathbf{a}^*$,
and the period
\[
\eta_{\infty}:=\frac{\pi}{\Omega_{\infty} \mathrm{Vol}(\mathbf{a})}
=-\frac{2\pi i}{\Omega_{\infty}^*}.
\]

For $u$ and $v$ in $K$ let $u'=\Omega_{\infty} u$ and $v'=\overline{\Omega}^{-1}_{\infty}v$. Then
by \eqref{homog-Eisenstein-k-l},
\begin{equation*}\label{ekl-ablian}
E_{k,l}(u',  v',L)
=
\frac{    E_{k,l}(u,v,\mathbf{a})    }
      {    \Omega^{l+k}_{\infty}        }
=\Gamma(l) \cdot
     \frac{   \eta_{\infty}^k  }
            {\Omega^{l}_{\infty}  }    
            \cdot       
  E_l^k(u,v,\mathbf{a}).
\end{equation*}
If $u\in K/\mathbf{a}$ and $v\in K/\mathbf{a}^*$, then this element belongs to 
$K^{\mathrm{ab}}$
according to Damerrell (cf. \cite[Corollary 2.11]{BK}) and so we have the following result.
\begin{prop}\label{Damerell}
For every $n\geq 0$, the distribution defined by
\[
\tilde{E}^n(Q,U,\varXi)=\frac{\eta_{\infty}^{\deg q}}{\Omega_{\infty}^{\deg q^*+n}}E^n(Q,U,\varXi)
\]
belongs to $\mathcal{D}_{K^{\mathrm{ab}}}(\mathbf{G}_n)$.
\end{prop}

\begin{rem}\label{rescaling-periods-distri}
In general, given $\phi\in \mathcal{D}_\C(\mathbf{H}_n)$, for some $\mathbf{H}_n$ as in \Cref{The space of distributions}, we denote by $\tilde{\phi}$
 the distribution in $\mathcal{D}_\C(\mathbf{H_n})$
given by
\[
\tilde{\phi}(Q,U,\varXi)=
\frac{\eta_{\infty}^{\deg q}}
  {  \Omega_{\infty}^{\deg q^*+\,n}   }
 \phi(Q,U,\varXi). 
\]
\end{rem}

We finish this section proving  an identity relating the Eiesentein number $E_{t,r+1}(0,w,L)$ with  the numbers $E_{k,1}(0,w,L)$, $0\leq k \leq t+t$, which will be used later
in the proof of the congruence between Eisesntein-Kronecker numbers (cf. \Cref{power-series-(r+1)Ert+1}). To simplify the notation, we will omit the letter $L$.

\begin{prop}\label{polynomial-identity}
 There exists a (non-unique) polynomial $Q_{t,r}$ in $\Z[X_1,\dots, X_{t+r},Y_1,\cdots, Y_{t+r}]$
of degree $r+1$, isobaric of weight $r+t$; $X_{i}$ and $Y_i$ are assigned weight $i$,  such that
\[
(r+1)E_{t,r+1}(0,w)=Q_{t,r}(E_{0,1}(0,w),\cdots, E_{t+r,1}(0,w),
E_{0,1}(0,0),\dots, E_{t+r,1}(0,0)),
\]
where 
\[
Q_{t,r}=b_{t,r}(-X_1)^rX_{t}+(\text{terms in which $X_1$ appears of degree $<r$})
\]
and 
\[
b_{t,r}=\begin{cases}
1,  &\text{if $t=0$,}\\
(r+1),& \text{if $t>0$}. 
\end{cases}
\]
\end{prop}
\begin{proof}
By the functional equation \eqref{symmetry-E-k-l},  we can equivalently show that there exists a (non-unique) polynomial $P_{k,l}$ in $\Z[X_1,\dots, X_{k+l},Y_1,\cdots, Y_{k+l}]$
of degree $k+1$, isobaric of weight $k+l$; $X_{i}$ and $Y_i$ are assigned weight $i$,  such that
\[
(k+1)E_{k,l}(z,0)=P_{k,l}(E_{0,1}(z,0),\cdots, E_{0,k+l}(z,0),
E_{0,1}(0,0),\dots, E_{0,k+l}(0,0)).
\]
Furthermore,
\[
P_{k,l}=a_{k,l}(- X_1)^{k}X_l+ (\text{terms in which $X_1$ appears of degree $<k$}),
\]
where 
\[
a_{k,l}=\begin{cases}
1,  &\text{if $l=1$,}\\
(k+1),& \text{if $l>1$}. 
\end{cases}
\]
In order to do this, we will tweak the argument of \cite[Proposition 9]{CS}. First, observe that it is enough to prove this for $l=1$, since for $l>1$ we simply take the derivative   $\partial_{-z}^{l-1}$ of
\[
(k+1)E_{k,1}(z,0)=P_{k,1}(E_{0,1}(z,0),\cdots, E_{0,k+1}(z,0),
E_{0,1}(0,0),\dots, E_{0,k+1}(0,0)).
\]
We will prove this identity by induction on $k$. The result is trivial for $k=0$. Suppose the result is true for all $0\leq j \leq k$ ( and all $l>0$).  From the proof of Proposition 9 of \cite{CS} we have
\begin{align*}
(k+2)E_{k+1,1}(z,0)=-(k+1)&E_{0,1}(z,0)E_{k,1}(z,0)+(k+1)E_{k,2}(z,0)-E_{0,k+2}(0,0)+\\
& \sum_{i=0}^{k-1}  {k+1 \choose i+1}(i+1)E_{i,1}(z,0)E_{0,k-i+1}(0,0),
\end{align*}
that is,
\[
P_{k+1,1}=-X_1P_{k,1}+P_{k,2}-Y_{k+2}+\sum_{i=0}^{k-1}{k+1 \choose i+1}P_{i,1}Y_{k-i+1},
\]
which clearly satisfies all the required conditions.
\end{proof}

\subsection{$\p$-adic periods}\label{sec-p-adic-periods}
Let $E(\C)\simeq \C/L$ be the elliptic curved introduced in \eqref{Weierstrass model-E}, with $L=\Omega_{\infty} \mathbf{a}$.
Let $p>3$ be a rational prime  which splits in $K/\Q$ as $\p\overline{\p}$
 and such that 
  \begin{equation}\label{condition for p-integral}
\text{ $\mathfrak{p}$  is relatively prime to  $\f_1$ and 
$\mathbf{a}$ is prime to $p$.}
 \end{equation}  
 Since $\f_1$ is the conductor  of the grossencharacter $\psi_{E/K(\f_1)}$, then the  fact that  $\p$ is prime to $\f_1$ implies that \eqref{Weierstrass model-E} has good reduction at all the primes above $\mathfrak{p}$.
 
 We further assume that the Weierstrass model of $E$ is over $\mathcal{O}_{K(\f_1)}$. 
Let $\mathfrak{P}$ be a prime ideal of $K(\f_1)$ above $\mathfrak{p}$ and let $\mathcal{O}_{\mathfrak{P}}$ be the completion of $\mathcal{O}_{K(\f_1)}$ at $\mathfrak{P}$. Let $t=-2x/y=-2 \wp(z,L)/\wp'(z,L)$ be the parameter of the formal group $\widehat{E}$ associated to $E$ (cf. \cite{dS} Chapter II, 1.10). Since $p$ splits in $K$ then $\widehat{E}$ is  furthermore a  Lubin-Tate formal group, i.e., of \textit{height} 1. Let $l_E(t)$ be the formal logarithm associated to $E$.

The formal groups $\widehat{E}$ and $\mathbb{G}_m$ are isomorphic over $W$, the ring of integers of the completion of the maximal unramified extension of $K(\f_1)_{\mathfrak{P}}$, that is
\[
\iota: \widehat{E}\to \mathbb{G}_m,
\]
where $\iota$ is given by a power series
$
\iota(X)=\eta_{\mathfrak{p}}X+\cdots\in W[\![X]\!]
$
with $\eta_{\mathfrak{p}}\in W^{\times}$. More specifically, 
$\iota(X)=\mathrm{exp}(\eta_{\mathfrak{p}\,}l_E(X))-1$
with inverse given by
\[
\iota^{-1}(X)=l_E^{-1}(\eta_{\mathfrak{p}}^{-1}\log(X+1)),
\]
where $\log$ is the usual logarithm.

Fix an embedding of $\overline{K}$ into $\C_p$. Denote by $\widehat{\mathcal{O}}$ the ring of integers of $\C_p$.
Let
$\Omega_{\mathfrak{p}}=\Omega_{\p,\mathbf{a}}, \Omega_{\p}^*=\Omega_{\p,\mathbf{a}}^*\in \widehat{\mathcal{O}}$ such that
\[
\frac{1}{\Omega_{\p}}=
-2i \mathrm{Vol}(\mathbf{a})\eta_{\p}
\quad
\text{and} 
\quad
\Omega_{\p}^*=\frac{1}{\eta_\p}=-2i\mathrm{Vol}(\mathbf{a})\,\Omega_\p.
\]
Then $\overline{K}(\eta_{\infty})$ is isomorphic to $\overline{K}(\eta_{\mathfrak{p}})$ (cf. \cite[ III \S 2]{CS}).

Assume that $1\in \mathbf{a}$,  so that $\tilde{1}=(-2i\mathrm{Vol}(\mathbf{a}))^{-1}\in \mathbf{a}^*$,
then for $\theta \in \p^{-\infty}\mathbf{a}/\mathbf{a}$ and $\tau \in \mathfrak{p}^{-\infty}\mathbf{a}^*/\mathbf{a}^*$ we have that
$
e\left(\theta \,|\, \tilde{1} \right)
$ and
$
e(1|\tau)$
are both $p^{\infty}$th roots of 1. Let $t(\theta)$ be the $\mathfrak{p}^{\infty}$-division point of $\widehat{E}$ such that
$
\iota\big(\,t(\theta)\,\big)\,=\,e\left(\theta \,|\, \tilde{1} \right)-1,
$
that is,
\begin{equation}\label{p-infty-division point formal}
t(\theta)=\iota^{-1} \left( \,
e\left(\theta \,|\, \tilde{1} \right)-1
 \,   \right).
\end{equation}

\subsection{$\p$-adic expansions  of Eisenstein-Kronecker series}
\label{$p$-adic expansions of $E_{k,l}$}
Let $\mathbf{a}$, $\Omega_{\infty}$, $\eta_{\infty}$ and $\p$ as above.
In this section we will deduce the $\p$-integrality of 
$E_1^0(u,v,\mathbf{a})/\Omega_{\infty}$  
via $\p$-adic expansions of
 Eisenstein-Kronecker numbers (cf. \Cref{p-integrality-Et-1-uv}).  
We begin by finding $\p$-adic expansions of
\[
\frac{\eta_{\infty}^t}{\Omega_\infty}E_{1}^t(0,w, \mathbf{a})
\quad \text{ for $t\geq 0$.}
\]

Let $\pi\in \mathcal{O}$ be a generator of $\mathfrak{p}^h$, where $h$ is the class number of $K$.   Let  $\widehat{\mathcal{O}}$ be the ring of integers of $\C_p$.
 Throughout this section we assume that  $1\in  \mathbf{a}$.

\begin{prop}\label{congruence-Colmez}
Given $w\in K- \p^{-\infty} \mathbf{a}^*$,  there exist power series $G_t(X)=G_{t,w }(X)\in \widehat{\mathcal{O}}[\![X]\!]$ for all $t\geq 0$, such that
\[
\frac{1}{\Omega_{\infty}}E_{1}^0(0,w+\tau, \mathbf{a})
=
-\frac{1}{\Omega_{\p}}\,\overline{w}  
 +G_{0}\big(\, e(1|\tau)-1\,\big)
\]
and
\[
\frac{  \eta_{\infty}^t  }
      {    \Omega_{\infty}      }
        E_{1}^t(0,w+\tau, \mathbf{a})
        =
G_{t}\big(\,e(1|\tau)-1\,\big) \quad (t>0)
\]
for all 
 $\tau \in \mathfrak{p}^{-\infty} \mathbf{a}^*/ \mathbf{a}^*$.
\end{prop}
\begin{proof}
We will start by showing that
for every $\alpha\in K- \mathfrak{p}^{-\infty} \mathbf{a}$ and every $t\geq 0$, there exist power series $H_{t,\alpha}(X)\in \widehat{\mathcal{O}}[\![X]\!]$ such that
for every $\theta \in \mathfrak{p}^{-\infty} \mathbf{a}/ \mathbf{a}$
\begin{equation}\label{power series E01}
\frac{1}{\Omega_{\infty}}
E_{1}^0(\alpha +\theta,0, \mathbf{a})
=- \eta_{\p}\,\overline{\alpha} +H_{0,\alpha }\big(\,e(\theta\,|\,\tilde{1})-1\,\big)
\end{equation}
and for $t>0$
\begin{equation}\label{power series E0t1}
\frac{1}{\Omega_{\infty}^{t+1}}
E_{t+1}^0(\alpha +\theta,0, \mathbf{a})
= H_{t,\alpha }\big(\,e(\theta\,|\,\tilde{1})-1\,\big).
\end{equation}

According to \cite[Prop. 10]{CS} there exists a power series
\[
G_{1}(\alpha ,t)=-\eta_{\p}\,\overline{\alpha }+\sum_{n=0}^{\infty}b_n(\alpha )t^n,
\]
where the $b_n(\alpha )$ are in $\widehat{\mathcal{O}}$, such that
for all $\theta \in \mathfrak{p}^{-\infty} \mathbf{a}/ \mathbf{a}$
\[
\frac{1}
      {    \Omega_{\infty}  }
      E_{1}^0(\alpha +\theta,0, \mathbf{a})
=E_{0,1}(\Omega_{\infty}(\alpha+\theta ),0 , L)
=
G_{1}\big(\,\alpha \,, t(\theta)\,\big),
\]
where the $t(\theta)$ is given by   \eqref{p-infty-division point formal}. Thus, if we let $H_{0,\alpha }(X):=
\sum_{n\geq 0} b_n(\alpha )(\iota^{-1}(X))^n$, \eqref{power series E01} follows.

As for \eqref{power series E0t1} ( i.e., the case $t>0$), it follows from \cite[Prop. 11]{CS} and its proof. 

From   \eqref{power series E01}  and \eqref{power series E0t1},  the 
main claim follows now from
\begin{align*}
\frac{\eta_{\infty}^t}{\Omega_{\infty}}
E_{1}^t(0,w+\tau, \mathbf{a})
& =
\frac{1}{\Omega_{\infty}^{t+1}} 
E_{t+1}^0\big(\,-2i\mathrm{Vol}( \mathbf{a})(w+\tau),0, \mathbf{a}\,\big) ,   
\end{align*}
which is a consequence of \eqref{symmetry-E-k-l}, and also by noticing that
\[
e\big(\,-2i\mathrm{Vol}( \mathbf{a}) \tau \,|\,\tilde{1}\,\big)=e(-\tau|1)=e(1|\tau).
\]
\end{proof}

Let $u,v\in K$  such that
  $u\in K- \mathfrak{p}^{-\infty}  \mathbf{a} $ 
and $v\in K- \p^{-\infty}   \mathbf{a}^*$. We will now use the above result to obtain $\p$-adic expansions for 
\[
\frac{\eta_{\infty}^t}
     {\Omega_{\infty}}
     E_{1}^t(u,v, \mathbf{a})
\]
for $t\geq 0$. Let  $f=\overline{\pi}^{ \mathtt{M}}\pi^{ \mathtt{S}}d$ in  $\mathcal{O}- \{0\}$, with $ \mathtt{M}, \mathtt{S}\geq 0$ and $d$ relatively prime to $\p\overline{\p}$,  such that
$fu\in  \mathbf{a}$.  Then for all $t,r\geq 0$
\begin{equation}\label{expansion-Er,t}
E_{r+1}^t(u,v,  \mathbf{a}  )=
\frac{f^r}{ \overline{f}^{t+1}}
\sum_{\beta, \gamma,\tau}
\,
e(-fu|\beta+\gamma+\tau)
\cdot
E_{r+1}^t(0,w+\tau,   \mathbf{a}   ),
\end{equation}
where the indices in this sum run though
\begin{equation}\label{running-beta-gamma-tau}
 \beta\in \frac{\overline{d}^{-1} \mathbf{a}^*}{ \mathbf{a}^*}
 ,\quad
 \gamma\in \frac{\overline{\pi}^{- \mathtt{S} } \mathbf{a}^*}{ \mathbf{a}^*},
 \quad \text{and}\quad 
 \tau\in \frac{\pi^{-  \mathtt{M} } \mathbf{a}^*}{ \mathbf{a}^*},
\end{equation}
and
\begin{equation}\label{defin-w-dist}
w=\frac{v}{\overline{f}}+\beta+\gamma.
\end{equation}
Note that $w\in K- \p^{-\infty} \mathbf{a}^*$ since $v\in K- \p^{-\infty} \mathbf{a}^*$. 

\begin{lem}\label{power series-integrality-E}
Let $u,v\in K$  such that
  $u\in K- \mathfrak{p}^{-\infty}  \mathbf{a} $ 
and $v\in K- \p^{-\infty}   \mathbf{a}^*$, and $f$ as above. Then
for every $t\geq 0$
there exists a power series   $H_t(X)=H_{t,u,v}(X)\in \widehat{\mathcal{O}}[\![X]\!]$ such that
\[
\pi^{ \mathtt{M}t}
\cdot
\Gamma(r+1) 
\frac{\eta_{\infty}^t}{\Omega_{\infty}}
E_{1}^t(u,v,  \mathbf{a})=
\frac{1}{ \pi^{ \mathtt{M}}}
\sum_{\tau\in \pi^{- \mathtt{M}} \mathbf{a}^*/ \mathbf{a}^*}
e(-fu|\tau)
\cdot
H_{t}(e(1|\tau)-1).
\]

\end{lem}
\begin{proof}
For   $w$ in \eqref{defin-w-dist}, let  $y=v\overline{f}^{-1}+\gamma$ so that $w=y+\beta$. According  to   
\Cref{congruence-Colmez}, we can find
 power series $G_{t}=G_{t,w}\in \widehat{\mathcal{O}}[\![X]\!]$ such that
\[
\frac{     \eta_{\infty}^t  }
       {        \Omega_{\infty}     }
        E_{1}^t(0,w+\tau,  \mathbf{a}  )  
       =b_{t}+G_t(e(1|\tau)-1),
\]
for all $\tau \in \p^{-\infty}\mathbf{a}^*$, where
$
b_0= -\Omega_{\p}^{-1} 
\overline{y}
 $
 and $b_{t}=0$ for $t>0$. Note here that for $G_0(X)$ we are actually taking the translated power series $-\Omega_{\p}\overline{\beta}+G_{0,w}(X)\in\widehat{\mathcal{O}}[\![X]\!]$ instead. 

Since  
 $fu\not \in \overline{\pi}^{ \mathtt{M}} d\,  \mathbf{a}$, otherwise
  $\pi^{ \mathtt{S}} u\in  \mathbf{a}$ which leads to the contradiction $x\in \p^{-\infty} \mathbf{a}$, then
\begin{equation}\label{sum-zero}
\sum_{\beta, \tau }
e(-fu|\beta+\tau)=0.
\end{equation}
Thus we can take  for
$H_t(X)\in \widehat{\mathcal{O}}[\![X]\!]$  the power series
\[
H_t(X):=
\frac{1}{\overline{\pi}^{ \mathtt{S}} \overline{d}}
\,
\sum_{\beta, \gamma }
\,
e(-fu|\gamma+\beta)\cdot 
G_{t,w}(X)
\]
with indices $\beta$ and $\gamma$ as in \eqref{running-beta-gamma-tau}.

\end{proof}

The following result will help us prove the congruences between Eisenstein-Kronecker numbers. For an  $x\in K$, let $\mathtt{K}(x)$ denote the smallest nonnegative integer $\mathtt{M}$
such that $\pi^{\mathtt{M}}\overline{x}\in \mathcal{O}_{(\p)}$.

\begin{lem}\label{power-series-(r+1)Ert+1}
Let $w=v+z\in K- \p^{-\infty} \mathbf{a}^*$. Let $\mathtt{S}$ be an integer such that $\mathtt{S}\geq  \mathtt{K}(z)$. Then, there exists a power series $h(X)=h_{t,r,w}(X)\in \widehat{\mathcal{O}}[\![X]\!]$ such that
for any integer $ \mathtt{N} >0$ for which $ \mathtt{N}+ \mathtt{S}\geq  \mathtt{K}(v)$  and for all $\tau\in \p^{-\infty} \mathbf{a}^*$ we have that
\begin{align*}
(r+1)
  \pi^{( \mathtt{N}+ \mathtt{S})r}
   \cdot 
   \Gamma(r+1)
   \frac{ \eta_{\infty}^t  }
        {\Omega_{\infty}^{r+1}}
E_{r+1}^t(0,w+\tau,  \mathbf{a}  )
\end{align*}
equals
\begin{align}\label{power-series-h-aux}
(r+1)
\left(\pi^{ \mathtt{N}+ \mathtt{S}}\,
\Omega_{\p}^{-1} \,
 \overline{v} \right)^{r}
   \cdot 
   \frac{\eta_{\infty}^t }
       {\Omega_{\infty}}
       E_{1}^t(0,w+\tau, \mathbf{a})+b_{t,r}
+\pi^{  \mathtt{N} }h\big(\,e(1|\tau)-1\,\big),
\end{align}
where 
\[
\text{
$\displaystyle b_{0,r}=r\cdot 
\frac{    
      \left(\pi^{ \mathtt{N}+ \mathtt{S}}\,
\Omega_{\p}^{-1} 
\,
 \overline{v}  \right)^{r+1}}
 { \pi^{ \mathtt{N} }   }$
and $b_{t,r}=0$ for $t>0$.}
\]
\end{lem}
\begin{proof}
We start by  observing that
$\pi^{ \mathtt{N}+ \mathtt{S}}\,
\Omega_{\p}^{-1} 
 \, \overline{z}$ is in $\pi^{  \mathtt{N} }\widehat{\mathcal{O}}$. 
 Now, 
for $t>0$ the result follows from   \Cref{polynomial-identity}
and  \Cref{congruence-Colmez}; bearing in mind that the numbers $E_{k,1}(0,0)$, $k\geq0$, are $\p$-integral (cf., e.g., \cite[Corollary 2.17]{BK} ).
For $t=0$, we have to make one additional observation with regards to the term
\[
(-\pi^{ \mathtt{N}+\mathtt{S} }  E_{0,1}(0,w+\tau,   \mathbf{a}   ))^{r}\cdot 
E_{0,1}(0,w+\tau,  \mathbf{a}   ).
\]
According to  \Cref{congruence-Colmez}, we have
\begin{equation}\label{randowm-ident-E01}
-\pi^{ \mathtt{N}+\mathtt{S}}\cdot  \frac{1}{\Omega_{\infty}}
E_{0,1}(0,w+\tau,   \mathbf{a}  )
=
 \pi^{ \mathtt{N} +\mathtt{S} }
    \cdot
     \frac{ 1   }
            {   \Omega_{\p}} 
   \cdot 
   \overline{v+z}  
 -\pi^{ \mathtt{N} +\mathtt{S}  }
 G_{0}(e(1|\tau)-1).
\end{equation}
Letting  
$H_{0}(e(1|\tau)-1)= - \pi^{\mathtt{S}}\Omega_{\p}^{-1} \,\overline{z}  
 + \pi^{\mathtt{S}}  G_{0}(e(1|\tau)-1)\in \widehat{\mathcal{O}}[\![e(1|\tau)-1]\!]$ and
raising \eqref{randowm-ident-E01} to the $r+1$ power, $r>0$, we have
that 
\[
\left(\,-\pi^{ \mathtt{N}  +\mathtt{S}  }  
\cdot 
 \frac{1}{\Omega_{\infty}}E_{0,1}(0,w+\tau,  \mathbf{a}   )\,\right)^{r+1}\\
\]
is congruent to
\[
( \pi^{ \mathtt{N} +\mathtt{S} } \Omega_{\p}^{-1} \,\overline{v}  )^{r+1}
\,  -  \,
(r+1)\cdot ( \pi^{ \mathtt{N} +\mathtt{S} } \Omega_{\p}^{-1}  \,\overline{v}  )^{r}
\cdot 
\pi^{  \mathtt{N} }  \cdot  H_{0}(e(1|\tau)-1)
\]
modulo  $\pi^{ 2  \mathtt{N}  }\widehat{\mathcal{O}}[\![e(1|\tau)-1]\!]$.
Now  using \eqref{randowm-ident-E01} we substitute 
$\pi^{  \mathtt{N} }  H_{0}(e(1|\tau)-1)$ by 
$ \pi^{  \mathtt{N} +\mathtt{S}  } \, \Omega_{\p}^{-1} \,\overline{v}  
+\pi^{  \mathtt{N}+\mathtt{S} } E_{0,1}(0,w+\tau,   \mathbf{a}    )/\Omega_{\infty}$, so the above is congruent to
\[
-
r( \pi^{  \mathtt{N} +\mathtt{S} }\, \Omega_{\p}^{-1} \,\overline{v}  )^{r+1}
-
(r+1)( \pi^{  \mathtt{N} +\mathtt{S}  }\, \Omega_{\p}^{-1} \,\overline{v}  )^{r}
\cdot
\left(\,\pi^{  \mathtt{N} } 
\cdot
\frac{   1  }
       {  \Omega_{\infty} }
      E_{0,1}(0,w+\tau,   \mathbf{a}  )  
       \,\right)
\]
modulo $\pi^{2  \mathtt{N} }\widehat{\mathcal{O}}[\![e(1|\tau)-1]\!]$.
Multiplying by $-1$ and dividing by $\pi^{   \mathtt{N} }$ the result follows for $t=0$.
\end{proof}


\subsection{$\p$-adic properties of Eisenstein-Kronecker numbers}
\label{$p$-adic properties of $E_{k,l}$}

We can now deduce the following $\p$-integrality result.

\begin{prop}\label{p-integrality-Et-1-uv}
For any
  $u\in K- \mathfrak{p}^{-\infty}  \mathbf{a} $ 
and $v\in K- \p^{-\infty}   \mathbf{a}^*$
the number
\[
  \frac{1}{\Omega_{\infty}}E_{1}^0(u,v,  \mathbf{a})
  \quad \text{is $\mathfrak{p}$-integral.}
\]
\end{prop}
\begin{proof}
This follows immediately from \Cref{power series-integrality-E}
since 
 it is enough to note that
\begin{equation}\label{p-div-sum-tau-1}
\sum_{\tau} ( e(1| \tau) -1 )^i\cdot e(-fu|\tau) \in \pi^{ \mathtt{M}} \widehat{\mathcal{O}}
\end{equation}
for all $i\geq 0$.
\end{proof}

\begin{rem}\label{p-int-E-u-v-gen}
For general $u,v\in K$, we have that
\[
  \frac{\pi^{\mathtt{M}+\mathtt{N}}}{\Omega_{\infty}}E_{1}^0(u,v,  \mathbf{a})
  \quad \text{is $\p$-integral, }
\]
where $\mathtt{M}$ and $\mathtt{N}$ are integers such that $\pi^{\mathtt{M}}\overline{u}, \pi^\mathtt{N}\overline{v}\in \mathcal{O}_{(\p)}$.
Indeed, this follows from \eqref{expansion-Er,t} and the fact that for $\alpha\in (K- \p^{-\infty}\mathbf{a})\cup \mathbf{a}$,
\[
G_{1}(\alpha ,t)=
1_{\mathbf{a}}(\alpha)t^{-1}-\eta_{\p}\,\overline{\alpha }+\sum_{n=0}^{\infty}b_n(\alpha )t^n,
\]
where $1_{\mathbf{a}}$ is the characteristic function of $\mathbf{a}$ (cf. \cite[Prop. 11]{CS}).
\end{rem}

\begin{prop}\label{2-variable congruence-general}
Let $t,r\geq 0$ .  Let   
$u\in K- \mathfrak{p}^{-\infty} \mathbf{a}$
and  $v\in K- \p^{-\infty} \mathbf{a}^*$.
Then, for all $ \mathtt{M}\geq  \mathtt{K}(u)$ and $ \mathtt{N}\geq  \mathtt{K}(v)$,
\begin{align*}
(t+1)(r+1)\pi^{ \mathtt{N}r+ \mathtt{M}t}\cdot 
 \Gamma(r+1)
\frac{   \eta_{\infty}^t }
      {    \Omega_{\infty}^{r+1}                  }
      E_{r+1}^t(u,v,   \mathbf{a}    ) 
\end{align*}
is congruent to
\begin{align*}
(t+1)(r+1)\cdot
\pi^{ \mathtt{M}t}\overline{u}^t
  \cdot   \pi^{ \mathtt{N}r} 
 \, \overline{v}^r
  \cdot
  \frac{  \eta_\p^t  }
    {  \Omega_\p^r  }
       \cdot 
      \frac{  1 }
           {    \Omega_{\infty}    }
         E_{1}^0 ( u, v, \mathbf{a}   )    
\end{align*}
modulo $    \pi^{\min\{ \mathtt{N}, \mathtt{M}\}}  \widehat{\mathcal{O}}  $.
\end{prop}
\begin{proof}

We will start by proving the congruence
\begin{align}\label{lemma-congruence-super-key} 
(r+1)
 \pi^{ \mathtt{N}r+ \mathtt{M}t}
 & \Gamma(r+1)
\frac{ \eta_{\infty}^t }
     {\Omega_{\infty}^{r+1}}
E_{r+1}^t(u,v,   \mathbf{a}  )
\equiv\\
& (r+1)
   (  \pi^{ \mathtt{N}} \,
\Omega_{\p}^{-1} 
    \, \overline{v}    )^r
   \pi^{ \mathtt{M}t}
   \cdot
   \frac{  \eta_{\infty}^t  }
        { \Omega_{\infty}}
   E_{t,1}( u, v,  \mathbf{a} )
    \pmod{   \pi^ {\mathtt{N}}  \widehat{\mathcal{O}}  }.\notag
\end{align}
In order to do this, let  $f=\overline{\pi}^{ \mathtt{M}}\pi^{ \mathtt{S}}d$ and $w$ be as in
\eqref{expansion-Er,t}. Since $w\in K- \p^{-\infty}\mathfrak{a}$,
we can apply \Cref{power-series-(r+1)Ert+1} to $w$ decomposed as $v'+z'$ for
$v'=v/\overline{f}$ and $z'=\gamma+\beta$,
which is possible since $\mathtt{S}\geq \mathtt{K}(z)$
and $\mathtt{N}+\mathtt{S}\geq \mathtt{K}(v')$.
Thus,  let $h_{t,r,w}(X)\in\widehat{\mathcal{O}}[\![X]\!]$ be the power series from 
\eqref{power-series-h-aux}. From \eqref{expansion-Er,t}, bearing in mind
 that $\sum b_{t,r} e(-fu|\beta+\tau)$ by \eqref{sum-zero}, we see that
\[
(r+1)\pi^{ \mathtt{N}r+ \mathtt{M}t}
    \frac{ E_{t,r+1}(u,v, \mathbf{a}) }
         {   \Omega_{\infty}^{t+r+1}   }
=
(r+1)
   (  \pi^{ \mathtt{N}} \,
 \Omega_{\p}^{-1} 
    \, \overline{v}    )^r
   \pi^{  \mathtt{M}  t}
   \frac{ E_{t,1}(u,v, \mathbf{a}) }
        { \Omega_{\infty}^{t+1}  }
          +  \pi^{\mathtt{N}}Y_{t,r},
\]
where
\[
Y_{t,r}:=
   \frac{1}{  \overline{f}  }
\sum_{\tau}
e(-fu|\tau)   
\cdot 
   j_{t,r}\big(\,e(1|\tau)-1\,\big)
\]
and $j_{t,r}(X)\in   \widehat{\mathcal{O}}[\![X]\!]$ is the power series
\[
j_{t,r}(X):=(r+1)\pi^{\mathtt{N}r+\mathtt{M}t} 
\cdot \frac{f^r}{\overline{f}^t}\cdot 
\sum_{\beta,\gamma}
e(-fu| \beta+\gamma)\cdot h_{t,r,w}(X).
\]
It now remains to invoke \eqref{p-div-sum-tau-1} to conclude  that $Y_{t,r}\in \widehat{\mathcal{O}}$. Thus \eqref{lemma-congruence-super-key} follows.

We can now prove the main congruence. We start by observing, by virtue of  \eqref{symmetry-E-k-l}, 
\begin{equation}\label{-techfun-eq-e-dd}
\frac{  \eta_{\infty}^t }
     {      \Omega_{\infty}    }
    E_{1}^t ( u,v ,   \mathbf{a} )
=
e(x|y)\cdot 
\frac{
   E_{t+1}^0 (  v^*,   u^*   ,   \mathbf{a})  
       }{     \Omega_{\infty}^{t+1}     },
\end{equation}
where 
$u^*=-u/(2i\mathrm{Vol}( \mathbf{a}))
\in
 K- \mathbf{a}^*$ 
 and $v^*=-2i\mathrm{Vol}( \mathbf{a})v \in K-  \mathbf{a}$. Applying the congruence \eqref{lemma-congruence-super-key},
this time to $E_{t+1}^0(v^*,u^*, \mathbf{a})$,  we obtain
\[
(t+1)\pi^{ \mathtt{M}t}
\frac{1}{\Omega_{\infty}^{t+1}}
E_{t+1}^0
(v^*,u^*, \mathbf{a})
\ \equiv \ (t+1)( \pi^{ \mathtt{M}}\eta_{\p}\overline{u}  )^t
\cdot 
\frac{1}{   \Omega_{\infty}   }
E_{1}^0(v^*,u^*, \mathbf{a})
\pmod{\pi^{ \mathtt{M}}\widehat{\mathcal{O}}}. 
\]
Multiplying both sides by $e(u|v)$ and using again \eqref{-techfun-eq-e-dd} we obtain
\[
(t+1)
 \pi^{ \mathtt{M}t}
\frac{  \eta_\infty ^t }
     {\Omega_{\infty} }
         E_{1}^t ( u, v, \mathbf{a} )
\ \equiv \ 
(t+1)( \pi^{ \mathtt{M}}  \eta_{\p}\overline{u}  )^t
\cdot 
\frac{1}{\Omega_{\infty}}
E_{1}^0( u,  v, \mathbf{a})
\pmod{\pi^{ \mathtt{M}}\widehat{\mathcal{O}}}.
\]
Putting the congruences together we obtain the result.
\end{proof}

\subsection{$\mathfrak{p}$-adic properties of $E^n$}
\label{padic properties of En}

We maintain the same notation as before and let $\Lambda=\mathbf{a}^n$ and $\varLambda=\Lambda\times \Lambda^*$. For nonnegative integers $\mathtt{M}$ and $\mathtt{N}$ and $\varXi=\Xi\times \Xi^*\in \mathcal{W}_n$, let 
$
\varXi_{\mathtt{M},\mathtt{N}}=\Xi_{\mathtt{M},\mathtt{N}}\times 
\Xi_{\mathtt{M},\mathtt{N}}^*,
$
where $\Xi_{\mathtt{M},\mathtt{N}}=\overline{\pi}^{\mathtt{M}}\pi^{-\mathtt{N}}\Xi$
and $\Xi_{\mathtt{M},\mathtt{N}}^*=\overline{\pi}^{\mathtt{N}}\pi^{-\mathtt{M}}\Xi^*$.
To simplify the notation further, we will use the conventions from 
\Cref{notation-dis-Gn} and \Cref{rescaling-periods-distri}.

  Applying \Cref{dist-gen-xi-upsilon}
\[
\tilde{E}^n(Q,U,\varLambda)=
\sum_{
\substack{
V\in K^{2n}/\varLambda_{\mathtt{M},\mathtt{N}}\\
V\equiv U \,\mathrm{mod}\, \varLambda
}
}
\tilde{E}^n_{U+\varLambda}(Q,V,\varLambda_{\mathtt{M},\mathtt{N}}).
\]
Note that for $c=\overline{\pi}^{\mathtt{M}}\pi^{-\mathtt{N}}$ we have, by \Cref{defin-space-dist} \eqref{homon-t},
\begin{equation}\label{E-Lambda-Mn}
\tilde{E}_{U+\Lambda}^n(Q,V,\varLambda_{\mathtt{M},\mathtt{N}})
=
\frac{e(-u|v^*-u^*)}{\overline{\pi}^{n(\mathtt{M}+\mathtt{N})}}
\frac{\overline{c}^{\deg q}}{c^{\deg q^*}}
\tilde{E}(Q,Vc^{-1},\varLambda).
\end{equation}

For $U\in K^{2n}$ and $Q=(q,q^*)\in H^2$, let $Q(\overline{U})=Q(\overline{u},\overline{u^*})$ denote the product $q(\overline{u})q^*(\overline{u^*})$.

\begin{lem}\label{Main-cong-Lambda-pq}
Let $U\in K^{2n}$ such that 
 $
 U+\p^{-\infty}\varLambda\subset
K^{\times 2n}.
 $
Then, for any $V\in U+\varLambda$,
\begin{equation}\label{real-integalityF_0-1-delt}
\text{$\displaystyle 
\tilde{E}^n_{U+\Lambda}(V,\varLambda_{\mathtt{M},\mathtt{N}})$ is $\mathfrak{p}$-integral.}
\end{equation}
Furthermore, for any $Q\in H^2$,
\begin{equation*}\label{Fk_N-M}
\tilde{E}_{U+\varLambda}^n(Q,V,\varLambda_{\mathtt{M},\mathtt{N}})
\equiv
Q(\,\overline{V}\,)\cdot
 \frac{\eta_{\\p}^{\deg q}}
  {  \Omega_{\p}^{\deg q^*}   }
 \tilde{E}^n_{U+\varLambda}(V,\varLambda_{\mathtt{M},\mathtt{N}})
\pmod{\pi^{\min\{ \mathtt{M}, \mathtt{N}\}-\delta}\widehat{\mathcal{O}}},
\end{equation*}
 where $\delta=\delta(Q,U,\varLambda)$ is a nonnegative integer that depends only on $(Q,U,\varLambda)$. More specifically, it depends   on the order of $U$ in $K^{2n}/\varLambda$ and on the coefficients of the polynomials of $Q$.
\end{lem}

\begin{proof}
The claim \eqref{real-integalityF_0-1-delt} is an immediate consequence of  
\Cref{p-integrality-Et-1-uv} and \eqref{E-Lambda-Mn}. Now, for $Q$ of the form $Q=(X^{\mathbf{t}},X^{\mathbf{r}})$, $\mathbf{t},\mathbf{r}\in \Z_{\geq 0}^n$,  the congruence is immediate from \Cref{2-variable congruence-general} and \eqref{defin-F-general}, producing a nonnegative integer $  \delta_{\mathbf{r},\mathbf{t}}:=\delta(Q,U,\varLambda)$.
For a general $Q=(q(X),q^*(X))\in H^2$, we obtain the result by taking  for
$\delta=\delta(Q,U,\varLambda)$ the maximum between $   \delta_{\mathbf{r},\mathbf{t}}$, $-v_{\mathfrak{p}}(q^*_{\mathbf{r}})$ and $-v_{\mathfrak{p}}(q_{\mathbf{t}})$ for all $\mathbf{r}$ and
 $\mathbf{t}$; here $v_{\p}$ denotes the valuation of $K$
 associated to the prime $\p$.

\end{proof}

Let $\mathbf{G}_{n,\p,\varLambda}$ denote the set of tuples $(Q,U,\varXi)$ such that
\begin{equation}\label{p-adic-tuple}
\begin{cases}
U + \varXi + \p^{-\infty}\varLambda   \subset     K^{\times 2n}\\
Q( \,\overline{U+\varXi}  \,)  \subset     \O_{(\p)}.
\end{cases}  
\end{equation}

Let $(Q,U,\varXi)\in  \mathbf{G}_{n,\p,\mathbf{a}}$ and
let $\varUpsilon=\Upsilon\times \Upsilon^*\in \mathcal{W}_n$ such that
\begin{equation}\label{condition-upsilon-M-N-1}
[\Upsilon:\Xi\cap \Upsilon]_{\mathcal{O}}=\p^{\mathtt{N}} 
\quad \text{and}\quad
[\Xi:\Xi\cap \Upsilon]_{\mathcal{O}}= \overline{\p}^{\mathtt{M}}.
\end{equation}
Then , for $X=U+\varXi$,
\[
\tilde{E}^n(Q,U,\varXi)
=\sum_{
\substack{
V\in K^{2n}/ \varUpsilon\\
V\equiv U\,\mathrm{mod}\, \varXi}
}
\tilde{E}^n_{X}(Q,V,\varUpsilon),
\]
and, by  \Cref{extra-prop-dist-13} applied to $\varTheta=\varXi_{\mathtt{M},\mathtt{N}}$,
for $V\in X$
\begin{equation}\label{sum-of-sum-upsilon-xi} 
\tilde{E}^n_{X}(Q,V,\varUpsilon)
=
\sum_{
\substack{
W\in K^{2n}/ \varXi_{ \mathtt{M},\mathtt{N} }\\
     W \equiv V\,\mathrm{mod}\,\varXi \cap \varUpsilon
           }
     }
\tilde{E}^n_{X}(Q,W,\varXi_{\mathtt{M},\mathtt{N}}).
\end{equation}

\begin{prop}\label{p-adic-En}
Let $(Q,U,\varXi)\in \mathbf{G}_{n,\p,\varLambda}$ and $\mathfrak{g}=[\Lambda:\Lambda\cap \Xi]_{\mathcal{O}}$.
Then
\[
\mathfrak{g}\cdot
\tilde{E}^n(Q,U,\varXi)
\quad \text{is $\p$-integral.}
\]
Furthermore, if $\varUpsilon\in \mathcal{W}_n$ satisfies 
\eqref{condition-upsilon-M-N-1}, 
then for any $V\in X$
\begin{equation*}\label{exp-Fk-upsilon-int}
\mathfrak{g}\cdot
\tilde{E}^n_{X}(Q,V,\varUpsilon)
\quad \text{is $\p$-integral}
\end{equation*}
and it is congruent to
\[
\mathfrak{g}
\cdot
\sum_{
\substack{
W\in K^{2n}/\varXi_{ \mathtt{M},\mathtt{N} }\\
     W \equiv V\,\mathrm{mod}\,\varXi \cap \varUpsilon
           }
     }
Q(\,\overline{W}\,)
 \frac{\eta_{\p}^{\deg q}}
  {  \Omega_{\p}^{\deg q^*}  }
\tilde{E}^n_{X}(W,\varXi_{\mathtt{M},\mathtt{N}})
\]
modulo $\pi^{\min\{\mathtt{M},\mathtt{N}\}-\delta}\widehat{\O}$, where $\delta=\delta(Q,U,\varXi,\varLambda)$ is a nonnegative integer independent of $\Upsilon$.

\end{prop}
\begin{rem}
If  $q(X),q^*(X)\in \mathcal{O}_{(\p)}[X]$, then the above $\delta$ would not depend on $Q$.
\end{rem}

\begin{proof}
According to \eqref{sum-of-sum-upsilon-xi} it is enough to prove that, for every $V\in X$,
\begin{equation}\label{F-k-upsilon}
\mathfrak{g}\cdot 
\tilde{E}^n_{X}(Q,V,\varXi_{\mathtt{M},\mathtt{N}})
\,\equiv \, 
\mathfrak{g}\cdot
Q(\,\overline{V}\,)
  \frac{\eta_{\p}^{\deg q}}
  {  \Omega_{\p}^{\deg q^*}    } 
\tilde{E}^n_{X}(V,\varXi_{\mathtt{M},\mathtt{N}})
\pmod{\pi^{\min\{ \mathtt{M}, \mathtt{N}\}-\delta}\widehat{\mathcal{O}}}
\end{equation}
 where $\delta=\delta(Q,U,\varXi,\varLambda)$ is a nonnegative integer independent of $V$ and $\varUpsilon$.

Let $[\Lambda:\Lambda\cap\Xi]_{\mathcal{O}}=\mathfrak{c}\overline{\p}^a$
and
 $[\Xi:\Lambda\cap\Xi]_{\mathcal{O}}=\mathfrak{d}\p^b$
for nonnegative integers $a$ and $b$ and  integral ideals $\mathfrak{c}$ and $\mathfrak{d}$ prime to $\overline{\p}$ and $\p$, respectively. 
Let $d,c\in \mathcal{O}$ such that $c\mathcal{O}=\mathfrak{c}^{ah}$
and
$d\mathcal{O}=\mathfrak{d}^h$, where $h$ is the class number of $K$. 
Note that
 \begin{equation}\label{simple contain}
 c\overline{\pi}^{a}\Lambda \subset \Xi
 \quad \text{and}\quad
 d\pi^{b}\,\Xi\subset \Lambda.
 \end{equation}
 Let  $\varTheta=\Theta\times \Theta^*$, where $\Theta=\Lambda_{a+\mathtt{M},b+\mathtt{N}}$. Then $ d\overline{\pi}^{a+\mathtt{M}}\pi^{-\mathtt{N}}\Xi
  \subset 
 \Theta
  \subset
   \overline{\pi}^{\mathtt{M}}   c^{-1} \pi^{-b-\mathtt{N}} \Xi$,
  which implies 
   \[
  d\overline{\pi}^{a+\mathtt{M}}\Xi
  \subset
     \Xi\cap \Theta
         \subset 
              \overline{\pi}^{\mathtt{M}}\Xi
   \quad \text{and}\quad
\overline{c}\overline{\pi}^{b+\mathtt{N}}\Xi^* 
         \subset
               \Xi^*  \cap  \Theta^*
               \subset \overline{\pi}^{\mathtt{N}} \Xi^* .
  \]
   
  Then  $\Theta$ satisfies the hypothesis of  \Cref{extra-prop-dist-13} (for $\varUpsilon=\varXi_{\mathtt{M},\mathtt{N}}$) and so
 \begin{equation}\label{F-ups-theta-dis}
  \tilde{E}^n_{X}(Q,V,\varXi_{\mathtt{M},\mathtt{N}})  =
  \sum_{
  \substack{
  W\in K^{2n}/ \varTheta\\
  W\equiv V\,\mathrm{mod}\,\varXi\cap  \varXi_{\mathtt{M},\mathtt{N}}
  }
  }
   \tilde{ E}^n_{X}(Q,W,\varTheta).
  \end{equation} 
  However, since $e(v-u|w^*-u^*)=1$, we have
\begin{equation}\label{suma-congruen-Xi-Lam}
  \tilde{ E}^n_{X}(Q,W,\varTheta)=
T\cdot
  \tilde{ E}^n_{V+\varLambda}(Q,W,\varTheta)
  \quad \text{where}\quad 
  T:= \frac{  [\Theta: \Lambda\cap \Theta]}
 { [ \Theta : \Xi\cap \Theta  ]  }.
\end{equation}

From the hypothesis \eqref{p-adic-tuple} we obtain first that
 $
 W
 +\p^{-\infty}\varLambda
 \subseteq
 K^{\times 2n}$. 
Then we also obtain that  $ Q(\overline{W})\in \mathcal{O}_{(\p)}$ 
 and, since $W\equiv V\in \varXi\cap \varXi_{\mathtt{M},\mathtt{N}}=\overline{\pi}^{\mathtt{M}} \Xi\times \overline{\pi}^{\mathtt{N}}\Xi^*$,  the  congruence $  Q(\,\overline{W}\,)
  \equiv 
    Q(\,\overline{V}\,) \mod{{\pi}^{\min\{\mathtt{M},\mathtt{N}\}}\mathcal{O}_{(\p)}}.$
   Therefore, by \Cref{Main-cong-Lambda-pq}, 
  \begin{equation}\label{cong-F-0-u-theta-alpha}
     \tilde{ E}^n_{V+\varLambda}(Q,W,\varTheta)
   \,  \equiv  \,
      Q(\,\overline{V}\,)  
   \frac{\eta_{\p}^{\deg q}}
  {  \Omega_{\p}^{\deg q^*}    } 
  \tilde{E}^n_{V+\varLambda}(W,\varTheta)
  \pmod{\pi^{\min\{\mathtt{M},\mathtt{N}\}-\delta}\widehat{\O}},
  \end{equation}
where $\delta$ 
   depends only on $(Q,U,\varXi)$ and $\varLambda$, but not on $\mathtt{M}$ or $\mathtt{N}$.
   
   It remains to show that the factor $T \cdot\mathfrak{g}$ is $\p$-integral, because once we multiply it to \eqref{cong-F-0-u-theta-alpha}, then adding over all $W$, and  applying \eqref{F-ups-theta-dis} and \eqref{suma-congruen-Xi-Lam} with $Q=(1,1)$ to the right hand-side of the congruence, the congruence \eqref{F-k-upsilon} will follow. To this end, note that it will be enough to show that
   \begin{equation}\label{divisibilit-modules-im}
   \text{
  $[ \Theta : \Xi\cap \Theta ]_{\mathcal{O}}$ 
  \quad divides\quad $[\Theta:\Lambda\cap \Theta]_{\mathcal{O}}[\Lambda:\Lambda\cap \Xi]_{\mathcal{O}}$
  }
   \end{equation}
since $[\Theta:\Xi\cap \Theta]_{\mathcal{O}}$ is prime to $\overline{\p}$; indeed, by \eqref{simple contain} we have $c\overline{\pi}^{a+\mathtt{M}}\Lambda=c\overline{\pi}^{a}\Lambda\cap \Theta\subset \Xi\cap \Theta\subset \Theta=\overline{\pi}^{a+\mathtt{M}}\pi^{-(b+\mathtt{N})}\Lambda$.
That
 \eqref{divisibilit-modules-im} holds, follows at once from the identity
\[
[ \Theta : \Xi\cap \Theta ]_{\mathcal{O}}
      [ \Xi\cap \Theta :\varLambda\cap \Xi\cap \Theta]_{\mathcal{O}}
=
[ \Theta : \varLambda\cap \Theta ]_{\mathcal{O}}
[\varLambda\cap \Theta : \varLambda\cap \Xi\cap \Theta]_{\mathcal{O}}.
\]
\end{proof}

\subsection{$\p$-adic  properties of $D$}
\label{padic  properties of D}

Suppose $(Q,U,\varXi)\in \mathbf{G}_n$. Let
$\varUpsilon\in \mathcal{W}_n$  such that for nonnegative integers $\mathtt{M}$ and $\mathtt{N}$ we have
\begin{equation}\label{p-adic-condition-Upsilon}
[\Upsilon:\Xi\cap \Upsilon]_{\mathcal{O}}=\p^{\mathtt{N}} 
\quad \text{and}\quad
[\Xi:\Xi\cap \Upsilon]_{\mathcal{O}}= \overline{\p}^{\mathtt{M}}.
\end{equation}
For $\sigma\in \mathrm{GL}_n(K)$ assume that 
$(Q,U,\varXi)\cdot \sigma\in   \mathbf{G}_{n,\p,\varLambda}$, i.e.,
\[
(U+\varXi)\sigma+\p^{-\infty}\varLambda
\ \subset\
K^{\times 2n}
\quad \text{and}\quad
Q(\,\overline{U+\varXi}\,)\subset \O_{(\p)}.
\]
\begin{prop}\label{arithmetic-cong-of-Dk}
Let $(Q,U,\varXi)\in \mathbf{G}_{n}$ and $\sigma\in \mathrm{GL}_n(K)$ such that
 $(Q,U,\varXi)\cdot \sigma\in   \mathbf{G}_{n,\p,\varLambda}$ and let
$
\mathfrak{g}:=[\Lambda: \Xi]_{\mathcal{O}}
[\Xi\sigma:\Lambda \cap  \Xi\sigma]_{\mathcal{O}}.
$
Then for every $V\in X$ we have that
\begin{equation}\label{Dk-upsio}
\mathfrak{g}\cdot
\tilde{D}_{X}(\sigma,Q,V,\varUpsilon)
\end{equation}
is $\p$-integral and, furthermore, it is congruent to
\[
\mathfrak{g} 
\cdot
\sum_{
\substack{
W\in K^{2n}/\varXi_{ \mathtt{M},\mathtt{N} }\\
     W \equiv V\,\mathrm{mod}\,\varXi \cap \varUpsilon
           }
     }
Q(\,\overline{W}\,)\cdot
\frac{\eta_\p^{\deg q}}{\Omega_\p^{\deg q^*}}
\tilde{D}_{X}(\sigma,W,\varXi_{\mathtt{M},\mathtt{N}} )
\]
modulo $\pi^{  \min\{\mathtt{M},\mathtt{N}\}  -\delta}\widehat{\O}$, where 
$\delta=\delta(\sigma,Q,U,\varXi,\varLambda)$ is a nonnegative integer that depends only on $(\sigma,Q,U,\varXi)$ and $\varLambda$.
\end{prop}

\begin{proof}
Bearing in mind that $[\Upsilon:\Upsilon\cap \Xi]=[\Upsilon\sigma:\Upsilon\sigma\cap \Xi\sigma]$, we can express \eqref{Dk-upsio} as
\[
[\Lambda: \Xi]_{\mathcal{O}}
[\Xi\sigma:\Lambda \cap  \Xi\sigma]_{\mathcal{O}}\cdot\frac{\det(\sigma)}{[\Upsilon\sigma:\Upsilon\sigma\cap \Xi\sigma]}
\tilde{E}^n(\sigma^tQ,(U+V)\sigma,\varUpsilon\sigma).
\]
We now see that we can apply \Cref{p-adic-En} to $(\sigma^t Q,(U+V)\sigma,\varUpsilon\sigma)$, bearing in mind that 
\begin{align*}
&[\Lambda: \Xi]_{\mathcal{O}}
[\Xi\sigma:\Lambda \cap  \Xi\sigma]_{\mathcal{O}}\det(\sigma)\\
&=[\Lambda: \Xi]_{\mathcal{O}}
[\Xi:\Xi\sigma]_{\mathcal{O}}
[\Xi\sigma:\Lambda \cap  \Xi\sigma]_{\mathcal{O}}\\
&=
[\Lambda:\Lambda\cap \Xi\sigma]_{\mathcal{O}}.
\end{align*}
Finally, we  take $\delta(\sigma,Q,U,\varXi,\varLambda)$ to be $\delta(\sigma^tQ,U\sigma,\varXi\sigma,\Lambda)$. 
\end{proof}

\begin{rem}\label{D-is-measure-rem}
From \Cref{p-int-E-u-v-gen}, \eqref{hom-sigma-D} and \eqref{F-ups-theta-dis}, we can prove the following:
For any $\sigma\in \mathrm{GL}_n(K)$, $U\in K^{2n}$ and $\varXi\in \mathcal{W}_n$, there exists a nonnegative integer $\mathtt{S}=\mathtt{S}(\sigma,U,\varXi)$ such that
\[
\pi^{\mathtt{S}}
\cdot
\tilde{D}_{X}(\sigma,V,\varUpsilon)
\]
is $\p$-integral for any $V\in X=U+\varXi$ and $\Upsilon\in \mathcal{W}_n$ satisfying \eqref{p-adic-condition-Upsilon}, where  $\mathtt{M}(U)$ and $\mathtt{N}(U)$ are nonnegative integers such that 
$(U+\varXi)\overline{\pi}^{\mathtt{M}(U)}\pi^{-\mathtt{N}(U)}\subset \varLambda$
\end{rem}


\section{The smoothed cocycle  and its integrality properties}
\label{The smoothed cocycle and its integrality properties}

In this section, we introduce the smoothing operation for distributions and use it to recall the definition of the (smoothed) Eisenstein-Kronecker cocycle $\Phi_\l$. In Section \ref{Integrality properties of psi-lk}, we prove the integrality and congruence properties of $\Phi_\l$. As a corollary, we recover the integrality of the critical values of partial $L$-functions in \Cref{Integrality of the special values}.

\subsection{The smoothed distributions}
\label{The smoothed distributions}
Let $\l$ by a prime ideal of $K$.
 For a prime ideal $\mathfrak{p}$ of $K$ let $\O_{(\p)}$ denote the localization of $\O$ at $\p$, i.e., the ring of fractions $a,b\in \mathcal{O}$ with $b\notin \p$.
For a nonzero integral ideal $\mathfrak{a}$ of $K$ we define $\mathcal{O}_{(\aa)}=\cap _{\p\mid \mathfrak{a}}\mathcal{O}_{(\p)}$ and
\[
 \mathcal{O}\left[\frac{1}{\mathfrak{a}}\right]:=
\cap_{\p\nmid \mathfrak{a}}\mathcal{O}_{(\p)},
\]
i.e., this is the ring of fractions $a/b$ in $K$ such that $b\mathcal{O}$ divides a  nonnegative power of 
 $\mathfrak{a}$.
 
Let $W_{n,\l}$ be the set of $\mathcal{O}$-modules $\Xi$ in $W_n$ such that
$
 \Xi \otimes_{\mathcal{O}_{K}}\mathcal{O}_{(\l)}=\mathcal{O}_{(\l)}^n.$ 
 This is equivalent to assuming that  $[\Xi:\Xi\cap \mathcal{O}^n]_{\mathcal{O}}$ and $[\mathcal{O}^n:\mathcal{O}^n\cap \Xi]_{\mathcal{O}}$ are prime to $\l$.
For such a $\Xi$ we define the $\mathcal{O}$-module 
$
\Xi_{\l} :=\varXi\cap (\l\mathcal{O}_{(\l)}\times \mathcal{O}_{(\l)}^{n-1}),
$
i.e., $\Xi_{\l}=  \{(x_1,\dots, x_n)\in \Xi: x_1\in \l\mathcal{O}_{(\l)}  \}$. 
Note that 
$
[\Xi:\Xi_\l]_{\mathcal{O}}=\l $
and
if $\Upsilon\in W_{n,\l}$,  which implies $[\Xi: \Xi \cap \Upsilon ]_{\mathcal{O}} $ being prime to $\l$, then
$(\Xi\cap \Upsilon)_{\l}=\Xi_\l \cap \Upsilon_\l=\Xi_\l \cap \Upsilon.$
Furthermore, we have
\begin{equation}\label{l-properties}
\text{
$
[\Xi:\Upsilon]_{\mathcal{O}}=[\Xi_\l:\Upsilon_\l]_{\mathcal{O}}
$
and if $\Upsilon \subseteq \Xi$, then $\Xi_\l/\Upsilon_\l \xrightarrow{\sim}\Xi/\Upsilon$.
}
\end{equation}

Let $\Gamma_\l$ be the group of automorphims of $\l\mathcal{O}_{(\l)}\times \mathcal{O}_{(\l)}^{n-1}$, i.e.,
\begin{equation*}
\Gamma_{\mathfrak{l}}:
=
\Gamma_0(\l\mathcal{O}_{(\l)})
=\{A\in \mathrm{GL}_n(\O_{(\l)}): A\equiv
\begin{pmatrix}
\ast	& \ast		& \ast \\ 
0 & \ast	&\ast \\
\vdots					& \vdots	& \vdots \\
0	& \ast		& \ast
\end{pmatrix} \ \mathrm{mod}\, \mathfrak{l}\}.
\end{equation*}
For $\Xi\in W_{n,\l}$ and $A\in \Gamma_\l$, we have that $\Xi A\in W_{n,\l}$ 
 and also $(\Xi A)_{\l}=\Xi_\l A$. 

Let $\mathcal{K}_{n,\l}:=(\l\mathcal{O}_{(\l)}\times \mathcal{O}_{(\l)}^{n-1})\times K^{n}$.
For $\varXi=\Xi\times\varXi^*$ with $\Xi\in W_{n,\l}$, let 
\[
\varXi_\l=\varXi\cap\
\mathcal{K}_{n,\l}=\Xi_\l\times \Xi^*
\]
and $\mathcal{W}_{n,\l}:=\{\varXi_\l: \varXi=\Xi\times\Xi^*,\ \Xi\in W_{n,\l}\}$. We define 
\[
\mathbf{G}_{n}^{\l}=H^2\times \mathcal{K}_{n,\l} \times \mathcal{W}_{n,\l}
\]
with the right action of $A\in \Gamma_\l$ given by \eqref{Action-Gln-tuple} and scalar multiplication
by a $t\in K^{\times}$ relatively prime to $\l$ given by \eqref{scalar-mult-tuple}.
Moreover, if $\mathbf{H}_n\subseteq \mathbf{G}_n$ is as in \Cref{The space of distributions}, we let 
$\mathbf{H}_n^{\l}:=\{(Q,U,\varXi_\l)\in\mathbf{G}_{n}^{\l}\,:\,   (Q,U,\varXi)\in\mathbf{H}_n\}$.

\begin{defi}\label{smoothed-distrib}
We denote by
$\mathcal{D}_S(\mathbf{H}_{n}^{\l})$ the $K$-vector space  of maps
$
\phi: \mathbf{H}_{n}^{\l} \to S
$ that are $K$-linear in the $Q$ component and
such that for every $(Q,U,\varXi)\in \mathbf{H}_{n}^{\l}$ the following holds
\begin{enumerate}[(i)]
\item $ \phi(Q,U+V,\varXi_\l)=e(u|v^*)\phi(Q,U,\varXi_\l)$ for all $V=(v,v^*) \in \varXi_{\l}$,
\item \label{homgeneity-phi-2} $\phi(Q,U,\varXi_\l)=t^n\phi\big(\,(Q,U,\varXi_\l)\cdot t\big)$ for all $t\in K^{\times}$ relatively prime to $\l$, and
\item \label{distri-prop-iii-smooth} for any $\varUpsilon_\l=\Upsilon_\l\times \Upsilon^*\in \mathcal{W}_{n,\l}$ we have
\[
\phi(Q,U,\Xi_\l)=\frac{1}{[\Upsilon_\l:\Xi_\l\cap \Upsilon_\l]}
\sum_{V\in  \varXi_\l/ \varXi_\l\cap \varUpsilon_{\l}}
e(-u|v^*)\phi(Q,U+V,\varUpsilon_\l).
\]
\end{enumerate}
If $A\in\Gamma_\l$  and $\phi\in \mathcal{D}_S(\mathbf{H}_{n}^{\l})$ 
we define the distribution $A\phi\in\mathcal{D}((\mathbf{H}_{n}^{\l}A,S)$ by
\begin{equation}\label{action-gamma-l}
A
\phi(Q,U,\varXi_\l)=
\det(A)\phi\big((Q,U,\varXi_\l)\cdot A \big).
\end{equation}
In particular, if $\mathbf{H}_{n}^{\l}$ is invariant under  $\Gamma_\l$, then this defines a left action
on   $\mathcal{D}_S(\mathbf{H}_{n}^{\l})$.
\end{defi}

\begin{rem}
Recalling the notation from \Cref{notation-dis-Gn}, i.e.,
\[
\phi_{U+\varXi_\l}(Q,V,\varUpsilon_\l)
:=
\frac{e(-u|v^*-u^*)}{[\Upsilon_\l:\Xi_\l\cap \Upsilon_\l]}
\phi(Q,V,\varUpsilon_\l).
\]
Note that by \eqref{l-properties} we  have $[\Upsilon_\l:\Xi_\l\cap \Upsilon_\l]=[\Upsilon:\Xi\cap \Upsilon]=[\Upsilon^\l:\Xi^\l\cap \Upsilon^\l]$.
Also, note that if $V\in U+\varXi_\l$, the value $\phi_{U+\varXi_\l}(Q,V,\varUpsilon)$ is independent of the class
of $V$ in $\mathcal{K}_{n,\l}/\varUpsilon_\l$.
Furthermore, when $Q=(1,1)$, we remove the letter from the notation and simply write $\phi(U,\varXi_\l)$, $\phi_{U+\varXi_\l}(V,\varUpsilon_\l)$, etc.
\end{rem}

With this notation, we can rewrite \eqref{distri-prop-iii-smooth} as
\[
\phi(Q,U,\Xi_\l)
=
\sum_{
\substack{
V\in \mathcal{K}_{n,\l}/  \Upsilon_{\l} \\
V\equiv U \,\mathrm{mod}\, \varXi_\l
}
}
\phi_{U+ \varXi_\l}(Q,V,\varUpsilon_\l).
\]
and, just as in \Cref{extra-prop-dist-13}, we can deduce the following identity.

\begin{lem}\label{lemma-decomp-phi-Xi-Upsilon}
Let $\varTheta=\Theta\times \Theta^*$
as in \Cref{extra-prop-dist-13}  with $\Theta\in W_{n,\l}$ . 
Then for any
$V\in U+\varXi_{\l}$,
\begin{equation*}\label{phi-Xi-upsioon}
\phi_{U+ \varXi_\l}(Q,V,\varUpsilon_\l)
=
\sum_{
\substack{
W\in \mathcal{K}_{n,\l}/\varTheta_\l \\
W\equiv V\,\mathrm{mod}\, \varXi_\l\cap \varUpsilon_{\l}
}
}
\phi_{U+ \varXi_\l}(Q,W,\varTheta_\l).
\end{equation*}
\end{lem}

\begin{proof}
The same proof as in \Cref{extra-prop-dist-13} works, applying
\eqref{l-properties} to the bijection \eqref{bij-map-dist-123-23}.
\end{proof}

We finish this section by defining the smoothing operation of a distribution.

\begin{defi}\label{smoothing-op}
We define the smoothing operator $\mathcal{D}_S(\mathbf{H}_{n})\to  \mathcal{D}_S(\mathbf{H}_{n}^{\l}) $ : $\phi\to \phi_\l$
where $\phi_\l$ is given  by
\[
\phi_{\l}(Q,U,\varXi_\l):=
\phi(Q,U,\varXi^{\l})
-\phi(Q,U,\varXi)
\]
for $(Q,U,\varXi_\l)\in \mathbf{H}_{n}^{\l}$, where $\varXi^{\l}:=\Xi_\l\times (\Xi_\l)^*$.
\end{defi}

From \Cref{defin-space-dist}  \eqref{prop-dist-main-1}, we can express $\phi_{\l}$
in terms of $\phi$ and $\varXi^\l$ follows
\begin{equation}\label{smooth-dedek-decomp-F}
\phi_{\l}(Q,U,\varXi_\l)
=\sum_{
\substack{U_0\in \varXi/ \varXi_\l \\ U_0\notin \varXi_\l} }
-e(-u|u_0^*)\phi(Q,U+U_0,\varXi^\l).
\end{equation}

\subsection{The smoothed Eisenstein-Kronecker cocycle}
\label{The smoothed Eisenstein-Kronecker cocycle}

We define  the smoothed Dedekind sum $D_\l$ by $D_\l(\sigma)=D(\sigma)_\l$ for $\sigma\in M_n(\mathcal{O}_{(\l)})$. Note that, in particular, if $\sigma\in \Gamma_\l$ then according to the observations in \Cref{The smoothed distributions} we have that $(\Xi\sigma)_\l=\Xi_\l\sigma$ and so $D_\l(\sigma)=\sigma E_\l^n$.
Similarly, we define $D_{\l}(s)$, $\Psi_\l(s)$ and $\Phi_\l(s)$, respectively, the smoothing
of  $D(s)$, $\Psi(s)$ and $\Phi(s)$.

From \Cref{The Eisenstein cocycle} we have that $\Psi_\l(s), \Phi_\l(s):\Z[\Gamma_\l^n]\to \mathcal{D}_\C(\mathbf{G}_{n,k}^{\l})$ represent a cohomology class in $H^{n-1}\big(\Gamma_\l, \mathcal{D}_\C(\mathbf{G}_{n,k}^{\l})\big)$ for all $\mathrm{Re}(s)>1+\frac{k}{2}$. Furthermore, $\Phi_\l=\Phi_\l(0)$ represents a cohomology class in 
$H^{n-1}\big(\Gamma_\l,\mathcal{D}_S(\mathbf{H}_{n}^{\l})\big)$ 
where $\mathbf{H}_{n}^{\l}=\cup_{k\geq 0} \mathbf{G}_{n,k}^{\l}$ for $n=1,2$
and $\mathbf{H}_{n}^{\l}=\mathbf{G}_{n,0}^{\l}$ for $n\geq 3$.

In \cite{BCG2}, it is proven that the map $\a\mapsto D_\l(\sigma(1))$ in $C^{n-1}(\Gamma_\l,\mathcal{D}_\C(\mathbf{G}_{n}^{\l}))$
 satisfies the cocycle property, thus it  represents a cohomology class in 
$H^{n-1}\big(\Gamma_\l,\mathcal{D}_\C(\mathbf{G}_{n}^{\l})\big)$. Since, by \Cref{identity-colmez}, this map coincides with $\Phi_\l$, whenever $\Phi_\l(s)$ has analytic continuation at $s=0$, we will also denote it by $\Phi_\l$, i.e.,
$\Phi_\l(\mathfrak{A})=D_\l(\sigma(1))$.

\begin{cor}\label{l-smoohted-in-terms-of-smoothed-Dedekind sums}
For all $\mathrm{Re}(s)>1+\frac{k}{2}$ we have that $\Psi_\l(s)=\Phi_{\l}(s)$
on  $\mathbf{G}_{n,k}^\l$. Furthermore, this implies
$
\Psi_{\l}  =\Phi_{\l},
$
whenever these cocycles have analytic continuation at $s=0$.
\end{cor}
\begin{proof}

 Fix $\mathfrak{A}=(A_1,\dots, A_n)\in \Gamma_{\mathfrak{l}}^n$ and let
  $(Q,U,\varXi)\in \mathbf{G}_{n,k}^\l$.
In order to do this, let $d=(d_1,\dots,d_n)$ be an $n$-tuple of integers with $1\leq d_j\leq n$, and let $\sigma(d)$
 denote the $n\times n$ matrix whose $i$th column is the $d_i$th column of $A_i$. Consider the space $A(d)\subset \C^n$ generated by all columns $A_{ij}$ with $j< d_i$, and let $A(d)^\perp$ be the orthogonal complement. Then define 
\begin{equation*}
\label{X(d)}
X(d)=A(d)^\perp\backslash\bigcup^n_{i=1} A^\perp_{id_i},
\end{equation*}
so that if $X(d)$ is nonempty, then it is a subspace of $\C^n$ with a finite number of codimension one subspaces removed. Let $D=D(A_1,\dots,A_n)=\{d:X(d)\neq \varnothing \}.$ Then by construction of $X(d)$, we can associate with $(A_1,\dots,A_n)$ a finite decomposition of $\C^n$ into a {\em disjoint} union
\begin{equation*}
\C^n\backslash\{0\}=\bigcup_{d\in D} X(d).
\end{equation*}

Therefore, the cocycle can be rearranged as
\begin{equation} \label{decom}
\Psi(s,\mathfrak{A},Q,U,\varXi)=\sum_{d\in D}\det(\sigma(d))
\sum_{ \mathbf{r} } q^*_{\mathbf{r}}(\sigma(d))
G_{k,\mathbf{r}}^s(\sigma(d),M,U,\varXi)
\end{equation}
where we are using the notation in \eqref{coeff-poly-P} and 
\begin{equation*}
G_{k,\mathbf{r}}^s(\sigma(d),M,U,\varXi)
   :=\sum_{x\in X(d)\cap (\Xi+u)} 
         e( x|u^*)\cdot 
             \prod_{j=1}^n 
                  \frac{r_j!}{\langle x,\sigma_j\rangle^{1+r_j}}
                  \cdot \frac{\overline{N_M(x)}^k}{|N_M(x)|^{2s}}.
                 \end{equation*}

 Then for $d\neq (1,\dots,1)$, we have $X(d)\cap( \Xi+u) = X(d)\cap (\Xi_{\mathfrak{l}}+u)$ for every  $u\in \Xi_{\l}\otimes_{\mathcal{O}}\mathcal{O}_{(\l)}=\l\mathcal{O}_{(\l)}\times \mathcal{O}_{(\l)}^{n-1}$.
Indeed, clearly 
$X(d)\cap( \Xi_{\mathfrak{l}}+u)\subset X(d)\cap (\Xi+u)$ since $\Xi_{\mathfrak{l}}+u \subset \Xi+u$. To show the reverse containment, let $i$ be the first entry of $d$ for which $d_i\neq 1$. For simplicity, assume $i=1$. Given 
$x+u \in X(d) \cap ( \Xi +u)$, we want to show that $x\in \l\mathcal{O}_{(\l)}\times \mathcal{O}_{(\l)}^{n-1}$. Observe that $x\in \Xi\subset \mathcal{O}_{(\l)}^{n}$, and
so $xA_1\in \mathcal{O}_{(\l)}^{n}$. By our assumption on $A_1$, we have that the first coordinate of $(x+u)A_1$ is zero, therefore $(x+u)A_1\in \l\mathcal{O}_{(\l)}\times \mathcal{O}_{(\l)}^{n-1}$. Since $uA_1\in \l\mathcal{O}_{(\l)}\times \mathcal{O}_{(\l)}^{n-1}$, then $xA_1\in \l\mathcal{O}_{(\l)}\times \mathcal{O}_{(\l)}^{n-1}$ and so
$x\in \l\mathcal{O}_{(\l)}\times \mathcal{O}_{(\l)}^{n-1}$ since $A_1$ is an automorphism of $\l\mathcal{O}_{(\l)}\times \mathcal{O}_{(\l)}^{n-1}$.

From the above observation we  have
\[
G_{k,\mathbf{r}}^s(\sigma(d),M,U,\varXi^{\mathfrak{l}})= 
G_{k,\mathbf{r}}^s(\sigma(d),M,U,\varXi) 
\]
for all $d\neq (1,\dots, 1)\in D$. Therefore, by
 \eqref{decom},  we have that 
$\Psi_\l(s,\mathfrak{A},Q,U,\varXi) $ is equal to
\[
\det(\sigma(1))\sum_{\mathbf{r}} 
       q^*_{\mathbf{r}}(\sigma(1))\left\lbrace 
G_{k,\mathbf{r}}^s(\sigma(1),M,U,\varXi^{\mathfrak{l}})- 
G_{k,\mathbf{r}}^s(\sigma(1),M,U,\varXi)      
       \right\rbrace .
\]
But
\[
\det(\sigma(1))\sum_{\mathbf{r}} 
       q^*_{\mathbf{r}}(\sigma(1))
G_{k,\mathbf{r}}^s(\sigma(1),M,U,\varXi^{\mathfrak{l}})
=D(s,\sigma(1),Q,U,\varXi_{\mathfrak{l}})
\]
and
\[
\det(\sigma(1))\sum_{\mathbf{r}} 
       q^*_{\mathbf{r}}(\sigma(1))
G_{k,\mathbf{r}}^s(\sigma(1),M,U,\varXi)
=D(s,\sigma(1),Q,U,\varXi).
\]
Thus the result follows.
\end{proof}

From  \Cref{l-smoohted-in-terms-of-smoothed-Dedekind sums}
and \eqref{smooth-dedek-decomp-F}
we obtain 
\begin{equation*}
\Psi_{\l} (s, \mathfrak{A},Q,U,\varXi_\l)=
 \sum_{
      \substack{
          U_0 \in \varXi /\varXi_\l \\  U_0 \not\in\varXi_\l
          }  }
-e(-u|u_0^*)D(s,\sigma(1),Q,U+U_0,\varXi^\l)
\end{equation*}
and
\begin{equation}\label{final decompi-del cocycle psi}
\Phi_{\l} ( \mathfrak{A},Q,U,\varXi_\l)=
 \sum_{
      \substack{
          U_0 \in \varXi /\varXi_\l \\  U_0 \not\in \varXi_\l
          }  }
-e(-u|u_0^*)D(\sigma(1),Q,U+U_0,\varXi^\l).
\end{equation}

\subsection{$\p$-adic properties of $\Phi_{\l}$ }
\label{Integrality properties of psi-lk}

Let $\f_1$, $\p$, $\Omega_{\infty}$, $\eta_{\infty}$, $\eta_{\p}$ and  $\Omega_{\p}$ be as in \Cref{Eisenstein-Kronecker numbers and properties}
for $\mathbf{a}=\l$. The algebraicity of $\Phi_{\l}$ will now be an easy consequence of \Cref{final decompi-del cocycle psi} and \Cref{Damerell}.

\begin{prop}\label{algeb-smoothed-psi-l-1}
For $\mathfrak{A}\in \Gamma_{\l}^n$ and 
$(Q,U,\varXi_\l)\in \mathbf{G}_{n}^{\l}$  the map
\begin{equation*}\label{cyc-alg-b}
\tilde{\Phi}_{\l}(\mathfrak{A},Q,U,\varXi_\l)
:=
\frac{\eta_{\infty}^{\deg q}}{\Omega_{\infty}^{\deg q^*+\,n}}
\Phi_{\l}(\mathfrak{A},Q,U,\varXi_\l),
\end{equation*}
 defines a cohomology class in  
 $H^{n-1}\big(\,\Gamma_{\mathfrak{l}},\mathcal{D}_{K^{\mathrm{ab}}}(\mathbf{G}_{n}^{\l})\,\big)$.
\end{prop}

We now deal with the integrality of $\Phi_{\l}$. For $\p$ as above, we let $\mathfrak{A}\in \Gamma_\l$ and let $(Q,U,\varXi_\l)\in \mathbf{G}_{n}^{\l}$ be such that
\begin{equation}\label{p-adic-condition-tuple}
\left\{
\begin{aligned}
&\ 
\varXi_\l\otimes_{\mathcal{O}}\mathcal{O}_{K_{\overline{\p}}}=\mathcal{O}_{K_{\overline{\p}}}^{2n},\\
&\ \big(\,u^*+\p^{-\infty}(\Xi_\l)^*\,\big)\sigma^*\subset K^{\times n},\\
&\ 
Q(\,\overline{U+\varXi+ \varXi^\l}\,)\subset \O_{(\p)},
\end{aligned}
\right.
\end{equation}
 Note that the first condition is equivalent to $[\Lambda:\Xi\cap \Lambda]_{\mathcal{O}}$ and
 $[\Xi:\Xi\cap \Lambda]_{\mathcal{O}}$ being prime to $\p\overline{\p}$.
 
The following result is an immediate  consequence of \Cref{arithmetic-cong-of-Dk}.

\begin{thm}\label{congruence-psi}
For every  $(Q,U,\varXi_\l)\in \mathbf{G}_{n}^{\l}$ satisfying \eqref{p-adic-condition-tuple},  the element
\[
\tilde{\Phi}_{\l}(\mathfrak{A},Q,U,\varXi_\l)
\]
is $\mathfrak{p}$-integral. 
 Consider  $\Upsilon_\l\in W_{n,\l}$ satisfying \eqref{p-adic-condition-Upsilon}
  and let  $\varTheta:=\varXi_{\mathtt{M},\mathtt{N}}  $.
  Then, for any $V\in  U+\varXi_\l$
we have that
\begin{equation*}
\label{main-big-cong}
\tilde{\Phi}_{\l,U+\varXi_\l}(\mathfrak{A}, Q,V,\varUpsilon_\l)
\quad
\text{is $\p$-integral}
\end{equation*}
 and  congruent to
\begin{equation*}
 \label{main-big-cong-2}
 \sum_{
 \substack{
 W\in \mathcal{K}_{n,\l}/\varTheta_\l  \\
     W \equiv V\,\mathrm{mod}\,  \varXi_\l\cap \varUpsilon_\l
     }
     }
  Q(\,\overline{W}\,)
 \cdot
 \frac{\eta_{\p}^{\deg q}}{\Omega_{\p}^{\deg q^*}}
 \cdot
\tilde{ \Phi}_{\l,U+\varXi_\l}(\mathfrak{A}, W,\varTheta_\l) 
\end{equation*}
modulo $\pi^{\min\{ \mathtt{M}, \mathtt{N}\}-\delta }\widehat{\mathcal{O}}$, 
where 
$\delta=\delta(\sigma(1),Q,U,\Xi,\varLambda)$ is a nonnegative integer independent of
 $V$ and  $\Upsilon$.
\end{thm}

\begin{proof}
We observe that rescaling $\sigma=\sigma(1)$ by some $t\in K^{\times}$ prime to $\l$ if necessary, we may assume 
that 
$\Xi_{\l}\sigma \subset \Lambda
\quad (\text{or equivalently}\quad
\Lambda^*\subset \Xi_{\l}^*\sigma^*).$

The result will follow from
\Cref{final decompi-del cocycle psi} and
\Cref{arithmetic-cong-of-Dk} after making the following observations.

First, observe that $\varXi\cap \varTheta\subset \Xi\cap \varUpsilon$ 
induces (by \eqref{l-properties}) the canonical isomorphims of $ \varXi\cap \varTheta\, /\, \varXi_\l\cap \varTheta_\l$ with each of the quotients $\varXi/\varXi_\l$, $\varUpsilon/\varUpsilon_\l$
and $\varTheta/\varTheta_\l$. Then we can take for representatives of the sum in \Cref{final decompi-del cocycle psi} elements 
$U_0=(u_0,u_0^*)\in  \varXi\cap \varTheta\, /\, \varXi_\l\cap \varTheta_\l$ with $U_0\notin \varXi_\l\cap \varTheta_\l$ and with $u_0^*=0$. 

This implies that $Q(\overline{U+U_0+\varXi+\varXi^\l})\equiv Q(\overline{U+\varXi+\varXi^\l})$ 
mod $\pi^{\min\{\mathtt{M},\mathtt{N}\}}\O_{(\p)}$. Furthermore, note that if $V$ runs through a full set of representatives of the classes $ \mathcal{K}_{n,\l}/\Upsilon_\l$ congruent to $U\,\mathrm{mod}\,\varXi_\l$, then 
$V+U_0$ runs through a full set of representatives of the classes $ K^{ 2n}/\Upsilon^\l$ congruent to $U+U_0\,\mathrm{mod}\,\varXi^\l$.
Similarly, if $W$ runs through a full set of representatives of the classes $ \mathcal{K}_{n,\l}/\Theta_\l$ congruent to $V\,\mathrm{mod}\,\varXi_\l\cap \varUpsilon_\l$, then 
$W+U_0$ runs through a full set of representatives of the classes $ K^{ 2n}/\varTheta^\l$ congruent to $V+U_0\,\mathrm{mod}\,\varXi^\l\cap \varUpsilon^\l$

From these observations, according to \Cref{arithmetic-cong-of-Dk},
it remains to verify that $(Q,U+U_0,\varXi^\l)\cdot\sigma\in \mathbf{G}_{n,\p,\varLambda}$. This is clear except perhaps for the condition
 $
 (U+U_0+\varXi_\l)\sigma+\p^{-\infty}\varLambda \subset K^{\times 2n}.
 $
  Since $u_0^*=0\in \subset \Xi_\l^*$, then from the initial observation, we have that
 \[
 (u^*+u_0^*+\Xi_\l^*)\sigma^*+\p^{-\infty}\Lambda^*
 \subset  (u^*+\p^{-\infty}\Xi_\l^*)\sigma^*
 \subset K^{\times n}.
 \]
Now let us show that
\begin{equation*}\label{non-van-usigma -Xi}
(u+u_0+\Xi_\l)\sigma+ \p^{-\infty}\Lambda\subset K^{\times n}.
\end{equation*}
Let us recall that $\mathfrak{A}=(A_{1},\dots, A_n)\in \Gamma_\l^n$ and
$\sigma=(\sigma_1,\dots, \sigma_n)$, where 
$\sigma_i$ denotes the first column of $A_{i}$.  Since $A_i\in \Gamma_\l$, then
$
u\sigma+\Xi_\l\sigma+\p^{-\infty}\Lambda\subset (\l\mathcal{O}_{(\l)})^n$. Therefore, it will be enough to prove that
\begin{equation}\label{w-A-i}
U_0\sigma_i\notin \l\mathcal{O}_{(\l)},
\quad i=1,\dots,n.
\end{equation}
Since $u_0\in \Xi- \Xi_\l$ and $A_i$ is an automorphism of $\l\mathcal{O}_{(\l)}\times\mathcal{O}_{(\l)}^{n-1}$, then $u_0A_i\in \mathcal{O}_{(\l)}^{n}- \l\mathcal{O}_{(\l)}\times\mathcal{O}_{(\l)}^{n-1}$, which implies \eqref{w-A-i}.

 \end{proof}

\subsection{The smoothed partial $L$-function }\label{The smoothed partial $L$-function}
Let $\f$ be an integral ideal and $\mathfrak{b}$ be a fractional ideal of $F$. Let $\c$ be  integral ideal of $F$
with prime norm to $K$ 
 such that 
\be\label{condition-c-1}
\text{$N_{F/K}(\mathfrak{c})$ is coprime to $\f$}.
\ee
Let  $\l$ denote the prime ideal $\overline{N_{F/K}(\mathfrak{c})}$  of $K$.
 Let $r\in F$ and $r^* \in F^*$ such that
\be\label{smoothed-condition-r-gamma}
\f \,r \subset \mathfrak{b}
\quad \mathrm{ and}\quad 
 \overline{\f} r^* \subset (\mathfrak{c}^{-1}\mathfrak{b})^*
\ee
and satisfying \eqref{condition-gamma},
 then we define the $\l$-smoothed partial $L$-function by
\[ 
    \L_{\f,\c}(\lambda,\mathfrak{b} ,r,r^*,s):
   =
    \L_\f( \lambda,\mathfrak{c}^{-1}\mathfrak{b},r,r^*,s)
    \,- \, 
 N_{F/\Q}(\mathfrak{c})   \L_\f(\lambda, \mathfrak{b},r,r^*,s).
\]

In particular, if $\mathfrak{a}$ is a fractional ideal of $F$ relatively prime to $\f$, then we define
\[
L_{\f,\c}(\aa,\chi,s):=L_\f(\aa \c,\chi,s)-\chi(\c)N_{F/\Q}(\c)^{1-s}L_{\f}(\aa,\chi,s),
\]
for which we have
\[
L_{\f,\c}(\aa,\chi,s)=\frac{\chi(\aa \c)}{N_{F/\Q}(\aa \c )^s} \L_{\f,\c}(\lambda,\aa^{-1}\f,1,0,s)
\]
and
\[
\sum_{a\in G_{\f}}L_{\f,\c}(\aa,\chi,s)=(1-\chi(\c)N_{F/\Q}(\c)^{1-s})L_{\f}(\chi,s).
\]

We can now parametrize the smoothed partial $L$-function by the smoothed cocycle just defined as follows. 
 Let $ \mathbf{G}_{n,\l}(\lambda,\mathfrak{b},r,r^*)$ be the set of tuples
$ (Q,U,\varXi)\in  \mathbf{G}_{n}(\lambda,\mathfrak{b},r,r^*)$
such that $ \varXi^*=\Xi^*\times \Xi\in \mathcal{W}_{n,\l}$ satisfies
$
\b^*=\Xi^*\cdot m^*$
and $ (\c^{-1}\b)^*=(\Xi^*)_\l\cdot m^*$, i.e., $(\c^{-1}\b)^*\times \b=\varXi^*_\l(m^*,m)$.
 By \eqref{smoothed-condition-r-gamma}, we have  that
$
u\in \f_0^{-1}\Xi$
and
$
u^*\in 
\overline{\f}_0^{-1}(\Xi^*)_\l$,
where $\f_0:=\mathrm{Ann}_{\mathcal{O}}(\mathcal{O}_F/\mathfrak{\f})=\mathfrak{f}\cap \mathcal{O}$. Thus from  \eqref{condition-c-1},  $u^*$ has order prime to $\l$ in $K^n/\Xi_{\l}^*$, i.e., $u^*\in \Xi_{\l}^*\otimes_{\mathcal{O}}\mathcal{O}_{(\l)}=\l\mathcal{O}_{(\l)}\times \mathcal{O}_{(\l)}^{n-1}$, thus $(Q^*,U^*,\varXi^*_{\l})\in \mathbf{G}_n^\l$

The following result is a direct consequence of \Cref{dual-L-func-parametrization-via cocycle} and \Cref{l-smoohted-in-terms-of-smoothed-Dedekind sums}, and the fact 
that $W(\lambda,\mathfrak{c}^{-1}\mathfrak{b},r,r^*)=N_{F/\Q}(\mathfrak{c})W(\lambda,\mathfrak{b},r,r^*)$ (see also \cite[Theorem 1.2]{BCG2})

\begin{thm}\label{parametrization-dual-smoothed-L-function}
The  value
\begin{equation*}
 [U_\f:V_\f]\cdot 
  \Gamma(l)^n \cdot 
  (2 \pi i)^{nk}
    \L_{\f,\c}(\lambda,\mathfrak{b},r,r^*,0)
\end{equation*}
equals
\[
W_{\mathfrak{c}}
    \det(\overline{M})
( -2  \pi  i)  ^ {   n (l-1)    }
\Phi_{\l}(\mathfrak{E}^*,Q^*,U^*,(\varXi^*)_\l),
\]
where $W_{\c}:=W(\lambda,\mathfrak{c}^{-1}\mathfrak{b},r,r^*)$ is given by \eqref{functional-factor-2}, $\mathfrak{E}^*$ is given by \eqref{dual-cycle}, and $(Q,U,\varXi)\in \mathbf{G}_{n,\l}(\lambda,\mathfrak{b},r,r^*)$.
In particular, for $\mathfrak{b}=\mathfrak{a}\f^{-1}$, $r=1$ and $r^*=0$, we have
\begin{equation*}
 [U_\f:V_\f]\cdot 
  \Gamma(l)^n \cdot 
  (2 \pi i)^{nk}
    L_{\f,\c}(\mathfrak{a},\chi,0)
    =
 T ( -2  \pi  i)  ^ {   n (l-1)    }\Phi_{\l}(\mathfrak{E}^*,Q^*,U^*,(\varXi^*)_\l),
\end{equation*}
where  $T=\chi(\mathfrak{a})W_\c
    \det(\overline{M})$.
  
    \end{thm}

\subsection{A special choice for the parameters $(Q,U,\varXi)\in \mathbf{G}_n^{\l}(\lambda,\mathfrak{b},r,r^*)$}
\label{A special choice for the  complex  and $p$-adic periods}

Let $\p$ be a prime ideal of $K$ above the rational prime $p>3$ which splits in $K/\Q$.
By \cite[\S 81D]{Ome}, we can pick a basis $\{m_i\}$ of $F/K$ such that
\begin{equation}\label{dec-basis-b--df}
\left\{
\begin{aligned}
&\, \b\,\otimes_{\mathcal{O}}\mathcal{O}_{(\l)}=
\mathcal{O}_{(\l)} m_1+\cdots+\mathcal{O}_{(\l)} m_n \\
&\b\c^{-1}\otimes_{\mathcal{O}}\mathcal{O}_{(\l)}=
\l^{-1}\mathcal{O}_{(\l)} m_1+\cdots+\mathcal{O}_{(\l)} m_n\\
&\, \b\,\otimes_{\mathcal{O}}\mathcal{O}_{(\p)}=
\mathcal{O}_{(\p)} m_1+\cdots+\mathcal{O}_{(\p)} m_n\\ 
&\, \b\,\otimes_{\mathcal{O}}\mathcal{O}_{(\overline{\p})}=
\mathcal{O}_{(\overline{\p})}\, m_1+\cdots+\mathcal{O}_{(\overline{\p})} \,m_n
\end{aligned}
\right.
\end{equation}
Note that if $\b$ is $\overline{\p}$-integral, this implies that
\begin{equation}\label{det-M-p-integral}
\text{$\overline{\det(M)}$ is $\p$-integral.}
\end{equation}
For this $M$, let $\mathfrak{E}_{\f,M^*}$ and $\mathfrak{A}_{M^*,\pi}$  be as in \eqref{dual-cycle}. Let $\sigma_\pi=\sigma_{\pi}(1)$ denote the matrix formed by taking the first column of each of the matrices of $[\varrho(\epsilon_{\pi(1)})^h|\cdots | \varrho(\epsilon_{\pi(n-1)})^h) ]$. Let 
$
\eta_{1,\pi}=1,\ \eta_{2,\pi}=\epsilon_{\pi(1)},\dots,\ 
\eta_{n,\pi}=\epsilon_{\pi(1)}\cdots \epsilon_{\pi(n-1)}$,
then $\sigma_{\pi}=\sigma_{\pi}(1)\in  M_n(\mathcal{O}_{(\l)})$ 
satisfies, according to \eqref{matrix-rep-trace},
\[
\sigma_{\pi}^h=
\begin{pmatrix}
\mathrm{Tr}_{F/K}(m_1\eta_{1,\pi}\overline{m_1^*})
& \dots &\mathrm{Tr}_{F/K}(m_1\eta_{1,\pi} \overline{m_n^*}) \\
\vdots& \ddots &\vdots\\
\mathrm{Tr}_{F/K}(m_1\eta_{n,\pi}\overline{m_1^*})
& \dots &\mathrm{Tr}_{F/K}(m_1\eta_{n,\pi}\overline{m_n^*})
\end{pmatrix}
\]
so that
$
\sigma_{\pi}^h(m_1,\dots, m_n)^t=(m_1\eta_{1,\pi},\dots, m_1\eta_{n,\pi})^t.$
Note that if $\det(\sigma_{\pi})\neq 0$ for  $\pi\in S_n$, then  
$\{\eta_{\pi,i}\}$ is a basis of $F/K$. Thus, if $u\in K^n$  is such that $r=u\cdot m$, then
\[
\frac{r}{m_1}=u\sigma_{\pi}^* \cdot  (\eta_{1,\pi},\dots,\eta_{n,\pi})^t.
\]
Let $\mathfrak{e}$ denote $\b r^{-1}\cap \mathcal{O}_F $, the order of $r$ in $F/\mathfrak{b}$ as an $\mathcal{O}_F$-module, so that $\mathfrak{e}|\f$. Then we have the following result.

\begin{lem}\label{condition-u-important} 

Suppose that $\b\mathfrak{e}^{-1}=(\mathfrak{b}^{-1}\cap r^{-1}\mathcal{O}_F)^{-1}$ is relatively prime to $\p\overline{\p}\l$
and that $\mathfrak{e}$ is  divisible by an integral ideal $\mathfrak{g}$ of $K$ that is also relatively prime to $\p\overline{\p}\l$ and satisfies $n!<N_{K/\Q}(\mathfrak{g})$.
Then we can choose a basis $\{m_i\}$ of $F/K$ satisfying \eqref{dec-basis-b--df} such that
\begin{equation}\label{condition-u-important-sigma}
0\notin (u+\mathcal{O}_{(\mathfrak{e})}\Xi)\sigma_{\pi,i}^*
\end{equation}
for all $i=1,\dots, n$ and all $\pi\in S_{n-1}$ for which $\det(\sigma_{\pi})\neq 0$, where $\sigma_{\pi,i}^*$ denotes the $i$th column of $\sigma_{\pi}^*$. 
In particular,
note that if $\p$ is coprime to $\mathfrak{e}$ then 
$\p^{-\infty}\subset \mathcal{O}_{(\mathfrak{e})}$.

\end{lem}

\begin{proof}
We start by noticing that it suffices to show that
\begin{equation}\label{prime-idel}
0\notin (u+\Xi)\sigma_{\pi,i}^*
\quad \text{}
\end{equation}
for all $i=1,\dots, n$ and all $\pi\in S_{n-1}$ for which $\det(\sigma_{\pi})\neq 0$. Indeed,
if there exists a nonzero $h\in \O$ prime to $\mathfrak{e}$ such that 
 $hu\sigma_{\pi,i}^*\in \Xi\sigma_{\pi,i}^*$, then we can find  $x\in O$ and $e\in \mathfrak{e}$ such that $1=hx+e$
 which implies $u\sigma_{\pi,i}^*\equiv hu\sigma_{\pi,i}^*\,\mathrm{mod}\, \Xi\sigma_{\pi,i}^* $
which contradicts 
 \eqref{condition-u-important-sigma}.

Letting $\{\eta_{i,\pi}^*\}$  denote the dual basis of $\{\eta_{i,\pi}\}$, for every $\pi\in S_{n-1}$, we see that  \eqref{prime-idel} is equivalent to showing that there exists a basis $\{m_i\}$ such that
\begin{equation}\label{trace-nonzero-r}
0\notin \mathrm{Tr}_{F/K}\left( \, 
(r+ \b) \cdot \frac{\overline{\eta^*_{i,\pi}}}{m_1}\, \right)
\quad \text{for all $i=1,\dots,n$ and $\pi \in S_{n-1}$.}
\end{equation}
In order to do this, we will show that there exists a $\beta\in r\mathfrak{e}\b^{-1}\subseteq \mathcal{O}_F$ ( which implies $\beta\b\subseteq r\mathfrak{e}$) such that $\beta\equiv 1\mod{\p\overline{\p}\l}$ and 
\begin{equation}\label{nonzero-trace-ver-2}
0\notin\mathrm{Tr}_{F/K}\left(\,(\beta+\mathfrak{g}\mathcal{O}_F)\, \frac{\overline{\eta_{i,\pi}^*}}{m_1}r\,\right)
\quad 
\text{for all $i=1,\dots, n$ and $\pi\in S_{n-1}$,}
\end{equation}
 because in this case if $\{m_i\}$ 
is a basis satisfying \eqref{dec-basis-b--df}, then we replace it for $\{m_1/\beta, m_2,\dots, m_n\}$ and thus \eqref{trace-nonzero-r}
will follow from \eqref{nonzero-trace-ver-2} and the following 
\begin{equation*}\label{chain-ver 3}
(r+\b) \frac{\overline{\eta_{i,\pi}^*}}{m_1}\beta
=
(r\beta+\beta\b) \frac{\overline{\eta_{i,\pi}^*}}{m_1}
\subseteq
(r\beta+r\mathfrak{e}\mathcal{O}_F) \frac{\overline{\eta_{i,\pi}^*}}{m_1}
\subseteq
(\beta+\mathfrak{g}\mathcal{O}_F) \frac{\overline{\eta_{i,\pi}^*}}{m_1}r.
\end{equation*}

The existence of this $\beta$ can be proven as follows.
For any $x\in F- \{0\}$ we have that $T_x:=\{y\in \mathcal{O}_F\,:\, \mathrm{Tr}_{F/K}(yx)\in \mathrm{Tr}_{F/K}(x\mathfrak{g}\mathcal{O}_F)\}$ is a free $\mathcal{O}$-module of $F$ with $\mathfrak{g}\mathcal{O}_F\subseteq T_x \subseteq \mathcal{O}_F$  such that $[\mathcal{O}_F:T_x]_{\mathcal{O}}=\mathfrak{g}$  and $[T_x:\mathfrak{g}\mathcal{O}_F]_{\mathcal{O}}=\mathfrak{g}^{n-1}$.
Indeed, $[\mathcal{O}_F:\mathfrak{g}\mathcal{O}_F]_{\mathcal{O}}=\mathfrak{g}^{n}$ and $T_x$  is the kernel of the surjective map
\[
\mathcal{O}_F\to 
\mathrm{Tr}_{F/K}(x\mathcal{O}_F)/\mathfrak{g}\mathrm{Tr}_{F/K}(x\mathcal{O}_F)\ :\ 
y\mapsto  \mathrm{Tr}_{F/K}(yx).
\]

Let $x_{i,\pi}:=r\,\overline{\eta_{i,\pi}^*}/m_1$, for $i=1,\dots,n$ and $\pi\in S_{n-1}$.
Since $[\mathcal{O}_F:\mathfrak{g}\mathcal{O}_F]=N_{K/\Q}(\mathfrak{g})^n$ and $[T_x:\mathfrak{g}\mathcal{O}_F]=N_{K/\Q}(\mathfrak{g})^{n-1}$, then $\bigcup_{i,\pi} T_{x_{i,\pi}}/\mathfrak{g}\mathcal{O}_F$ has at most $n!N_{K/\Q}(\mathfrak{g})^{n-1}$ elements. Since we are assuming $n!<N_{K/\Q}(\mathfrak{g})$, then $\mathcal{O}_F/\mathfrak{g}\mathcal{O}_F \neq \bigcup_{i,\pi} T_{x_i}/\mathfrak{g}\mathcal{O}_F$, thus there exists a $\theta\in \mathcal{O}_F$ such that
\[
0\notin\mathrm{Tr}_{F/K}\big(\,(\theta+\mathfrak{g}\mathcal{O}_F)\,x_{i,\pi}\,\big)
\quad \text{for all $i=1,\dots, n$ and $\pi\in S_{n-1}$. }
\]

On the other hand, note that if $\mathfrak{B}$ is a prime ideal of $F$ that divides $\mathfrak{e}r\b^{-1}$, then it must be relatively prime to $\mathfrak{e}$; otherwise we would get the contradiction $\mathfrak{e}\mathfrak{B}^{-1}=\mathcal{O}_F\cap \b r^{-1}$.
Since we are assuming that $\mathfrak{g}|\mathfrak{e}$, then the ideal
 $\mathfrak{e}r\b^{-1}$ is coprime to $\mathfrak{g}$ and, furthermore, we may assume that it is also coprime to $\p\overline{\p}\l$, for if not we can find an $r'\in r+\mathfrak{b}$ 
 that is prime to $\p\overline{\p}\l$; so that $\mathfrak{e}r'\b^{-1}$ is too. Indeed,  let $t\in F$ 
such that $t\mathcal{O}_F=\mathfrak{e}\mathfrak{b}^{-1}\mathfrak{s}$ for some integral ideal $\mathfrak{s}$ of $F$ prime to $\p\overline{\p}\l$; note that $t$ is prime to $\p\overline{\p}\l$. Let $\mathfrak{h}:=\mathrm{g.c.d}\{\mathfrak{e},\p\overline{\p}\l\}$, 
then, by the Chinese Remainder Theorem, we can find an $x\in \mathcal{O}_F$ prime to $\p\overline{\p}\l\mathfrak{h}^{-1}$ and such that $x\equiv rt \mod \mathfrak{b}t=\mathfrak{e}\mathfrak{s}$. This $x$ is actually prime to $\p\overline{\p}\l$, for if  $x$ and $\mathfrak{h}$ have some factor in common, then it would divide $r$, and so $\mathfrak{e}r\mathfrak{b}^{-1}$ (since $\mathfrak{e}\mathfrak{b}^{-1}$ is prime to $\p\overline{\p}\l$). By the observation at the beginning of this paragraph,  this factor must be trivial.
Thus,  $r'=xt^{-1}\in r+\mathfrak{b}$ is prime to $\p\overline{\p}\l$.

  Therefore, once again by the Chinese Remainder Theorem, we can find a $\gamma\in \mathcal{O}_F$
such that 
\[
\gamma \equiv -\theta \mod{r\mathfrak{e} \b^{-1}},\
\gamma\equiv 0 \mod{\mathfrak{g}},\
\gamma \equiv 1-\theta \mod{\p\overline{\p}\l}. 
\]
Thus, $\beta=\theta+\gamma$ will satisfy the above requirements.

\end{proof}

\subsection{Integrality of the special values}
\label{Integrality of the special values}

Let $\f_0$ be the integral ideal  of $F$  which we assume is divisible  by an integral ideal $\mathfrak{g}$ of $K$ coprime to $\l$ such that $n!<N_{F/\Q}(\mathfrak{g})$ and $w_{\mathfrak{g}}=1$.
Let $p>3$
be a rational prime which splits in $K/\Q$ as $p=\p\overline{\p}$ and is relatively prime to $\l\f_0\mathfrak{D}_F$.  Let $\f$ be an integral ideal as defined in
 \Cref{The smoothed partial $L$-function},  coprime to $\p$ and divisible by $\f_0$. We choose for $\f_1$ and $\mathbf{a}$ in \Cref{Complex periods}  the ideals
\begin{equation*}\label{defin-f1-a}
\f_1:=\mathfrak{g}
\quad \text{and}\quad 
\mathbf{a}:=\mathfrak{l}.
\end{equation*}
Let $\varphi$, $\Omega_{\infty}$, $\Omega_{\infty}^*$, $L$
and  $E(\C)\simeq \C/L$ be as in  \Cref{Complex periods} . 
We  note that we can express the character $\lambda=\lambda_{F,k,l}$ in terms of $\varphi$ in the following way: 
\[
\lambda(x)=\overline{\psi(x) }^k\psi(x)^{-l},
\quad \text{for all $x\equiv 1 \mod \f$,}
\]
where $\psi=\varphi\circ N_{F/K}.$

For these periods we have, from 
\Cref{parametrization-dual-smoothed-L-function},  
\begin{equation}\label{special-value-cocycle-side}
[U_\f:V_\f]
\cdot 
 \Gamma(l)^n
\cdot
\frac{(2\pi i)^{nk}}
{\Omega_{\infty}^{nk}\Omega_{\infty}^{*nl}}
\cdot \L_{\f,\c}(\lambda,\mathfrak{b},r,r^*,0)
\,
=\,
C_\c\cdot \tilde{\Phi}_\l(\mathfrak{E}^*,Q^*,U^*,(\varXi^*)_\l),
\end{equation}
where $C_\c=C_\c(\lambda,\mathfrak{b},r,r^*)=W_{\mathfrak{c}}
    \det(\overline{M})\Omega^n_{\infty}\Omega_{\infty}^{-*n}\in K^{\mathrm{ab}}$; recalling that $\Omega_{\infty}^*\Omega_{\infty}^{-1}=-2i \mathrm{Vol}(\mathbf{a})$.
    Furthermore, for $\mathfrak{b}=\mathfrak{a}\f^{-1}$, $r=1$ and $r^*=0$, we have
\begin{equation}\label{special-value--cocycle-param-EK-co}
[U_\f:V_\f]
\cdot 
 \Gamma(l)^n
\cdot
\frac{(2\pi i)^{nk}}
{\Omega_{\infty}^{nk}\Omega_{\infty}^{*nl}}
\cdot 
L_{\f,\c}(\aa,\chi,0)
=
 \big\langle\, \big[\tilde{\Phi}_{\l}\big],\,\big[\mathfrak{Z}_{\mathfrak{a},\f,\chi,\c}\big]\,\big\rangle,
\end{equation}    
where   $\mathfrak{Z}_{\mathfrak{a},\f,\chi,\c}=\mathfrak{E}^*\otimes f^*\in \Z[\Gamma_\l^n]\otimes_{\Gamma_\l}\mathcal{D}_{\overline{\Q}}^{\vee}$ 
 for  $f^*=\chi(\mathfrak{a})\cdot C_\c\cdot f_{(Q^*,U^*,(\varXi^*)_\l)}$.
 
 We are now ready to deduce the algebraicity and integrality of the critical values.
\begin{cor}
\label{integ-smoothed}
The special value
\be\label{special value-L-at-0}
[U_\f:V_\f]
\cdot 
 \Gamma(l)^n
\cdot
\frac{(2\pi i)^{nk}}
{\Omega_{\infty}^{nk}\Omega_{\infty}^{*nl}}
\cdot \L_{\f,\c}(\lambda,\mathfrak{b},r,r^*,0)
\ee
is algebraic; more specifically, it belongs to $K^{\mathrm{ab}}$. Moreover,  if
$\mathfrak{g}$ divides $\mathfrak{e}:=\b r^{-1}\cap \mathcal{O}_F$
and $\mathfrak{b}\mathfrak{e}^{-1}$ is prime to $\p\overline{\p}\l$, then   \eqref{special value-L-at-0} is $\p$-integral.
In particular,  the special value
\begin{equation*}\label{special value-L-at-0-particular}
[U_\f:V_\f]
\cdot 
 \Gamma(l)^n
\cdot
\frac{(2\pi i)^{nk}}
{\Omega_{\infty}^{nk}\Omega_{\infty}^{*nl}}
\cdot 
L_{\f,\c}(\aa,\chi,0)
\end{equation*}
is $\p$-integral for every integral ideal $\mathfrak{a}$ of $F$ prime to $\p\overline{\p}\l$.
\end{cor}

\begin{proof}

 The first part of the statement, namely that \eqref{special value-L-at-0}
is algebraic,
follows from  \eqref{special-value-cocycle-side} and  \Cref{algeb-smoothed-psi-l-1}. 

As for the $\p$-integrality of \eqref{special value-L-at-0},  first observe 
that the factor $C_\c$ of \eqref{special-value-cocycle-side} is $\p$-integral by \eqref{condition for p-integral} and \eqref{det-M-p-integral}. Thus it remains to apply \Cref{congruence-psi} to 
$\tilde{\Phi}_{\l}(\mathfrak{A}_{M^*,\pi})$ evaluated at  the tuple $(Q^*,U^*,(\varXi^*)_\l)\in \mathbf{G}_n^\l$, for which we have to verify additionally  that
 \begin{equation}\label{-u-conditions}
\bigg(\ u+\p^{-\infty} \big((\Xi^*)_\l\big)^*\ \bigg)\sigma_{\pi}^*\subset K^{\times n},\ \text{and}
  \end{equation}
\begin{equation}\label{condition-dual-psi}
Q^*\big(\,\overline{U^*+\varXi^*+(\varXi^*)^\l} \,\big) 
\subset \mathcal{O}_{(\p)}.
\end{equation}

The condition 
\eqref{-u-conditions} follows from  \Cref{condition-u-important} since $\big((\Xi^*)_\l\big)^*\subset \overline{\l}^{-1}\Xi$ and $\overline{\l}$ is prime to $\f$. 
Now, by the very definition, we have
\[
Q^*\big(\,\overline{U^*+\varXi^*+(\varXi^*)^\l} \,\big) 
=
N_{F/K}\big(\,\overline{-r^*+\b^*}\,\big)^{l-1}
\cdot
\overline{N_{F/K}(\,r+\c^{-1}\b\,)}^k.
\]
 Since $r\in \mathfrak{e}^{-1}\mathfrak{b}$, $r^*\in \overline{\f}^{-1}(\mathfrak{c}^{-1}\mathfrak{b})^*$, then
 \[
 Q^*\big(\,\overline{U^*+\varXi^*+(\varXi^*)^\l} \,\big)
  \subset 
    N_{F/K}\big(\,\f^{-1}\overline{\b^*}\,\big)^{l-1}
\cdot
\overline{N_{F/K}(\,\mathfrak{e}^{-1}\c^{-1}\b\,)}^k,
 \]
which is contained in $\mathcal{O}_{(\p)}$ since $\f$ is coprime to $\p$, $\mathfrak{e}^{-1}\mathfrak{b}$ is coprime to $\p\overline{\p}\l$, and thus $\mathfrak{b}$ and
 $\overline{\mathfrak{b}^*}=\mathfrak{b}^{-1}\mathfrak{D}_F^{-1}$ are corpime to $\p$ as well. Thus, \eqref{condition-dual-psi} follows.

Note, in particular, that for $\mathfrak{b}=\aa^{-1}\f$, $r=1$ and $r^*=0$ we have
$\mathfrak{e}=\f$
\[
 Q^*\big(\,\overline{U^*+\varXi^*+(\varXi^*)^\l} \,\big)
  \subset 
    N_{F/K}\big(\,\f^{-1} \aa \mathfrak{D}^{-1}_F\,\big)^{l-1}
\cdot
\overline{N_{F/K}(\,\c^{-1}\aa^{-1}\,)}^k,
 \]
thus the second part of the corollary follows.

\end{proof}


\section{$p$-adic measures and $p$-adic zeta functions}
\label{p-adic measures and p-adic zeta functions}

In this section, we describe the construction of $p$-adic measures from the Eisenstein-Kronecker cocycle, following the method outlined in \cite{CD}. More specifically, we obtain a measure-valued cocycle $\mu_\l$ and, using the results of  \Cref{parametrization-dual-smoothed-L-function}, derive the interpolation of the critical values of $L^*_{\f_0,\c}$.
\subsection{$p$-adic measures associated to $\Phi_{\l}$}
Let $\p$ be  a prime ideal of $K$ as in \Cref{sec-p-adic-periods} and let $\mathcal{O}_{K_{\overline{\p}}}$ be the ring of integers of the $\overline{\p}$-adic completion $K_{\overline{\p}}$ of $K$ at the prime $\overline{\p}$. 
Let $\varXi_\l=\Xi_\l\times \Xi\in \mathcal{W}_{n,\l}$ such that 
\[
X:=  \varXi_\l  \otimes_{\mathcal{O}}  \mathcal{O}_{K_{\overline{\p}}}
\simeq 
\mathcal{O}_{K_{\overline{\p}}}^{2n}.
\]

For $\mathfrak{A}\in \Gamma_\l^n$ and $U\in (\l\mathcal{O}_{(\l)}\times \mathcal{O}_{(\l)}^{n-1})\times K^{n}$, we  define a measure 
$\mu_\l(\mathfrak{A},U)=\mu_\l(\mathfrak{A},U,\varXi_\l)$ on
$
X_{U}:=  U+X \subset K_{\overline{\p}}^n\times  
K_{\overline{\p}}^n
$
as follows. 
Let $\sigma$ be the matrix whose columns are the first columns of the matrices in the tuple $\mathfrak{A}$. If $\det(\sigma)=0$, then $\mu_\l(\mathfrak{A},U)$ is the zero measure. Suppose then that $\det(\sigma)\neq 0$.

  Let $\varUpsilon_\l=\Upsilon_\l\times \Upsilon^* \in \mathcal{W}_{n,\l}$ such that
  \[
  [\Xi: \Xi\cap \Upsilon]_{\mathcal{O}}\subset \overline{\p}^{\infty}
  \quad \text{and}\quad
   [\Upsilon: \Upsilon\cap \Xi]_{\mathcal{O}}\subset \p^{\infty}. 
  \]
  This implies 
$
  Y:=\varUpsilon_\l\otimes_{\mathcal{O}} 
  \mathcal{O}_{K_{\overline{\p}}}=(\varXi_\l\cap \varUpsilon_\l) \otimes_{\mathcal{O}} \mathcal{O}_{K_{\overline{\p}}}
 \subset 
X.
$
 For $V\in U+\varXi_\l$ we define
$
Y_{V} 
:=  \, V+ Y$.
 These sets  form a  basis of compact open subsets of 
$X_{U}$ and we define the measure
\[
\mu_\l(\a,U)(Y_{V})
=
\tilde{D}_{\l,U+ \varXi_\l}
\left(\sigma(1),V,\varUpsilon_\l\right)
\in \C_p.
\]
The distribution relation from \Cref{lemma-decomp-phi-Xi-Upsilon} and \Cref{D-is-measure-rem} imply that this is indeed a measure.
Note that if $\big(\,(1,1),U,\varXi_\l\,\big)$ satisfies \eqref{p-adic-condition-tuple}, then the measure is
integral, i.e., for all sets $Y_{V}$ we have
\[
\mu_{\l}(\a,U)(Y_{V})
\in \widehat{\mathcal{O}}.
\]

\begin{cor}\label{general-region-integration}
Let $Y_{V}$ as above and let $Q\in H^2$, then  we have 
\begin{align}\label{holas}
\frac{\eta_{\p}^{\deg q}}{\Omega_{\p}^{\deg q^*}}
\int_{Y_{V}} 
 Q(\overline{X})
\, 
d\mu_\l(\a,U)= 
 \tilde{D}_{\l,U+\Xi_\l}(\sigma(1),Q,V,\varUpsilon_\l),
\end{align}
  where 
  $Q(\overline{X})=Q(\overline{x},\overline{x^*})=q(\overline{x})q^*(\overline{x^*})$.
\end{cor}
\begin{proof}
By linearity, we may assume $Q(\overline{U+\varXi+\varXi^\l})\subset \mathcal{O}_{(\p)}$, so that
$(Q,U,\varXi_\l)$ satisfies \eqref{p-adic-condition-tuple}. Now, let $\mathtt{M},\mathtt{N}$ be nonnegative integers large enough such that
$\overline{\pi}^{\mathtt{M}}\Xi_\l\subset \Xi_\l\cap \Upsilon_\l$ and $\overline{\pi}^{\mathtt{N}}\Xi^*\subset \Xi^*\cap \Upsilon^*$, and consider the $\mathcal{O}$-module $\varTheta_\l:=(\varXi_\l)_{\mathtt{M},\mathtt{N}}\in \mathcal{W}_{n,\l}$. 
Then, for $Z:= \varTheta_{\l}\otimes_{\mathcal{O}}  \mathcal{O}_{K_{\overline{\p}}}$, we have the isomorphism
\[
\varXi_\l\cap \varUpsilon_\l\,\big/ \varXi_\l\cap \varTheta_{\l}
\ \xrightarrow{\sim}\ 
 Y \,\big/ 
Z
\]
from which we
obtain  the following decomposition:
\[
Y_{V}= 
\bigsqcup_{H\in \varXi_\l\cap \varUpsilon_\l\,\big/ \varXi_\l\cap \varTheta_{\l}}
Z_{V+H},
\]
where $Z_{V+H}$ denotes the translate of $Z$ by $V+H$. Letting $W=V+H$, we see that as 
as $H$ suns through a set  of representatives of classes of $ \varXi_\l\cap \varUpsilon_\l\,\big/ \varXi_\l\cap \varTheta_{\l}$, then $W$ runs through a set of repesntatives of classes $\mathcal{K}_\l/\varTheta_\l$ such that $W\equiv V \,\mathrm{mod}\, \varXi_\l\cap \varUpsilon_\l$.

Using this decomposition, we can approximate the left-hand of \eqref{holas}  side via Riemann sums as follows: for
nonnegative integers $\mathtt{M}$ and  $\mathtt{N}$ sufficiently large we have that
the left-hand side of \eqref{holas} is congruent to
\[
\frac{\eta_{\p}^{\deg q}}{\Omega_{\p}^{\deg q^*}}
\sum_{
\substack{
W\in \mathcal{K}_\l/\varTheta_\l\\
W\equiv V \,\mathrm{mod}\, \varXi_\l\cap \varUpsilon_\l
}
}
Q(\,\overline{W}\,)\,
\mu_{\l}(\mathfrak{A},U)(Z_{W})
\]
modulo $\pi^{\min \{ \mathtt{M}, \mathtt{N}\}-\epsilon }  \widehat{\mathcal{O}}$,
 for some  nonnegative integer $\epsilon$ depending on the coefficients of $P$.
By  the very definition of $\mu_{\l}$,  this is equal to
\begin{equation}\label{expanded-sum-psi-0}
\sum_{
\substack{
W\in \mathcal{K}_\l/\varTheta_\l\\
W\equiv V \,\mathrm{mod}\, \varXi_\l\cap \varUpsilon_\l
}
}
Q(\,\overline{W}\,)\cdot 
\frac{\eta_{\p}^{\deg q}}{\Omega_{\p}^{\deg q^*}}
\cdot
\tilde{D}_{\l,U+ \varXi_\l}
\left(\sigma(1),W,\varTheta_\l\right)
\end{equation}

On the other hand, by \Cref{congruence-psi}, the right-hand side of \eqref{holas} 
 is congruent to \eqref{expanded-sum-psi-0}  
 modulo $\pi^{\min \{ \mathtt{M}, \mathtt{N}\}-\delta }  \widehat{\mathcal{O}}$,
for some nonnegative integer $\delta$ depending only on $(Q,U,\varXi,\varLambda)$. Therefore, the left-hand and right-hand side of \eqref{holas} are congruent modulo $\pi^{\min \{ \mathtt{M}, \mathtt{N}\}-
\max\{\epsilon,\delta\} }  \widehat{\mathcal{O}}$. Letting $\mathtt{M},\mathtt{N}\to \infty$ the identity follows.

\end{proof}

Let $\mathbb{M}_{p}$ denote the space of functions that assigns to each 
$U\in  K^{2n}_\l$ a $\C_p$-valued measure 
$\mu(U)$ on $X_U:=U+ \mathcal{O}_{K_{\overline{\p}}}^{2n}$
such that 
\begin{equation*}\label{measure-scalar}
\overline{\pi}^{n}\mu\left(U\pi^{-1}\right)
\left(\,Y \pi^{-1}\,\right)=\mu(U)(Y)
\quad \text{and}\quad 
\overline{\pi}^{n}\mu\left(U\overline{\pi}\right)
\left(\, Y \overline{\pi}\,\right)=\mu(U)(Y)
\end{equation*}
for all $Y\subset X_{U}$, where $Yt=\{(yt,y^*\overline{t}^{-1}   ):(y,y^*)\in Y\}$ for $t\in K^{\times}$.

Let
\[
\Gamma_{\l,p}:=
\Gamma_0\big(\,\l \mathcal{O}_{K}[1/p]\,\big)
=
\Gamma_{\l}\,\cap \, \mathrm{GL}_n\big(\,\mathcal{O}[1/p]\,\big).
\]
The space $\mathbb{M}_{p}$ has the structure of a $\Gamma_{\l,p}$--module given by
\[
(\gamma \mu)(U)(Y):=
[ \mathcal{O}_{K_{\overline{\p}}}^n : \mathcal{O}_{K_{\overline{\p}}}^n A^* ]\cdot
 \det(A) \cdot 
\mu(UA)(Y A),
\]
where $A=\frac{\overline{\pi}^r}{\pi^{s}}  \gamma$ for some nonnegative integers $r$ and $s$ large enough so that  $A$ and $ A^*$ are in $M_n(\mathcal{O}[1/\p])$; in particular, they are in $M_n(\mathcal{O}_{K_{\overline{\p}}})$ which implies
 $ \mathcal{O}_{K_{\p}}^{n}A\times \mathcal{O}_{K_{\p}}^{n}A^* \subset \mathcal{O}_{K_{\p}}^{2n}$, and where $Y A:=\{(yA,y^*A^*)\,:\,(y,y^*)\in Y\}$.

The cocycle $\Phi_{\l}$  defines the homogeneous $(n-1)$-cocycle
\begin{align*}
\mu_{\l}\,:\,&\Gamma_{\l,p}\to \mathbb{M}_{p}\\
 &\mathfrak{A}\mapsto \mu_\l(\mathfrak{A},\cdot)
\end{align*}
Indeed, the homogeneity is a consequence of \Cref{general-region-integration}.

\subsection{$\p$-adic interpolation of  special values}
\label{$p$-adic zeta functions}

We are now ready to prove the second main result of this paper: the interpolation of the critical values of  $L_{\f_0,\c}^*$
  from the measure-valued cocycle $\mu_\l$ using the method in \cite{CD}. To simplify the exposition, we will divide the proof into two parts. First, we will prove the interpolation of the values of $\mathcal{L}_{\f,\c}$  (cf. \Cref{Interpolation of the values of the special}). Then, we will deduce the interpolation of the values of $L_{\f_0,\c}^*$
  (cf. \Cref{Interpolation of the special values of L-star}).

We now fix the following notation for this section: Let $\Omega_{\infty}$, $\Omega_{\infty}^*$, $\f_0$, $\mathfrak{g}$, and $p=\p\overline{\p}$ be as defined in  \Cref{Integrality of the special values}. Let $\Omega_{\p}$
and $\Omega_{\p}^*$ be the $\p$-adic periods defined in  
\Cref{Integrality properties of psi-lk}.
Write  the integral ideal $\f$ of $F$ as $\f_0 \f_\p\f_{\overline{\p}}$,
where $f_0$ the prime to $p$-part of $\f$, and  $\f_\p$ and  $\f_{\overline{\p}}$ contain only primes above $\p$ and  $\overline{\p}$, respectively.

\subsubsection{Interpolation of the values of  $\mathcal{L}_{\f,\c}$}
\label{Interpolation of the values of the special}
Let $\lambda=\lambda_{k,\l}$, $\mathfrak{b}$, $r\in \mathfrak{b}\f_0^{-1}$ and $r^*\in (\c^{-1}\b)^* \overline{\f}_0^{-1}$ satisfying \eqref{condition-gamma}.   Furthermore, for $\mathfrak{e}=\b r^{-1}\cap \mathcal{O}_F$ we assume that $\mathfrak{g}|\mathfrak{e}$ and $\mathfrak{b}\mathfrak{e}^{-1}$ is prime to $p\l$.
Let $\mathfrak{h}_p$ denote the ideal $\mathfrak{h}_{\p}^{-1}\mathfrak{h}_{\overline{\p}}$,
where $\mathfrak{h}_{\p}$ is an integral ideal of $F$ whose factorization into prime ideals contains only primes  above $\p$ and $\mathfrak{h}_{\overline{\p}}$  only prime ideals  above $\overline{\p}$. 
Let $h\in \mathfrak{b}\mathfrak{h}_{p}\f_0^{-1}\f_{\overline{\p}}^{-1}$ and $h^*\in (\c^{-1}\mathfrak{b}\mathfrak{h}_p)^*\overline{\f}_0^{-1}\overline{\f}_{\p}^{-1}$ such that
\begin{equation}\label{H-R-cong}
(h^*, h)\equiv (r^*,r) \mod{ (\c^{-1}\b )^*\times\b}. 
\end{equation}
Then $\mathcal{L}_{\f,\c}(\lambda,\mathfrak{b}\mathfrak{h}_p,h,h^*,s)$ is well-defined, since $r$ and $r^*$ satisfy \eqref{condition-gamma} and so  $h$ and $h^*$
will satisfy it as well.

We will build an $\widehat{\mathcal{O}}$-valued measure
 $\mu_{r,r^*,\mathfrak{b},\c,\p}$  on $\mathcal{O}_F
  \otimes \Z_p$ that interpolates the values $\mathcal{L}_{\f,\c}(\lambda,\mathfrak{b}\mathfrak{h}_{p},h,h^*,0)$ for all   $\mathfrak{h}_{p}$, $h$ and $h^*$.
Using the identification
 \[
 (\,\mathcal{O}_{F^*}\times 
  \mathcal{O}_F\,)\otimes_{\mathcal{O}} \mathcal{O}_{K_{\overline{\p}}}
\ \simeq \
  \mathcal{O}_F
  \otimes  \Z_p,
  \]
let
$
\mathcal{O}_{\mathfrak{h}_{p}}:=
(\mathfrak{h}_{p}^*\times \mathfrak{h}_{p})\otimes_{\mathcal{O}}\mathcal{O}_{K_{\overline{\p}}}=
(\overline{\mathfrak{h}}_{\p}\times 
\mathfrak{h}_{\overline{\p}})\otimes_{\mathcal{O}}\mathcal{O}_{K_{\overline{\p}}}$
and for  $(h^*,h)\in (\mathcal{O}_{F^*}\times \mathcal{O}_{F})\otimes_{\mathcal{O}}\mathcal{O}_{K_{\overline{\p}}}$
 define
$
\mathcal{O}_{h^*,h,\mathfrak{h}_{p}}
:=(h^*,h)+\mathcal{O}_{\mathfrak{h}_{p}}.
$ Denote by
\begin{equation*}\label{map-N-l-k}
 N_{l,k}: \mathcal{O}_{F}\otimes \Z_p\to \widehat{\mathcal{O}}
 \end{equation*}
the map given by 
$
(x\otimes z)\mapsto z^{n(l-1)}\overline{z}^{nk}N_{F/K}(x)^{l-1}
\overline{N_{F/K}(x)}^k.
$

\begin{thm}\label{p-adic-inter-Lfcb}
There exist a unique $\widehat{\mathcal{O}}$-valued measure
 $\mu_{r,r^*,\mathfrak{b},\c,\p}$ on $\mathcal{O}_{F}\otimes \Z_p$, depending 
 only on $\f$ and the classes $r+\mathfrak{b}$ and $r^*+(\c^{-1}\b)^*$, such that \begin{align*}
 \frac{  1  }
   {   \Omega_{\p}^{nk} \Omega_{\p}^{*n(l-1)}    }
& \int_{\mathcal{O}_{h^*,h,\mathfrak{h}_{p}}} 
 N_{l,k}
\, 
d\mu_{r,r^*,\mathfrak{b},\c,\p}=\\
&
\frac{  e(\,-r|h^*\,) }
      {  N_{F/\Q}(\mathfrak{h}_\p)   }
\cdot
[U_\f:V_\f]
\cdot 
 \Gamma(l)^n
\cdot
\frac{(2\pi i)^{nk}}
{\Omega_{\infty}^{nk}\Omega_{\infty}^{*nl}}
\cdot 
\mathcal{L}_{\f,\c}(\lambda,\mathfrak{b}\mathfrak{h}_{p},h,h^*,0).
  \end{align*}
  \end{thm}
 \begin{proof}
Let
$(Q^*,U^*,(\varXi^*)_\l)\in \mathbf{G}_{n}^{\l}(\lambda,\mathfrak{b},r,r^*)$ 
satisfying the conditions of \Cref{condition-u-important}. 
 Let $ \Upsilon\in W_{n,\l}$ such that $\mathfrak{b}\mathfrak{h}_{p}=\Upsilon \cdot m$.
Then $\varUpsilon_\l^*=(\Upsilon^*)_\l\times \Upsilon^*\in \mathcal{W}_{n,\l}$
and $(\c^{-1} \mathfrak{b}\mathfrak{h}_{p} )^*\times \mathfrak{b}\mathfrak{h}_{p}=\varUpsilon^*_\l \cdot (m^*,m)$. Furthermore, let
$V\in K^{2n}$ such that $(h,h^*)=V(m,m^*)$, 
so that, by \eqref{H-R-cong},  $V^*-U^*\in (\varXi^*)_{\l} $ and
$
[\Upsilon^*:\Xi^*\cap \Upsilon^*]=[\Xi+\Upsilon:\Upsilon]=
N_{F/\Q}(\mathfrak{h}_{\overline{\p}}).
$
Also, by the conditions on $h-r$ and $h^*-r^*$, we have
\begin{equation}\label{exponen-factor}
e(-r| h^* ) 
=e(-h| h^*  )
\cdot e\big((h-r)|r^*\big).
\end{equation}

We  have
 the bijective map 
\begin{equation*}\label{bijective-map}
 X_{U^*}:=U^*+ \varXi^*_\l\otimes_{\mathcal{O}}\mathcal{O}_{K_{\overline{\p}}}
 \ \to  \
  (\mathcal{O}_{F^*}\times \mathcal{O}_F)   \otimes_{\mathcal{O}} \mathcal{O}_{K_{\overline{\p}}}
\end{equation*}
 given by  $(x^*,x)\mapsto \big(\,x^*m^* , xm\big )$. From this map we have that
$Y:=\varUpsilon^*_\l\otimes_{\mathcal{O}}
\mathcal{O}_{K_{\overline{\p}}}
\, \simeq \,
\mathcal{O}_{\mathfrak{h}_{p}}$ and
$Y_{V^*}:=V^*+ Y
\,\simeq\,\mathcal{O}_{h^*,h,\mathfrak{h}_{p}}$.
Also, $Q^*$ induces the map  
$X_{U^*}\to \widehat{\mathcal{O}}$ given by
$
(x^*,x)\mapsto Q^*(\overline{x^*},\overline{x})=N_{\overline{M^*}}(\overline{x^*})^{l-1}N_{\overline{M}}(\overline{x})^k
$
 which can be identified with the map $N_{k,l}$. With these identifications, we simply take
\[
\mu_{r,r^*,\mathfrak{b},\c,\p}
:=C_\c\cdot  \mu_{\l}(\mathfrak{E}^*,U^*,\varXi^*_\l),
\]
where $C_\c=C_\c(\lambda,\mathfrak{b},r,r^*)$ is the prime to $\p$-integral  constant
from \Cref{Integrality of the special values}.

Thus, by \Cref{general-region-integration},
 \[
 \frac{  1  }
   {   \Omega_{\p}^{nk} \Omega_{\p}^{*n(l-1)}    }
\int_{Y_{V^*}} 
Q^*
\,
d\mu_{r,r^*,\mathfrak{b},\c,\p}
=
C_\c
 \cdot 
\tilde{\Phi}_{\l,U^*+\Xi_\l^*}(\mathfrak{E}^*,Q^*,V^*,\varUpsilon^*_\l).
\]
From 
\Cref{parametrization-dual-smoothed-L-function} and \eqref{exponen-factor}, the result follows by observing that 
\[
e(u^*|v-u)=
e(v-u|u^*)= e\big(\,(h-r)\,\big|\,r^* \,\big).
\]
 
\end{proof}


\subsubsection{Interpolation of the critical values of $L_{\f_0,\c}^*$}
\label{Interpolation of the special values of L-star}

For the Hecke-character $\chi$, there exist characters
\[
\chi_{\f_\p}:(  \mathcal{O}_F\otimes_{\mathcal{O}} 
\mathcal{O}_{K_{ \p } }     )^{\times} 
\to   \overline{\mathbb{Q}}^{\times}
\quad \text{and}\quad
\chi_{\f_{\overline{\p}}}:(  \mathcal{O}_F\otimes_{\mathcal{O}} 
\mathcal{O}_{K_{\overline{\p} } }     )^{\times} 
\to   \overline{\mathbb{Q}}^{\times}
\]
of conductors $\f_\p$ and $\f_{\overline{\p}}$, respectively, 
such that
\begin{equation}\label{chi-on-prinp-ideal}
\chi((a))= \chi_{\f_\p}(a)\, \chi_{\f_{\overline{\p}}}(a)\,\overline{N_{F/K}(a)}^{k} \,N_{F/K}(a)^{-l}
\end{equation}
for every $a\in \mathcal{O}_F$ prime to $\f$ and such that $a\equiv 1\mod{\f_0}$. From the isomorphism $(  \mathcal{O}_F\otimes_{\mathcal{O}} 
\mathcal{O}_{K_{ \p } }     )^{\times}
\simeq \prod_{\mathfrak{P}|\p}\mathcal{O}_{F_\mathfrak{P}}^{\times}$, we can view  $\chi_{\f_\p}$ as the product of characters $\prod_{\mathfrak{P}|\p}\chi_{\mathfrak{P}}$, where $\chi_{\mathfrak{P}}$ is trivial (i.e., $\chi_{\mathfrak{P}}=1$) if and only if $\mathfrak{P}$
coprime to $\f_\p$. We extend $\chi_{\mathfrak{P}}$ to $\mathcal{O}_{F_\mathfrak{P}}$ by zero if $\mathfrak{P}|\f_\p$, and by 1 if $\mathfrak{P}$ is coprime to $\f_\p$.
This allows as to extend $\chi_{\f_\p} $ to $ \mathcal{O}_F\otimes_{\mathcal{O}} 
\mathcal{O}_{K_{ \p } }\simeq \prod_{\mathfrak{P}|\p}\mathcal{O}_{F_\mathfrak{P}}$ by letting $\chi_{\f_\p}=\prod_{\mathfrak{P}|\p}\chi_{\mathfrak{P}}$. Furthermore, note that we can extend $\chi_{\f_p}$ to any $x\in F$ such that $x\in \mathcal{O}_{F_\mathfrak{P}}$
for all $\mathfrak{P}|\f_\p$ by letting $\chi_{\f_p}(x)=\prod_{\mathfrak{P}|\f_\p}\chi_{\mathfrak{P}}(x)$. Finally, we extend the character $\chi_{\f_{\overline{\p}}}$
to $\mathcal{O}_F\otimes_{\mathcal{O}} 
\mathcal{O}_{K_{\overline{\p} }}$,
and $\chi_{\f_\p}^{-1}$ to $\mathcal{O}_F\otimes_{\mathcal{O}} \mathcal{O}_{K_{ \p } }$ in a similar fashion.

Bearing in mind the isomorphism
\begin{equation*}\label{isom-OF-Zp-times}
(\mathcal{O}_F\otimes_{\mathcal{O}} \mathcal{O}_{ K_{\p} }) \times 
(\mathcal{O}_F\otimes_{\mathcal{O}} \mathcal{O}_{  K_{\overline{\p}} })\, \simeq \,\mathcal{O}_F\otimes \Z_p,
\end{equation*}
we define the character $
\chi_p: \mathcal{O}_F\otimes \Z_p \to   \overline{\mathbb{Q}}^{\times}
$
by
\[
\chi_p(x,y)=\chi_{\f_\p}^{-1}(x)\chi_{\f_{\overline{\p}}}(y).
\]

For $y\in \f_{\p}^{-1}\mathfrak{D}_F^{-1}$ with $(y)=\f_{\p}^{-1}\mathfrak{D}_F^{-1}\mathfrak{n}$ for an integral ideal $\mathfrak{n}$ of $F$ prime to $\f_\p$ we define the Gauss sum (cf. \cite{nemchenok1993imprimitive})
\[
\tau_y(\chi_{\f_{\p}}^{-1})
=
\sum_{\theta \in (\mathcal{O}_F/\f_\p)^{\times}}
\chi_{\f_\p}^{-1}(\theta)\cdot 
e^{ 2\pi i   \mathrm{Tr}_{F/\Q}( y \theta   )  }.
\]

Let  $\f_\p=(\delta_0)\mathfrak{C}$ where $\mathfrak{C}$ is an integral ideal of $F$ prime to $\p\overline{\p}\l\f_0$, with $\delta_0\in F$ such that $\delta_0\equiv 1\,\mathrm{mod}^*\,\f_{\overline{\p}}\f_0$. In this case, we define
\[
\tau(\chi_{\f_\p}^{-1})
=\frac{\chi(\mathfrak{C}^{-1})\cdot
     \lambda_{\infty}^{-1}(\delta_0)\cdot 
\tau_{ \delta_{0}^{-1} }
(\chi_\p^{-1})}{N(\f_\p)}.
\]
Note that this definition is independent of the choice of $\delta_0$ and $\mathfrak{C}$.
Furthermore, observe that
\[
\tau_(\chi_{\f_\p}^{-1})=\prod_{\mathfrak{P}|\f_\p}\tau(\chi_{\mathfrak{P}}^{-1})\prod_{\mathfrak{P}|\mathfrak{g}_{\p}}\chi^{-1}(\mathfrak{D}_{F,\mathfrak{P}}),
\]
where $\tau_\p(\chi_{\mathfrak{P}}^{-1})$ is the local Gauss sum defined in, for example, 
\cite[XIV, \S 4, Theorem 6]{lang1994algebraic} or \cite[XV, \S 2.5]{cassels1967algebraic}, and  $\mathfrak{D}_{F,\mathfrak{P}}$ is the power of prime $\mathfrak{P}$ of $F$
in the prime decomposition of $\mathfrak{D}_{F}$.

\begin{thm}\label{interp-cirt-val-padic-meas}
There exists a unique $\widehat{\mathcal{O}}$-valued measure
 $\mu_{\mathfrak{a},\f_0,\c,\p}$ on $\mathcal{O}_{F}\otimes \Z_p$  such that
 \begin{align*}\label{Main-ident-p-adic-meas-L-special}
 \frac{  1  }
   {   \Omega_{\p}^{nk} \Omega_{\p}^{*n(l-1)}    }
\int_{  (\mathcal{O}_{F} \otimes \Z_p)^{\times} } 
&\chi_p\cdot
N_{l,k}
\, 
d\mu_{\mathfrak{a},\f_0,\c,\p}=\\
&
[U_{\f_0}:V_\f]
\cdot 
\tau(\chi_{\f_\p}^{-1})
\cdot 
 \Gamma(l)^n
\cdot
\frac{(2\pi i)^{nk}}
{\Omega_{\infty}^{nk}\Omega_{\infty}^{*nl}}
L_{\f_0,\c}^*(\aa,\chi,0).\notag
  \end{align*}

\end{thm}

\begin{proof} 
Let $\mathfrak{b}=\aa^{-1}\f_0 \mathfrak{C}$ and fix $r_0\in \mathfrak{b}\f_0^{-1}$ such that $r_0\equiv 1$ $\mathrm{mod}^*\,\f_0$ and   $r_0^*\in (\c^{-1}\b)^*$. We claim that the measure is
\[
\mu_{\mathfrak{a},\f_0,\c,\p}=
\chi(\aa \c \mathfrak{C}^{-1}) \cdot  \mu_{r_0,r_0^*,\mathfrak{b},\c,\p}.
\]

Indeed, let $\mathfrak{g}_p=\mathfrak{g}_\p^{-1}\mathfrak{g}_{\overline{\p}}$ and define $ \mathcal{O}_{\mathfrak{g}_{p}}$
 as in \Cref{Interpolation of the values of the special}, for fixed ideals 
$\mathfrak{g}_{\p}$ and $\mathfrak{g}_{\overline{\p}}$ as in the beginning of this section. We will prove
 \begin{align}\label{sub-Main-ident-p-adic-meas-L-special}
& \frac{  1  }
   {   \Omega_{\p}^{nk} \Omega_{\p}^{*n(l-1)}    }
\int_{  \mathcal{O}_{\mathfrak{g}_{p}}  } 
\chi_p\cdot 
N_{l,k}
\, 
d\mu_{\mathfrak{a},\f_0,\c,\p}=\\
&\hspace{30pt}
[U_{\f_0}:V_\f]
\cdot 
\tau(\chi_{\f_\p}^{-1})
\cdot 
 \Gamma(l)^n
\cdot
\frac{(2\pi i)^{nk}}
{\Omega_{\infty}^{nk}\Omega_{\infty}^{*nl}}
\cdot
\frac{\chi(  \mathfrak{g}_{p}  )}{  N( \mathfrak{g}_\p)}\cdot
L_{\f_0,\c}(\aa   \mathfrak{g}_{p}^{-1},\chi,0).\notag
  \end{align}
By an inclusion-exclusion principle, the result will follow. 
Furthermore, letting $\f_p=\f_{\p}^{-1}\f_{\overline{\p}}$,  the left-hand side of  \eqref{sub-Main-ident-p-adic-meas-L-special} can be expressed as
 \begin{align}\label{expansion-p-adic-integral-1}
 \sum_{r,r^*}\,
 \chi_{\f_\p}^{-1}(\overline{r^*})\cdot
\chi_{\f_{\overline{\p}}}(r)\cdot
 \frac{  1  }
   {   \Omega_{\p}^{nk} \Omega_{\p}^{*n(l-1)}    }
\int_{ \mathcal{O}_{\mathfrak{g}_p}\, \cap\, \mathcal{O}_{r,r^*,\f_p} } 
N_{l,k}
\, &
d\mu_{\mathfrak{a},\f_0,\c,\p},
  \end{align}
  where   $\mathcal{O}_{r,r^*,\f_p}$ is as in \Cref{Interpolation of the values of the special} and the sum is running over all
$
(r^*,r)\in(\mathcal{O}_{F^*}\times \mathcal{O}_{F})\otimes_{\mathcal{O}}\mathcal{O}_{K_{\overline{\p}}}\,\big/\, \mathcal{O}_{\mathfrak{f}_p}.$
We  may assume that $  r\in \b\f_0^{-1} / \b\f_{\p}$
 and $r\equiv r_0\,(\equiv 1)\,\mathrm{mod}^*\,\f_0$, and 
 $r^*\in (\c^{-1}\mathfrak{b})^*/(\c^{-1}\mathfrak{b}\f_\p^{-1})^*$.
Let $\mathfrak{h}_{\p}=\mathfrak{g}_{\p}\f_{\p}$, $\mathfrak{h}_{\overline{\p}}=\mathfrak{g}_{\overline{\p}}\f_{\overline{\p}}$ and $\mathfrak{h}_p=\mathfrak{g}_p\f_p$. For this $r$ and $r^*$ let 
\begin{equation}\label{condition-h-gro}
 h\in (\b\mathfrak{h}_{p})\f_0^{-1}\f_{\overline{\p}}^{-1}/(\b\mathfrak{h}_{p})
 \quad \text{and}\quad
 h\equiv 1\,\mathrm{mod}^*\,\f_0,
 \end{equation}
  and 
 \begin{equation}\label{condition-g-gro}
h^*\in (\c^{-1}\b\mathfrak{h}_p)^*\overline{\f}_\p^{-1}/
  (\c^{-1}\b\mathfrak{h}_p)^*,
 \end{equation}
such that
$
(h^*,h) \equiv (r^*,r) \mod 
\mathcal{O}_{\mathfrak{f}_p}.$ This congruence, together with the fact that $(h^*,h)\in \mathcal{O}_{\mathfrak{g}_p}$,
imply that
$
\mathcal{O}_{h,h^*,\mathfrak{h}_p} =\mathcal{O}_{\mathfrak{g}_p}\, \cap\, \mathcal{O}_{r,r^*,\f_p}$. Thus, according to \Cref{p-adic-inter-Lfcb}, we have
\begin{align*}\label{ident-measure-Partial-L}
 &\frac{  1  }
   {   \Omega_{\p}^{nk} \Omega_{\p}^{*n(l-1)}    }
\int_{  \mathcal{O}_{h,h^*,\mathfrak{h}_p}} 
 N_{l,k}
\, 
d\mu_{\aa,\f_0,\c,\p}
=\\
\frac{e(-h | h^*)  }
     {  N(\mathfrak{h}_\p)  }
\cdot
&\chi(\aa \c \mathfrak{C}^{-1})
\cdot
[U_\f:V_\f]
\cdot 
 \Gamma(l)^n
\cdot
\frac{(2\pi i)^{nk}}
{\Omega_{\infty}^{nk}\Omega_{\infty}^{*nl}}\cdot
\mathcal{L}_{\f,\c}(\lambda_{\infty}, \b \mathfrak{h}_p,h,h^*,0  ).
\notag
\end{align*}
Using the fact that
 $\chi_{\f_\p}(\overline{r^*})=\chi_{\f_\p}(\overline{h^*})$ and $\chi_{\f_{\overline{\p}}}(r)=\chi_{\f_{\overline{\p}}}(h)$, we see
 from \eqref{expansion-p-adic-integral-1}, that  \eqref{sub-Main-ident-p-adic-meas-L-special} will be a consequence of the following result.
\begin{lem}\label{expression-L-f-zero-cond}
\[
[U_{\f_0}:U_{\f}]\cdot
\tau(\chi_{\f_\p}^{-1})\cdot
\frac{\chi(\mathfrak{g}_{p})}
{   N(\mathfrak{g}_{\p})^{1-s}   N(\mathfrak{g}_{\overline{\p}})^s    }
\cdot
L_{\f_0,\c}(\mathfrak{a}\mathfrak{g}_{p}^{-1},\chi,s)
\]
is equal to
\[
T\cdot
\sum_{  h,h^* }\,
\chi_{\f_\p}^{-1}(\overline{h^*}) \cdot
\chi_{\f_{\overline{\p}}}(h)\cdot
e(-h |h^*) \cdot
\mathcal{L}_{\f,\c}(\lambda_{\infty},\b \mathfrak{h}_p,h,h^*,s  ),
\]
where $(h^*,h)$ is as in  \eqref{condition-h-gro} and \eqref{condition-g-gro}, and
$T:=
\chi (\aa \c \mathfrak{C}^{-1})N(\mathfrak{h}_\p)^{-1}N(\aa \c\mathfrak{C}^{-1}\f_\p)^{-s}.$

\end{lem}

\begin{proof}
It will be enough to show, for the  ideal $\mathfrak{a}$, that
\begin{equation}\label{L-partial-a-bp-cpbar}
[U_{\f_0}:U_{\f}]\cdot
\tau(\chi_{\f_\p}^{-1})\cdot
\frac{\chi(\mathfrak{g}_{p})}
{   N(\mathfrak{g}_{\p})^{1-s}   N(\mathfrak{g}_{\overline{\p}})^s    }
\cdot
L_{\f_0}(\mathfrak{a}\mathfrak{g}_{p}^{-1},\chi,s)
\end{equation}
 is equal to
\begin{equation}\label{other-side-L-partial-l}
T_{0} \cdot
\sum_{ (h^*,h)}\,
 \chi_{\f_\p}^{-1}(\overline{h^*})\cdot
\chi_{\f_{\overline{\p}}}(h)\cdot
e^{2\pi i \mathrm{Tr}_{F/\Q}(-h\,\overline{h^*}\,)}\cdot
\mathcal{L}_{\f}(\lambda_{\infty}, \b \mathfrak{h}_p,h,h^*,s  ),
\end{equation}
where $
T_{0}=
\chi (\aa \mathfrak{C}^{-1})N(\mathfrak{h}_\p )^{-1}N(\aa\mathfrak{C}^{-1}\f_\p)^{-s}
$, $h$ is as in \eqref{condition-h-gro} and
$h^*\in (\b\mathfrak{h}_p)^*\overline{\f}_\p^{-1}/(\b\mathfrak{h}_p)^*$. 
Furthermore, since $\f_\p$ is prime to $\f_0\f_{\overline{\p}}$, we may assume that $h\in (\b \mathfrak{g}_p \f_0^{-1} )/(\b \mathfrak{g}_p \f_{\overline{\p}})$.

Indeed, if this \eqref{L-partial-a-bp-cpbar} is equal to \eqref{other-side-L-partial-l} then, applying it now to  the ideal  $\c\aa$ and then using  the very definition of $L_{\f_0,\c}(\aa\mathfrak{g}_{p}^{-1},\chi,s)$, the result will follow.

According to the very definition of the $L$-function $\mathcal{L}_{\f}(\lambda_{\infty}, \b \mathfrak{h}_p,h,h^*,s  )$, and noting that
 $\chi_{\f_{\overline{\p}}}(y)=\chi_{\f_{\overline{\p}}}(h)$ 
 for $y\in \b \mathfrak{h}_p+h$, we have that \eqref{other-side-L-partial-l} is equal to 
\begin{equation}\label{inter-iden-L-fun-iden}
T_{0}\cdot
\sum_{h}
\sum_{y\in  (\b \mathfrak{h}_p+h)/U_{\f}}
\frac{
       \chi_{\f_{\overline{\p}}} (y) \lambda_{\infty}(y)
          }
           {N((y))^s}
           \sum_{h^*}
           \chi^{-1}_{\f_\p}(\overline{h^*})
           \cdot 
           e\big((y-h)|h^*\big).
\end{equation}

We will show that
\begin{equation}\label{gauss-sum-ident}
 \sum_{h^*}
           \chi^{-1}_{\f_\p}(\overline{h^*})
           \cdot 
            e\big((y-h)|h^*\big)
 =
 \chi_{\f_\p}(\delta_0y)\cdot \tau_{\delta_0^{-1}}(\chi_{\f_\p}^{-1}).         
\end{equation}
Indeed, let $\xi\in F$ such that $(\xi)=\aa^{-1}\mathfrak{g}_\p^{-1}\mathfrak{G}^{-1}$ for an intgeral ideal $\mathfrak{G}$
of $F$ prime to $\f_{\p}$. We will apply \cite[Proposition 4 (5) and (6)]{nemchenok1993imprimitive} to the character $\chi_0:=\chi_{\f_\p}^{-1}$, 
\[
\text{
$\mathfrak{m}_0:=\f_\p$,\quad 
$n_0:=\delta_0^{-1}$, \quad
$\mathfrak{n}_0:=\mathfrak{D}_F\mathfrak{C}$,\quad 
$\mathfrak{a}_0:= \b^{-1}\aa^{-1}\mathfrak{G}^{-1}$ }
\]
and $t_0:=\delta_0(y-h)\xi^{-1}\in \f_0\mathfrak{g}_{\overline{\p}}\f_{\overline{\p}}\mathfrak{G}  \subset \mathcal{O}_F\subset \mathfrak{n}_0^{-1}$ . Observe that $(n_0)=\mathfrak{m}_0^{-1}\mathfrak{D}_F^{-1}\mathfrak{n}_0$, 
$\mathfrak{n}_0$ is an integral ideal prime to $\mathfrak{m}_0$,  $\mathfrak{a}_0$ is a fractional ideal relatively  prime to $\mathfrak{m}_0=\f_\p$ and
 that $\xi \overline{h^*}$ runs through a full set of representatives of $\mathfrak{a}_0/\mathfrak{a}_0\mathfrak{m}_0$ 
 as $h^*$ runs through $(\b\mathfrak{h}_p)^*\overline{\f}_\p^{-1}/(\b\mathfrak{h}_p)^*$. Thus, with the interpretation of $\chi_{\f_p}$ and $\chi_{\f_\p}^{-1}$ defined on all of $x\in F$ such that $x\in \mathcal{O}_{F_{\mathfrak{P}}}$ for all $\mathfrak{P}|\f_\p$, we have
\begin{align*}
 \sum_{h^*}
           \chi^{-1}_{\f_\p}(\overline{h^*})
           \cdot 
           e\big((y-h)|h^*\big)
           &=  \chi_{\f_\p}(\xi)\sum_{h^*}
           \chi^{-1}_{\f_\p}(\xi \overline{h^*})
           \cdot 
           e^{2\pi i \mathrm{Tr}_{F/\Q}\big(n_0t_0\cdot ( \xi \overline{h^*})\big)}\\
           &=
           \chi_{\f_\p}(\xi)\cdot \tau_{n_0t_0}(\chi_{\f_\p}^{-1}) \\ 
            &=
           \chi_{\f_\p}(\xi)\cdot \chi_{\f_\p}(t_0) \cdot\tau_{n_0}(\chi_{\f_\p}^{-1})  \\
           &=\chi_{\f_\p}(\delta_0(y-h)) \cdot\tau_{n_0}(\chi_{\f_\p}^{-1}).     
\end{align*}
Noting that $\delta_0 y\in \mathcal{O}_{F_{\mathfrak{P}}}$  and $\delta_0 h\in \f_\p\mathcal{O}_{F_{\mathfrak{P}}}$ for all $\mathfrak{P}|\f_\p$, then 
$\chi_{\f_\p}(\delta_0(y-h)) =\chi_{\f_\p}(\delta_0y) $ and thus \eqref{gauss-sum-ident} follows.

Bearing in mind that $\chi_{\f_{\overline{\p}}}(\delta_0)=1$, so that 
$\chi_{\f_{\overline{\p}}}(y)=\chi_{\f_{\overline{\p}}}(\delta_0y)$, and observing that
$N(\f_\p\mathfrak{C}^{-1})^s=N((\delta_0))^s$, then from \eqref{chi-on-prinp-ideal} and \eqref{gauss-sum-ident},
we can rewrite \eqref{inter-iden-L-fun-iden} as 
\begin{equation*}\label{one-more-eq}
\frac{\tau(\chi_{\f_\p}^{-1})}{N(\b_\p )}\cdot
\frac{\chi(\mathfrak{a})}{N(\mathfrak{a})^s}\cdot
\sum_{h}
\sum_{y\in  (\b \mathfrak{h}_p+h)/U_{\f}}
\frac{
      \chi_{\f_\p}(\delta_0y ) \chi_{\f_{\overline{\p}}}(\delta_0y )
      \lambda_{\infty}(\delta_0y)
          }
           {N((\delta_0y))^s},
\end{equation*}

Therefore, noticing that $\delta_0\b \mathfrak{h}_p
=\aa^{-1} \mathfrak{g}_p\f_0\f_{\overline{\p}}$ and $\delta_0(\b \mathfrak{h}_p+h)\in \aa^{-1} \mathfrak{g}_p $ and bearing in mind \eqref{chi-on-prinp-ideal}, the above is equal to
\begin{align*}
\frac{\tau(\chi^{-1}_{\f_\p})}{N(\mathfrak{g}_\p )}\cdot
&\frac{\chi(\mathfrak{a})}{N(\mathfrak{a})^s}\cdot
\sum_{h}\mathcal{L}_{\f} (\chi ,\aa^{-1} \mathfrak{g}_p\f_0\f_{\overline{\p}} ,\delta_0h,0,s  ).
\end{align*}
Observe here that $\delta_0 h$ runs through a full set of representatives of the group
$\aa^{-1} \mathfrak{g}_p\f_\p/\aa^{-1} \mathfrak{g}_p\f$ congruent to $1\,\mathrm{mod}\,^*\f_0$. 
Fix a 
$\omega\in \aa^{-1}\mathfrak{g}_p$ 
such that 
$\omega\equiv 1\, \mathrm{mod}^*\,\f_0$, 
then $\iota=\delta_0h-\omega$ runs through a full set of representatives of the group
$ \aa^{-1} \mathfrak{g}_p\f_0
/\aa^{-1} \mathfrak{g}_p\f_0\f_{\overline{\p}}$.

Finally, it remains to observe that
\[
\sum_{\iota}\mathcal{L}_{\f} (\chi ,\aa^{-1} \mathfrak{g}_{p}\f_0\f_{\overline{\p}} ,\iota+\omega,0,s  )
=
\mathcal{L}_{\f} (\chi ,\aa^{-1} \mathfrak{g}_p\f_0 ,\omega,0,s  )
\]
and that
\begin{align*}
\frac{\tau(\chi_{\f_\p}^{-1})}{N(\mathfrak{g}_\p)}\cdot
\frac{\chi(\aa)}{N(\aa)^s}\cdot
&\mathcal{L}_{\f} (\chi,\aa^{-1} \mathfrak{g}_p\f_0 ,\omega,0,s  )\\
&=
\frac{\tau(\chi_{\f_\p}^{-1})}{N(\mathfrak{g}_\p)}\cdot
[U_{\f_0}:U_{\f}]\cdot
\frac{\chi(\aa)}{N(\aa)^s}\cdot
\mathcal{L}_{\f_0} (\chi,\aa^{-1} \mathfrak{g}_p\f_0 ,\omega,0,s  ),
\end{align*}
which coincides with  \eqref{L-partial-a-bp-cpbar}, by \eqref{relation-partial-partial-non-integral-a}.
\end{proof}

\end{proof}

\subsection{Acknowledgments} 

The author would like to thank C. Karabulut and T. Wong for their valuable discussions that helped shape this article. In particular, the author would like to thank T. Wong for his careful review of this paper and for providing helpful comments and suggestions for improvement.

\bibliographystyle{alpha}
\bibliography{Eiscocycle}
\end{document}